\documentclass[10pt,a4paper]{article}
\usepackage{amsmath}
\usepackage{amsfonts}
\usepackage{amssymb}
\usepackage{amsthm}
\usepackage{graphicx}
\usepackage[left=3cm,right=3cm,top=3cm,bottom=3cm]{geometry}
\usepackage{verbatim}
\usepackage{pgf,tikz,pgfplots}
\usetikzlibrary{calc,angles,quotes}
\usetikzlibrary{hobby}
\usetikzlibrary{arrows}
\usepackage{epsfig}
\usepackage{xspace}
\usepackage{hyperref,url}
\hypersetup{
    colorlinks=true,
    linkcolor=blue,
    filecolor=magenta,      
    urlcolor=cyan,
    citecolor=red
}
\usepackage{cite}
\usepackage{xstring}
\usepackage{intcalc}
\pgfplotsset{compat=1.15}
\usepackage{mathrsfs}
\usepackage{enumitem}
\usepackage{caption}

\usepackage[utf8]{inputenc}

\newcommand{\xx}{\tilde{x}}
\newcommand{\yy}{\tilde{y}}
\newcommand{\vv}{\tilde{v}}
\newcommand{\ww}{\tilde{w}}

\newcommand{\ff}{\tilde{f}}
\newcommand{\xxi}{\tilde{\xi}}
\newcommand{\eeta}{\tilde{\eta}}
\newcommand{\mm}{\tilde{m}}

\newcommand{\rgun}{\mathrm r1}

\newcommand{\dd}{d_\Omega}

\newcommand*\diff{\mathop{}\!\mathrm{d}}

\newcommand{\stereo}{s}

\newcommand{\R}{\mathbb{R}}
\newcommand{\K}{\mathbb{K}}

\newcommand{\C}{\mathbb{C}}
\newcommand{\Z}{\mathbb{Z}}
\newcommand{\N}{\mathbb{N}}

\newcommand{\Ecal}{\mathcal{E}}

\newcommand{\Lcal}{\mathcal{L}}
\newcommand{\Ccal}{\mathcal{C}}

\newcommand{\p}{\mathcal{P}}

\newcommand{\CC}{\mathcal{C}}
\newcommand{\GL}{\mathrm{GL}}

\newcommand{\PGL}{\mathrm{PGL}}

\newcommand{\PR}{\mathrm{P}}

\newcommand{\Hopf}{\mathrm{Hopf}}

\DeclareMathOperator{\Isom}{Isom}
\DeclareMathOperator{\Stab}{Stab}

\DeclareMathOperator{\supp}{supp}

\DeclareMathOperator{\interior}{int}

\DeclareMathOperator{\id}{id}

\DeclareMathOperator{\Aut}{Aut}
\DeclareMathOperator{\End}{End}

\DeclareMathOperator{\Vol}{Vol}
\DeclareMathOperator{\topp}{top}
\DeclareMathOperator{\hor}{\mathtt h}
\DeclareMathOperator{\inj}{inj}

\DeclareMathOperator{\Geod}{Geod}
\DeclareMathOperator{\axis}{Axis}
\DeclareMathOperator{\spl}{spl}
\DeclareMathOperator{\prox}{prox}
\DeclareMathOperator{\orb}{orb}
\DeclareMathOperator{\con}{con}
\DeclareMathOperator{\core}{cor}
\DeclareMathOperator{\sse}{sse}
\DeclareMathOperator{\smooth}{smooth}
\DeclareMathOperator{\sing}{sing}
\DeclareMathOperator{\bip}{bip}

\DeclareMathOperator{\prgun}{pr1}
\newcommand{\corfix}{F}

\newcommand{\ie}{i.e.\ }
\newcommand{\eg}{e.g.\ }
\newcommand{\resp}{resp.\ }

\def\tends[#1]{\underset{#1\to\infty}{\rightarrow}}

\newtheorem{prop}{Proposition}[section]
\newtheorem{thm}[prop]{Theorem}

\newtheorem{lemma}[prop]{Lemma}
\newtheorem{fait}[prop]{Fact}
\newtheorem{cor}[prop]{Corollary}
\newtheorem{obs}[prop]{Observation}
\newtheorem{ex}[prop]{Example}

\theoremstyle{definition}
\newtheorem{defi}[prop]{Definition}

\theoremstyle{remark}
\newtheorem{rqq}[prop]{Remark}
\newtheorem{nota}{Notation}

\renewcommand{\b}{\overline}

\numberwithin{equation}{section}

\author{Pierre-Louis Blayac}

\newcommand{\Addresses}{{
  \bigskip
  
  \textsc{Laboratoire Alexander Grothendieck, Institut des Hautes \'Etudes Scientifiques, Universit\'e Paris-Saclay, 35 route de Chartres, 91440 Bures-sur-Yvette, France}\par
  \textit{E-mail address}: \texttt{pierre-louis.blayac@universite-paris-saclay.fr}
}}

\newcounter{mycount}

\newcommand{\length}[1]{
    \setcounter{mycount}{0}
    \foreach \x in {#1}{
        \stepcounter{mycount}
        }
}

\makeatletter
\newcommand{\listnb}[3]{
    \foreach \temp@a [count=\temp@i] in {#1} {
        \IfEq{\temp@i}{#2}{\global\let#3\temp@a\breakforeach}{}
    }
    \par
}
\makeatother

\newcommand{\listnbmod}[2]{
\listnb{#1}{\intcalcInc{\intcalcMod{\intcalcDec{#2}}{\themycount}}}
}

\newcommand{\getanglepoints}[3]{
    \pgfmathanglebetweenpoints{\pgfpointanchor{#2}{center}}
                              {\pgfpointanchor{#3}{center}}
    \global\let#1\pgfmathresult  
}

\newcommand{\getanglelines}[5]{
    \pgfmathanglebetweenlines{\pgfpointanchor{#2}{center}}
                             {\pgfpointanchor{#3}{center}}
                             {\pgfpointanchor{#4}{center}}
                             {\pgfpointanchor{#5}{center}}
    \global\let#1\pgfmathresult  
}

\makeatletter
\newcommand{\getdistance}[3]{
  \pgfpointdiff{\pgfpointanchor{#2}{center}} 
               {\pgfpointanchor{#3}{center}} 
  \pgf@xa=\pgf@x
  \pgf@ya=\pgf@y
  \pgfmathparse{veclen(\pgf@xa,\pgf@ya)/28.45274/2} 
  \global\let#1\pgfmathresult 
}
\makeatother

\makeatletter
\newcommand{\chooseangle}[6]{
\getanglelines{\temp@a}{#1}{#2}{#2}{#3}
\getanglelines{\temp@b}{#2}{#3}{#3}{#4}
\getanglelines{\temp@c}{#3}{#4}{#4}{#5}
\getanglepoints{\temp@d}{#3}{#4}
\pgfmathsetmacro\temp@e{ifthenelse(greater(\temp@a +\temp@c ,0.001),-\temp@c /(max(\temp@a +\temp@c ,0.001))*\temp@b+\temp@d,-\temp@b /2 +\temp@d)}
\global\let#6=\temp@e
}
\makeatother

\makeatletter
\newcommand{\choosecontrolpoints}[7]{
\getanglepoints{\temp@c}{#1}{#2}
\getdistance{\temp@l}{#1}{#2}
\pgfmathsetmacro\temp@a{Mod(\temp@c -#3 ,360)}
\pgfmathsetmacro\temp@b{Mod(#4-\temp@c ,360)}
\pgfmathsetmacro\temp@la{
    ifthenelse(less(\temp@a,90),
        ifthenelse(less(\temp@b,90),abs(sin(\temp@b))*\temp@l,\temp@l),
        ifthenelse(less(\temp@b,90),abs(sin(\temp@b))*\temp@l,\temp@l))}
\pgfmathsetmacro\temp@lb{
    ifthenelse(less(\temp@b,90),
        ifthenelse(less(\temp@a,90),abs(sin(\temp@a))*\temp@l,\temp@l),
        ifthenelse(less(\temp@a,90),abs(sin(\temp@a))*\temp@l,\temp@l))}
        \pgfmathparse{#7*\temp@la}
\global\let#5=\pgfmathresult
        \pgfmathparse{#7*\temp@lb}
\global\let#6=\pgfmathresult
}
\makeatother

\makeatletter
\newcommand{\cvx}[2]{
\foreach \@i in {1,...,\themycount} {

\listnbmod{#1}{\intcalcSub{\@i}{2}}{\@A}
\listnbmod{#1}{\intcalcSub{\@i}{1}}{\@B}
\listnbmod{#1}{\@i}{\@C}
\listnbmod{#1}{\intcalcAdd{\@i}{1}}{\@D}
\listnbmod{#1}{\intcalcAdd{\@i}{2}}{\@E}
\listnbmod{#1}{\intcalcAdd{\@i}{3}}{\@F}

\chooseangle{\@A}
            {\@B}
            {\@C}
            {\@D}
            {\@E}
            {\@a}
\chooseangle{\@B}
            {\@C}
            {\@D}
            {\@E}
            {\@F}
            {\@b}
\choosecontrolpoints{\@C}
                    {\@D}
                    {\@a}
                    {\@b}
                    {\@la}
                    {\@lb}
                    {#2}
                    
\coordinate (G) at ($(\@C)+(\@a:\@la)$);
\coordinate (H) at ($(\@D)+(\@b+180:\@lb)$);

\draw (\@C) .. controls (G) and (H) .. (\@D);
}
}
\makeatother

\makeatletter 
\newcommand{\cvxx}[2]{
\foreach \@i in {1,...,\themycount} {

\listnbmod{#1}{\intcalcSub{\@i}{2}}{\@A}
\listnbmod{#1}{\intcalcSub{\@i}{1}}{\@B}
\listnbmod{#1}{\@i}{\@C}
\listnbmod{#1}{\intcalcAdd{\@i}{1}}{\@D}
\listnbmod{#1}{\intcalcAdd{\@i}{2}}{\@E}
\listnbmod{#1}{\intcalcAdd{\@i}{3}}{\@F}

\chooseangle{\@A}
            {\@B}
            {\@C}
            {\@D}
            {\@E}
            {\@a}
\chooseangle{\@B}
            {\@C}
            {\@D}
            {\@E}
            {\@F}
            {\@b}
\choosecontrolpoints{\@C}
                    {\@D}
                    {\@a}
                    {\@b}
                    {\@la}
                    {\@lb}
                    {#2}
                    
\coordinate (G) at ($(\@C)+(\@a:\@la)$);
\coordinate (H) at ($(\@D)+(\@b+180:\@lb)$);
\coordinate (A\@i) at (intersection of \@C--G and \@D--H);

\draw (\@C) .. controls (G) and (H) .. (\@D);
}
}
\makeatother

\makeatletter 
\newcommand{\draww}[7]{
\coordinate (@A) at ($(#1)!-#3!(#2)$);
\coordinate (@B) at ($(#2)!-#4!(#1)$);
\draw[#5] (@A)--(@B) node[#6]{#7};
}
\makeatother

\title{Patterson--Sullivan densities in convex projective geometry}
\date{}

\begin{document}

\maketitle

\begin{abstract}
 For any rank-one convex projective manifold with a compact convex core, we prove that there exists a unique probability measure of maximal entropy on the set of unit tangent vectors whose geodesic is contained in the convex core, and that it is mixing. We use this to establish asymptotics for the number of closed geodesics. In order to construct the measure of maximal entropy, we develop a theory of Patterson--Sullivan densities for general rank-one convex projective manifolds. In particular, we establish a Hopf--Tsuji--Sullivan--Roblin dichotomy, and prove that, when it is finite, the measure on the unit tangent bundle induced by a Patterson--Sullivan density is mixing under the action of the geodesic flow.
\end{abstract}

\tableofcontents

\pagebreak

\section{Introduction}

Compact real hyperbolic manifolds are fundamental objects in geometry and dynamical systems; their geodesic flows are among the prime examples of chaotic systems which led to standard notions such as ergodicity or entropy. One classical way to deform the geometry of a compact hyperbolic manifold $M_0$ is to deform its Riemannian metric to allow variable negative curvature. It was one breakthrough of the theory of Anosov flows to establish that many dynamical features of the induced geodesic flow remain after such deformations (\eg existence and uniqueness of the measure of maximal entropy, proved independently by Bowen \cite{bowen_BMmeas} and Margulis \cite{margulis_BMmeas}). 

There is another interesting way to deform the geometry of $M_0$. Consider the holonomy representation of the fundamental group $\rho_0:\pi_1(M_0)\rightarrow \mathrm{PO}(d,1)$, where $d$ is the dimension of $M_0$. Under certain conditions, one can deform it continuously (using \eg bending, see \cite{bending}) to get a representation $\rho$ valued in $\PGL_{d+1}(\R)$ which is not conjugate to a representation into $\mathrm{PO}(d,1)$. A theorem of Koszul \cite[Cor.\,p.\,103]{koszul68}, combined with a theorem of Benoist \cite[Th.\,1.1]{CD3} (due to Choi--Goldman \cite{ChoiGold93} for $d=3$), ensures that the representation remains faithful and discrete, and that $\rho(\pi_1(M_0))$ preserves and acts cocompactly on a properly convex open subset $\Omega$ of the real projective space $\PR(V)$, where $V=\R^{d+1}$. The quotient $\Omega/\rho(\pi_1(M_0))$ is a compact \emph{convex projective manifold} which, like Riemannian manifolds, admits a geodesic flow, as we now recall.

In general, a convex projective manifold is a quotient $M=\Omega/\Gamma$ of a properly convex open set $\Omega\subset \PR(V)$ by a discrete group $\Gamma\subset\PGL(V)$ of projective transformations preserving $\Omega$. If $M$ is compact, then we say that $\Gamma$ divides $\Omega$ and that $\Omega$ is a divisible convex set. The set $\Omega$ admits a Finsler metric, called the Hilbert metric, which is proper and $\Gamma$-invariant (hence $\Gamma$ acts properly discontinuously on $\Omega$). Moreover, the intersection of any projective line with $\Omega$ is a geodesic for the Hilbert metric, which we call a \emph{straight line} of $\Omega$. Thus, there is a natural geodesic flow $(\phi_t)_{t\in\R}$ on the unit tangent bundle $T^1M:=T^1\Omega/\Gamma$, which parametrises straight lines. Benoist \cite{CD1} initiated the study of $(\phi_t)_{t\in\R}$ in the divisible case. He proved that if $M$ is obtained by deforming a compact hyperbolic manifold, then $\Omega$ is \emph{strictly convex}, \ie its boundary $\partial\Omega\subset\PR(V)$ does not contain any non-trivial projective segment. He also proved that if $M$ is compact, then $\Omega$ is strictly convex if and only if the geodesic flow $(\phi_t)_{t\in\R}$ is Anosov.

In this paper, we are particularly interested in divisible convex sets that are \emph{not} strictly convex. Classical examples are the \emph{higher-rank symmetric} divisible convex sets, namely the projective models of the symmetric spaces of $\PGL_n(\K)$, where $n\geq 3$ and $\K$ is $\R$ or $\C$ or the classical quaternionic (or octonionic for $n=3$) division algebra (for more details see \cite[\S2.4]{benoist_survey}). Other interesting,  irreducible examples were constructed by Benoist \cite{CD4}, followed by Marquis \cite{LudoThese}, Ballas--Danciger--Lee \cite{BDL_cvxproj_3mfd} and Choi--Lee--Marquis \cite{choi2016convex}, in dimensions $3$ to $7$.

Bray \cite{bray_ergodicity,bray_top_mixing} studied the geodesic flow of $3$-dimensional irreducible compact convex projective manifolds $M=\Omega/\Gamma$ for which $\Omega$ is not necessarily strictly convex. Although the theory of Anosov flows does not apply in this setting, he managed to construct, on the unit tangent bundle, a flow-invariant ergodic measure with maximal entropy \cite[Th.\,1.1]{bray_ergodicity}. For this he used Benoist's precise and beautiful geometric description of such $3$-manifolds \cite[Th.\,1.1]{CD4}. He also adopted methods that had been elaborated for non-positively curved Riemannian manifolds, whose geodesic flows are not Anosov in general. More precisely, he drew on the work of Knieper \cite{knieper_BMmeasure} and Roblin \cite{roblin_smf}, whose main tool are \emph{Patterson--Sullivan densities}.

In the present article we study the dynamics of the geodesic flow of convex projective manifolds that have arbitrary dimension and are not necessarily compact. Like Bray, we use methods from the non-positively curved Riemannian world, in particular inspired by Knieper and Roblin. We generalise and improve Bray's results \cite{bray_ergodicity}, and develop more systematically the theory of Patterson--Sullivan densities in this setting.

\subsection{Rank-one convex projective manifolds}

Knieper~\cite{knieper_BMmeasure} considered non-positively curved compact Riemannian manifolds which satisfy a property called \emph{rank-one}. These generalise negatively curved compact manifolds, whose geodesic flow is uniformly hyperbolic, in only requiring that the geodesic flow have a hyperbolic behaviour along at least one geodesic, which is said to be \emph{rank-one} (see \cite[Def.\,5.1.1]{knieper_handbook}). Similarly, we will consider \emph{rank-one} convex projective manifolds, which generalise convex projective manifolds $M=\Omega/\Gamma$ where $\Omega$ is strictly convex and $\partial\Omega$ is smooth (by which we mean $\mathcal{C}^1$, see Section~\ref{duality}). We will use the following definition, recently introduced by M.\,Islam \cite{islam_rank_one}.

We say that a point of the boundary $\partial\Omega$ is \emph{strongly extremal} if it does not belong to any non-trivial segment contained in $\partial\Omega$. If $\Omega$ is strictly convex, then all points of $\partial\Omega$ are strongly extremal. We say that a point is \emph{smooth} if it admits a unique supporting hyperplane; such points are also commonly called $\mathcal{C}^1$. We denote by $\partial_{\sse}\Omega$ the set of smooth, strongly extremal points of $\partial\Omega$.  Given any vector $v\in T^1\Omega$, we denote by $\pi v\in\Omega$ its footpoint and by $\phi_{\pm\infty}v=\lim_{t\to\pm\infty}\pi\phi_{t}v$ the intersection points of the projective line generated by $v$ with $\partial\Omega$. We denote by $\Aut(\Omega)$ the subgroup of $\PGL(V)$ consisting of elements that preserve~$\Omega$, called automorphisms of $\Omega$.

\begin{defi}[\!{\!\cite[Def.\,6.2 \& Prop.\,6.3]{islam_rank_one}}]\label{def rank-one}
Let $\Omega\subset\PR(V)$ be a properly convex open set.  A vector $v\in T^1\Omega$, and the geodesic of $\Omega$ spanned by $v$, are called \emph{rank-one} if $\phi_{\infty}v$ and $\phi_{-\infty}v$ belong to $\partial_{\sse}\Omega$. An infinite-order automorphism of $\Omega$ is said to be \emph{rank-one} if it preserves a rank-one geodesic of $\Omega$.

Let $\Gamma\subset\Aut(\Omega)$ be a discrete subgroup, and $M:=\Omega/\Gamma$. A vector $v$ of $T^1M$, and the geodesic of $M$ spanned by $v$, are said to be \emph{rank-one} if any lift of $v$ to $T^1\Omega$ is rank-one. The convex projective manifold (or orbifold) $M$ is \emph{rank-one} if it contains a rank-one periodic vector, \ie if $\Gamma$ contains a rank-one element.
\end{defi}
We will see (Fact~\ref{period_is_translation_length}) that a rank-one automorphism $g$ of $\Omega$ is a \emph{proximal} element of $\PGL(V)$, \ie it has an attracting fixed point in $\PR(V)$; its inverse being also rank-one and hence proximal, $g$ is said to be \emph{biproximal}. We will also see that a rank-one automorphism preserves a unique rank-one geodesic in $\Omega$.

If $\Omega$ is strictly convex and $\partial\Omega$ is smooth, then $\partial_{\sse}\Omega$ is the whole projective boundary $\partial\Omega$, all geodesics of $\Omega$ are rank-one and any biproximal automorphism of $\Omega$ is rank-one. In fact, if $\Omega$ is strictly convex \emph{or} if $\partial\Omega$ is smooth, then $M=\Omega/\Gamma$ is rank-one as soon as $\Gamma$ contains a biproximal element, by Facts~\ref{les extremites des geodesiques periodiques biproximales sont lisses} and \ref{equivalences rang un} below.

When $\Omega$ is not strictly convex, an example of a geodesic in $\Omega$ that is \emph{not} rank-one is a geodesic that is contained in a \emph{properly embedded simplex} (PES), \ie a projective simplex $S$ of dimension $k\geq 2$ whose relative interior (see Section~\ref{duality}) is equal to $S\cap \Omega$. Such a simplex $S$ can be interpreted as a flat of $\Omega$ since it is isometric to $\R^k$ endowed with some norm. In many examples of convex projective manifolds $M=\Omega/\Gamma$, for instance when $M$ is $3$-dimensional, compact  and irreducible, the maximal (for inclusion) PES's of $\Omega$ satisfy good properties (such as being isolated, see \cite{CD4,Bobb,IZ19relhypb}) which imply that $M$ is rank-one.
More precisely, Islam used \cite{IZ19relhypb} to establish \cite[Prop.\,A.2]{islam_rank_one} that if $\Gamma\subset\Aut(\Omega)$ is a non-virtually abelian, discrete subgroup which is relatively hyperbolic with respect to a collection of virtually abelian subgroups of rank at least two, and which acts convex cocompactly on $\Omega$ in the sense of Danciger--Gu\'eritaud--Kassel \cite{fannycvxcocpct} (this notion is introduced in the next Section~\ref{Ssection : intro MME}), then $\Omega/\Gamma$ is rank-one. Islam's argument is explained in the particular case of 3-dimensional compact convex projective manifolds in Example~\ref{Exemple : dim3->rg1}.

Rank-one manifolds are interesting because they include a diversity of examples, and are in some sense generic: in both the Riemannian and the convex projective settings, there exist higher-rank rigidity theorems which classify compact higher-rank (\ie not rank-one) manifolds. See the work of Ballmann \cite[Cor.\,1]{ballmann_higher_rank} and Burns--Spatzier \cite[Th.\,5.1]{burns_spatzier_higher_rank} in the Riemannian case and the recent work of A.\,Zimmer \cite[Th.\,1.4]{zimmer_higher_rank} in the convex projective case.

\subsection{The Bowen--Margulis measure on quotients of convex cocompact actions}\label{Ssection : intro MME}

The generalisation of Bray's results \cite{bray_ergodicity} that we are about to state is analogous to \cite[Th.\,1.1.i]{knieper_BMmeasure}, which says that any non-positively curved rank-one compact Riemannian manifold has a unique measure of maximal entropy. However, our result does not restrict to compact manifolds: it concerns the more general class of manifolds that are quotients of \emph{convex cocompact actions}. 

Let $M=\Omega/\Gamma$ be a convex projective manifold. Following Danciger--Gu\'eritaud--Kassel, the action of $\Gamma$ on $\Omega$ is said to be \emph{naively convex cocompact} if there exists a non-empty $\Gamma$-invariant convex subset of $\Omega$ on which $\Gamma$ acts cocompactly; when $\Omega$ is not strictly convex, this notion is not quite satisfactory, because a small deformation of $\Gamma$ in $\PGL(V)$ may not preserve any properly convex open set (see \cite[\S4.1]{fannycvxcocpct}). Therefore, we will consider another stronger definition, introduced by Danciger--Gu\'eritaud--Kassel.

The \emph{full orbital limit set} of $\Gamma$ is the union over all $x\in\Omega$ of the set of accumulation points of the orbit $\Gamma\cdot x$, and is denoted by $\Lambda^{\orb}_\Omega(\Gamma)$, or simply $\Lambda^{\orb}$ when the context is clear. When $\Omega$ is not strictly convex, the set of accumulation of points of an orbit $\Gamma\cdot x$ may depend on the choice of $x\in\Omega$ (\eg if $\Omega$ is a triangle in $\PR(\R^3)$, and $\Gamma$ is generated by an infinite-order non-proximal element). The convex hull in $\Omega$ of the full orbital limit set is denoted by $\CC^{\core}_\Omega(\Gamma)$; it is $\Gamma$-invariant.

\begin{defi} [\!{\!\cite[Def.\,1.11]{fannycvxcocpct}}]\label{Def : convexe cocompacite} Let $\Omega\subset\PR(V)$ be a properly convex open set, and $\Gamma\subset\Aut(\Omega)$ a discrete subgroup. The action of $\Gamma$ on $\Omega$ is said to be \emph{convex cocompact} if $\CC^{\core}_\Omega(\Gamma)$ is non-empty and has compact quotient by $\Gamma$; the projection in $M$ of $\CC^{\core}_\Omega(\Gamma)$ is called the \emph{convex core} of $M$.
\end{defi}

We refer to \cite[\S1.4--1.6--1.7--4.1--10.7]{fannycvxcocpct} and \cite{DGKLM} for more details and examples on convex cocompactness. Note that if $\Gamma$ divides a properly convex open set $\Omega$, then the convex hull of any $\Gamma$-orbit in $\Omega$ is equal to $\Omega$ (this is due to Vey \cite[Prop.\,3]{vey}), hence $\CC^{\core}_\Omega(\Gamma)=\Omega$ and $\Gamma$ acts convex cocompactly on $\Omega$.

Danciger--Gu\'eritaud--Kassel \cite[Cor.\,4.8]{fannycvxcocpct} proved that if the action of $\Gamma$ on $\Omega$ is convex cocompact, then $\Lambda^{\orb}$ is closed. We denote by $T^1M_{\core}\subset T^1M$ the $(\phi_t)_{t\in\R}$-invariant (and compact if $\Gamma$ is convex cocompact) subset consisting of those vectors whose orbit under the geodesic flow is contained in the convex core, \ie the endpoints of any lift to $\Omega$ of the geodesic are in $\Lambda^{\orb}$. 

A rank-one convex projective manifold is said to be \emph{non-elementary} if its fundamental group does not contain $\Z$ as a finite-index subgroup.

\begin{thm}\label{Thm : proba d'entropie max dans le cas compact}
 Let $\Omega\subset\PR(V)$ be a properly convex open set, and $\Gamma\subset\Aut(\Omega)$ a discrete subgroup. Assume that $\Gamma$ acts convex cocompactly on $\Omega$ and that $M=\Omega/\Gamma$ is rank-one and non-elementary. Then there exists a unique $(\phi_t)_{t\in\R}$-invariant probability measure $m$ on $T^1M_{\core}$ with maximal entropy (\emph{Bowen--Margulis measure}). Moreover $m$ is mixing.
\end{thm}

Recall that if $(X,m)$ is a measured space with $m$ finite and invariant under a measurable flow $(\phi_t)_{t\in\R}$, then $m$ is said to be \emph{mixing} if $\lim_{t\to\infty}m(\phi_t(A)\cap B)\,m(X)=m(A)\,m(B)$ for all measurable subsets $A,B\subset X$. Reminders on the notion of entropy are given in Sections~\ref{rappels sur l'entropie topologique} and \ref{rappels sur l'entropie mesurable}.

\begin{rqq}
 In a first version of this paper, we had proved Theorem~\ref{Thm : proba d'entropie max dans le cas compact} under the stronger assumption that Gamma acts strongly irreducibly on $\PR(V)$. Replacing this assumption by non-elementary was made possible by Facts~\ref{fait:des rk1 partout} and \ref{non arithmeticite locale} below, which are proved in the paper \cite{EeERfH+} in collaboration with Feng Zhu.
 \ Note that some results in the present paper (in particular Theorem~\ref{Thm : The weak Hopf--Tsuji--Sullivan--Roblin dichotomy}) are used in \cite{EeERfH+}, but one can easily check that there is no circular reasoning between the two papers.
\end{rqq}

Following Knieper \cite{knieper_BMmeasure}, we will use Theorem~\ref{Thm : proba d'entropie max dans le cas compact} to establish asymptotic estimates on the number of closed geodesics on $M$, and equidistribution results on the Lebesgue measures on closed geodesics.

Recall that in our convex projective setting, the \emph{critical exponent} of $\Gamma$ is defined as 
\begin{equation*}
\delta_\Gamma := \limsup_{r\to\infty}\frac{1}{r}\log\left(\#\{\gamma\in\Gamma : d_\Omega(o,\gamma o)\leq r\}\right),
\end{equation*}
where $o$ is any point of $\Omega$, and $d_\Omega$ is the Hilbert metric on $\Omega$; it does not depend on the $\Gamma$-invariant properly convex open set $\Omega$, see Section~\ref{Section : Exposant critique}. For any element $g\in\PGL(V)$, we set 
\begin{equation}\label{ell}
\ell(g):=\frac{1}{2}\log\left(\frac{\lambda_1(\tilde{g})}{\lambda_{d+1}(\tilde{g})}\right),
\end{equation}
where $\tilde{g}\in\GL(V)$ is any lift of $g$, the integer $d+1$ is the dimension of $V$, and $\lambda_1(\tilde{g})\geq \dots \geq \lambda_{d+1}(\tilde{g})$ are the moduli of the (complex) eigenvalues of $\tilde{g}$. Given a convex projective manifold $M=\Omega/\Gamma$ and a positive number $T>0$, we denote by $[\Gamma]_T$ (\resp $[\Gamma]^{\sing}_T$, \resp $[\Gamma]^{\rgun}_T$) the set of conjugacy classes of elements (\resp of non-rank-one elements, \resp of rank-one elements) $\gamma\in\Gamma$ such that $\ell(\gamma)\leq T$. When $\gamma\in\Gamma$ is rank-one, we denote by $\Lcal [\gamma]$ the Lebesgue measure on the rank-one closed geodesic associated to $\gamma$, normalised to be a probability measure. The link between closed geodesics on $M$ and conjugacy classes of $\Gamma$ is recalled in Section~\ref{Periodic orbits and elements of Gamma}.

\begin{prop}\label{Prop : comptage dans le cas compact}
 In the setting of Theorem~\ref{Thm : proba d'entropie max dans le cas compact}, one can find constants $0<\delta<\delta_\Gamma$ and $C>0$ such that for any $T>C$,
 \begin{enumerate}[label=(\arabic*)]
  \item \label{Item : comptage 1} $\frac{1}{CT}e^{\delta_\Gamma T}\leq \#[\Gamma]_T\leq \frac{C}{T}e^{\delta_\Gamma T}$;
  \item \label{Item : comptage 2} $\#[\Gamma]^{\sing}_T\leq e^{\delta T}$;
  \item \label{Item : comptage 3} $\dfrac{1}{\#[\Gamma]^{\rgun}_T}\displaystyle\sum\limits_{[\gamma]\in[\Gamma]^{\rgun}_T}\Lcal[\gamma] \underset{T\to\infty}{\longrightarrow} m$ for the weak* topology.
 \end{enumerate}
\end{prop}
In the case that $M$ is compact, Proposition~\ref{Prop : comptage dans le cas compact}.\ref{Item : comptage 1} was previously established by Islam \cite[Th.\,1.12]{islam_rank_one} using different techniques.
One may compare Proposition~\ref{Prop : comptage dans le cas compact}.\ref{Item : comptage 2} with another result of Islam \cite[Th.\,1.11]{islam_rank_one} concerning random walks on $\Gamma$.
We prove a slightly more general version of Proposition~\ref{Prop : comptage dans le cas compact}.\ref{Item : comptage 3} in Proposition~\ref{Proposition : equidistribution} below.

\subsection{The machinery of Patterson--Sullivan densities}\label{Section : la machinerie des densites conformes}

The main idea to prove Theorem~\ref{Thm : proba d'entropie max dans le cas compact} is to use \emph{conformal densities} (Definition~\ref{horobord}), also called \emph{Patterson--Sullivan densities}. They are, for an arbitrary proper metric space $(X,d)$ acted on by a discrete subgroup $\Gamma\subset\Isom(X,d)$, families of finite measures (Patterson-Sullivan measures) on the \emph{horoboundary} $\partial_{\hor}X$ of $(X,d)$. More precisely, Patterson--Sullivan measures are quasi-invariant under the action of $\Gamma$, and there is an explicit formula for the Radon--Nikodym cocycle (see Definition~\ref{densites conformes}). They were originally introduced by Patterson \cite{Patterson76} and Sullivan \cite{Sullivan79} for discrete groups of isometries of real hyperbolic spaces, whose horoboundary is the boundary at infinity.  They were later used in much more general geometric settings, such as non-positively curved Riemannian manifolds by Knieper \cite{knieper_BMmeasure}, CAT($-1$)-spaces by Roblin \cite{roblin_smf}, and even CAT(0)-spaces by Picaud--Link \cite{link_ergo_rank_one} and Ricks \cite{ricks2019unique}. In all these settings, the horoboundary is equal to the visual boundary.

Conformal densities were also brought to convex projective geometry, in the strictly convex and smooth case by Crampon in his PhD thesis \cite{these_crampon} and recently by F.\,Zhu \cite{zhu_conf_densities}, and in the non-strictly convex case in dimension $3$ by Bray \cite{bray_ergodicity}.  While Zhu does not adapt Knieper's work (hence does not prove the existence and uniqueness of the measure of maximal entropy), his adaptation of Roblin's results goes further than in the present paper, enabling him for instance to obtain more precise estimates for various counting problems. In the paper \cite{EeERfH+}, the author and Zhu generalise the results of \cite{zhu_conf_densities} to the non-strictly convex case.

Let $M=\Omega/\Gamma$ be a convex projective manifold. If $\partial\Omega$ is smooth, then the horoboundary $\partial_{\hor}\Omega$ is the projective boundary $\partial\Omega$ (see Fact~\ref{smooth_points_are_in_hor}). When $\partial\Omega$ is not smooth, the situation is more delicate since $\partial_{\hor}\Omega$ is different from $\partial\Omega$. To handle this difficulty, Bray's strategy was to weaken the definition of conformal densities so that it has a meaning on $\partial\Omega$. We will do something slightly different: we will use a result of Walsh \cite[Th.\,1.3]{horoboundary}, who proved in general that the horoboundary $\partial_{\hor}\Omega$ \emph{dominates} $\partial\Omega$, in the sense that the identity on $\Omega$ extends to a continuous map $\partial_{\hor}\Omega\to\partial\Omega$, which is onto by density of $\Omega$. We will define conformal densities on $\partial\Omega$ simply as the push-forwards of conformal densities on $\partial_{\hor}\Omega$ by the natural projection $\partial_{\hor}\Omega\to\partial\Omega$. 

Let us detail the steps that we will follow below, adapted from the general theory of conformal densities. Recall that conformal densities depend on a parameter $\delta\geq 0$. Fix $o\in\Omega$.

\begin{enumerate}
 \item \label{Item : construction de PS} Construct a $\delta_\Gamma$-conformal density (this is actually very general, see Fact~\ref{smear}).
 \item \label{Item : construction de BM} Given $\delta\geq 0$ and a $\delta$-conformal density $(\mu_x)_{x\in\Omega}$ on $\partial_{\hor}\Omega$, construct a measure on $\partial_{\hor}\Omega^2\times\R$ which is invariant under the actions of $\Gamma$ and $\R$, and equivalent to $\mu_o\times\mu_o$ times the Lebesgue measure. Derive from it a $(\phi_t)_{t\in\R}$-invariant measure on $T^1M$, called the induced \emph{Sullivan measure}.
 \item \label{Item : HTSR} Given $\delta\geq 0$ and a $\delta$-conformal density $(\mu_x)_{x\in\Omega}$ on $\partial_{\hor}\Omega$, prove a \emph{Hopf--Tsuji--Sullivan--Roblin (HTSR) dichotomy}. It states in particular that the sum $\sum_{\gamma\in\Gamma} e^{-\delta d_\Omega(o,\gamma o)}$ is infinite if and only if the induced Sullivan measure is ergodic under the geodesic flow; in this case $\delta=\delta_\Gamma$, the $\delta_\Gamma$-conformal density is unique and its induced Sullivan measure is called the \emph{Bowen--Margulis measure}. Recall that ergodic means that any $(\phi_t)_{t\in\R}$-invariant measurable set has null or full measure. (A set has full measure if its complement has null measure.)
 \item \label{Item : melange} Assuming that $\sum_{\gamma\in\Gamma} e^{-\delta_\Gamma d_\Omega(o,\gamma o)}$ is infinite, prove that if the Bowen--Margulis measure is finite, then it is mixing under the geodesic flow.
 \item \label{Item : melange -> comptage} Use mixing to solve counting problems. There are many counting problems, such as estimating, as $R$ tends to infinity, the number of points in a fixed orbit $\Gamma\cdot o$ that lie in the ball  of radius $R$ centred at $o$, or the number of closed geodesics with length less than $R$. 
 
 As we are particularly interested in the compact case, we decided to follow, instead of Roblin's, Knieper's approach. It gives estimates for the number of closed geodesics with length less than $R$ as $R$ grows, but also establishes that the Bowen--Margulis measure is the unique measure with maximal entropy. 
\end{enumerate}
Recall that if $\sum_{\gamma\in\Gamma} e^{-\delta_\Gamma d_\Omega(o,\gamma o)}=\infty$, then $\Gamma$ is said to be \emph{divergent}; this notion does not depend on the $\Gamma$-invariant properly convex open set $\Omega\subset\PR(V)$, see Section~\ref{Section : Exposant critique}.

One strength of the theory of conformal densities in the various geometric settings we have mentioned is that many properties of compact manifolds still hold under the following much weaker assumptions: asking $\Gamma$ to be divergent and the Bowen--Margulis measure to be finite. For example, these assumptions hold if $\Gamma$ acts convex cocompactly on $\Omega$ and $M=\Omega/\Gamma$ is rank-one (see Proposition~\ref{Proposition : cvxcocpct -> divergent}).

Another broad class of manifolds which are generally well understood are the geometrically finite manifolds. In some settings, such as Riemannian geometry with variable negative curvature, there is no implication between geometric finiteness and finiteness of the Bowen--Margulis measure (see \cite[\S1.F]{roblin_smf}). Crampon--Marquis \cite[Def.\,1.4]{CM2014finitude} defined geometrically finite convex projective manifolds $M=\Omega/\Gamma$ in the case where $\Omega$ is strictly convex and $\partial\Omega$ is smooth. They carefully investigated the dynamics of the geodesic flow on these manifolds, without using conformal densities. In particular they extended Benoist's result that the flow is uniformly hyperbolic \cite[Th.\,1.2]{CM2014flot} (under an additional assumption of asymptotically hyperbolic cusps). Zhu and the author \cite[Th.\,C]{EeERfH+} proved (generalising earlier results by Zhu \cite[Th.\,10 \& Prop.\,14]{zhu_conf_densities}) that for these manifolds $\Gamma$ is divergent and the Bowen--Margulis measure is finite.

\subsection{The Hopf--Tsuji--Sullivan--Roblin dichotomy}\label{Ssection : HTSR}

Let us explain more precisely Steps \ref{Item : construction de PS}, \ref{Item : construction de BM} and \ref{Item : HTSR} of the previous section. Let $M=\Omega/\Gamma$ be a non-elementary rank-one convex projective manifold. On the one hand, Steps~\ref{Item : construction de PS} and \ref{Item : construction de BM}, \ie constructing a $\delta_\Gamma$-conformal density on $\partial_{\hor}\Omega$ and the induced Sullivan measure on $T^1M$, do not require major changes compared to the case where $M$ is a real hyperbolic manifold; these steps are done in Sections~\ref{rappel sur les densites conformes} and \ref{The Sullivan measures}. On the other hand, Step~\ref{Item : HTSR}, namely the HTSR dichotomy, is more delicate; let us state it formally. For this, we need to recall the definition of two limit sets, both contained in the full orbital limit set.

The \emph{conical limit set} of the action of $\Gamma$ on $\Omega$ is denoted by $\Lambda^{\con}_\Omega(\Gamma)$ (or simply $\Lambda^{\con}$ when the context is clear), and defined as follows. Fix $o\in\Omega$, then $\xi\in\partial\Omega$ belongs to $\Lambda^{\con}$ if there exists a sequence $(\gamma_n)_{n\in\N}\in\Gamma^\N$ going to infinity such that the sequence $(d_\Omega(\gamma_no,[o,\xi)))_{n\in\N}$ is bounded; this does not depend on the choice of $o$. The \emph{proximal limit set} of $\Gamma$, denoted by $\Lambda^{\prox}(\Gamma)$, or simply $\Lambda^{\prox}$ when the context is clear, is the closure of the set of attracting fixed points of proximal elements of $\Gamma$. 

In general, $\Lambda^{\con}\subset\Lambda^{\orb}$ and $\Lambda^{\prox}\subset\overline{\Lambda^{\orb}}$. Furthermore, Danciger--Gu\'eritaud--Kassel \cite[Cor.\,4.8 \& Lem.\,4.18]{fannycvxcocpct} established that $\Gamma$ acts convex cocompactly on $\Omega$ if and only if $\CC^{\core}_\Omega(\Gamma)$ is non-empty and $\Lambda^{\con}_\Omega(\Gamma)$ and $\Lambda^{\orb}_\Omega(\Gamma)$ are equal and closed. 

In \cite[Def.\,1.1]{topmixing} we introduced a $(\phi_t)_{t\in\R}$-invariant closed subset of $T^1M$, called the \emph{biproximal unit tangent bundle} and denoted by $T^1M_{\bip}$; it consists of those vectors $v\in T^1M$ such that $\phi_{\pm\infty}\vv\in\Lambda^{\prox}$ for any lift $\vv\in T^1\Omega$. When $M$ is rank-one and $\Gamma$ is divergent, the following result gives a new interpretation of $T^1M_{\bip}$, as the support of the Bowen--Margulis measure on $T^1M$.

\begin{thm}\label{Thm : The weak Hopf--Tsuji--Sullivan--Roblin dichotomy}
 Let $o\in\Omega\subset\PR(V)$ be a pointed properly convex open set, and $\Gamma\subset\Aut(\Omega)$ a discrete subgroup with $M=\Omega/\Gamma$ rank-one and non-elementary. Let $\delta\geq 0$, let $(\nu_x)_{x\in\Omega}$ be a $\delta$-conformal density on $\partial\Omega$ and let $m$ be the induced Sullivan measure. Then there are two possibilities:
\begin{enumerate}[label=(\arabic*)]
\item \label{Item : cas convergent} either $\sum_\gamma e^{-\delta d_\Omega(o,\gamma o)}<\infty$, and then $\nu_o(\Lambda^{\con})=0$ and the dynamical systems $(T^1M,\phi_t,m)$, $(\Geod^{\infty}(\Omega),\Gamma,\nu_o^2)$ and $(T^1\Omega,\Gamma\times\R,\tilde{m})$ are dissipative and non-ergodic;
\item \label{Item : cas divergent} or $\sum_\gamma e^{-\delta d_\Omega(o,\gamma o)}=\infty$, in which case $\delta=\delta_\Gamma$ (and $\Gamma$ is divergent), and 
\begin{itemize}[leftmargin=.3cm]
 \item $(\nu_x)_{x\in\Omega}$ is the only $\delta_\Gamma$-conformal density (up to a scalar multiple), and $m$ is called the \emph{Bowen--Margulis measure} on $T^1M$; 
 \item $\nu_o(\partial_{\sse}\Omega\cap\Lambda^{\prox}\cap\Lambda^{\con})=\nu_o(\partial\Omega)$ and $\nu_o$ is non-atomic, in particular $\supp(m)=T^1M_{\bip}$;
 \item $(T^1M,\phi_t,m)$, $(\partial\Omega^2,\Gamma,\nu_o^2)$ and $(T^1\Omega,\Gamma\times\R,\tilde{m})$ are conservative and ergodic;
 \item if $m$ is finite then it is mixing.
\end{itemize}
\end{enumerate}
\end{thm}

In \ref{Item : cas convergent} we denote by $\Geod^{\infty}(\Omega)$ the space of straight geodesics of $\Omega$, which consists of pairs $(\xi,\eta)\in\partial\Omega^2$ such that $[\xi,\eta]\cap\Omega$ is non-empty. Reminders on the dynamical notions of dissipativity and conservativity are given in Section~\ref{Section : decompo de Hopf}. 

To establish the mixing property in Theorem~\ref{Thm : The weak Hopf--Tsuji--Sullivan--Roblin dichotomy}.\ref{Item : cas divergent} when $m$ is finite, we use Babillot's strategy of proof \cite{Babillot02}, and also some general results of Coud\`ene \cite{coudene_melange} in ergodic theory that are inspired by \cite{Babillot02}. In particular, cross-ratios of quadruples of points on the boundary of a properly convex open set are a crucial component of the proof of the mixing property. Zhu proved the mixing property \cite[Th.\,18]{zhu_conf_densities} in the case where $\Omega$ is strictly convex with $\mathcal{C}^1$ boundary by using the same strategy. 

As we explained in Section~\ref{Section : la machinerie des densites conformes}, in order to prove Theorem~\ref{Thm : proba d'entropie max dans le cas compact} we will use Theorem~\ref{Thm : The weak Hopf--Tsuji--Sullivan--Roblin dichotomy}.\ref{Item : cas divergent}; hence we need to check that the latter can be applied to the case where $\Gamma$ acts convex cocompactly on $\Omega$ and $M=\Omega/\Gamma$ is rank-one. This is the goal of the next proposition.

\begin{prop}\label{Proposition : cvxcocpct -> divergent}
 Let $\Omega\subset\PR(V)$ be a properly convex open set and $\Gamma\subset\Aut(\Omega)$ a convex cocompact discrete subgroup. Suppose $M=\Omega/\Gamma$ is rank-one and non-elementary. Then $\Gamma$ is divergent, and the Bowen--Margulis measure is finite.
\end{prop}

In another article \cite{proxestbord}, we prove that for any non-elementary rank-one compact convex projective manifold $M=\Omega/\Gamma$, we have $\Lambda^{\prox}=\partial\Omega$. This implies that $T^1M_{\bip}=T^1M$, and from Theorem~\ref{Thm : The weak Hopf--Tsuji--Sullivan--Roblin dichotomy} and Proposition~\ref{Proposition : cvxcocpct -> divergent} we derive the following.

\begin{cor}\label{cor:supp}
 The measure of maximal entropy of any non-elementary rank-one compact convex projective manifold has full support.
\end{cor}

Benoist \cite[Prop.\,6.7]{CD1} proved that, in the setting of Corollary~\ref{cor:supp}, if $M$ is not hyperbolic, then the measure of maximal entropy is singular with respect to the Lebesgue measure.

%
%
%
%

\paragraph{Organisation of the paper}

In Section~\ref{Reminders} we collect definitions and basic properties in convex projective geometry, on conformal densities and on entropy. 

Sections~\ref{Construction of the Sullivan measures} to \ref{Section : HTSR divergent} are based on Roblin's work \cite{roblin_smf}. In Section~\ref{Construction of the Sullivan measures} we define the Hopf coordinates and the Gromov product in the setting of convex projective geometry, in order to make sense of Sullivan's formula \cite[Prop.\,11]{Sullivan79}, which defines Sullivan measures. In Section~\ref{The Shadow lemma} we state and prove a convex projective version of the Shadow lemma, a fundamental result in the study of conformal densities. In Section~\ref{Section : HTSR convergent} we establish the convergent case of the HTSR dichotomy. In Section~\ref{Section : HTSR divergent} we assume that $\Gamma$ is divergent, and follow closely Roblin's proof of HTSR dichotomy in order to establish the convex projective version of Theorem~\ref{Thm : The weak Hopf--Tsuji--Sullivan--Roblin dichotomy}.\ref{Item : cas divergent}. The proof is divided into several steps: proving that the conical limit set has full measure (Section~\ref{ssection:bcp de coniques});
proving that $\partial_{\sse}\Omega$ has full measure (Section~\ref{Ssection : negligeons les singularites}); proving that the Bowen--Margulis measure is ergodic, and moreover mixing when finite (Section~\ref{ssection:ergo+melange}); and finally proving that $\Lambda^{\prox}$ has full measure (Section~\ref{Section : supp(conf)}).

Sections~\ref{Preparatifs pour la partie Knieper} to \ref{Counting closed geodesics} are based on Knieper's work \cite{knieper_BMmeasure}. In Section~\ref{Preparatifs pour la partie Knieper} we collect some properties of convex cocompact projective actions; in particular we prove Proposition~\ref{Proposition : cvxcocpct -> divergent}, and give two refined versions of the Shadow lemma (Lemma~\ref{shadowlemma}) which will be used in Section~\ref{The measure of maximal entropy}. In Section~\ref{The measure of maximal entropy} we prove Theorem~\ref{Thm : proba d'entropie max dans le cas compact}. The main three steps are: estimating the measure of small dynamical balls (Section~\ref{the measure of dynamical balls}); bounding from below the entropy of the Bowen--Margulis measure (Section~\ref{ssection:entropie max}); proving the uniqueness of the measure of maximal entropy (Section~\ref{ssection:uniciteMME}). In Section~\ref{Counting closed geodesics} we establish Proposition~\ref{Prop : comptage dans le cas compact}. The main steps are: bounding the number of rank-one closed geodesics from below (Section~\ref{ssection:borne inf}) and from above (Section~\ref{ssection:borne sup rk1}); bounding from above the number of non-rank-one closed geodesics (Section~\ref{ssection:borne sup rksup}); and finally proving the equidistribution of closed geodesics (Section~\ref{ssection:equidistribution}).

\paragraph{Acknowledgements} I am grateful to Yves Benoist, Harrison Bray, Olivier Glorieux, Ludovic Marquis, Fr\'{e}d\'{e}ric Paulin and Samuel Tapie for helpful discussions. I thank my advisor Fanny Kassel for her time, help, advice and encouragements.

I warmly thank Feng Zhu for many fruitful discussions which led to the paper \cite{EeERfH+}, and made me realise that the strongly irreducible assumption in the main results of the present paper could be replaced by a weaker non-elementary assumption.

This project received funding from the European Research Council (ERC) under the European Union's Horizon 2020 research and innovation programme (ERC starting grant DiGGeS, grant agreement No 715982).

\section{Reminders}\label{Reminders}

\subsection{Properly convex open subsets of \texorpdfstring{$\PR(\R^{d+1})$}{PRd} and their geodesic flow}\label{distance}

In the whole paper we fix a real vector space $V=\R^{d+1}$, where $d\geq 1$. Let $\Omega\subset \PR(V)$ be a properly convex open set. Recall that $\Omega$ admits an $\Aut(\Omega)$-invariant proper metric called the \emph{Hilbert metric} and defined by the following formula: for $(a,x,y,b)\in\partial\Omega\times\Omega\times\Omega\times\partial\Omega$ aligned in this order (see Figure~\ref{figure_distance}),
\begin{equation}\label{Equation : Hilbert}
d_\Omega(x,y)=\frac{1}{2}\log([a,x,y,b]),
\end{equation}
where $[a,x,y,b]$ is the cross-ratio of the four points, normalised so that $[0,1,t,\infty]=t$. 

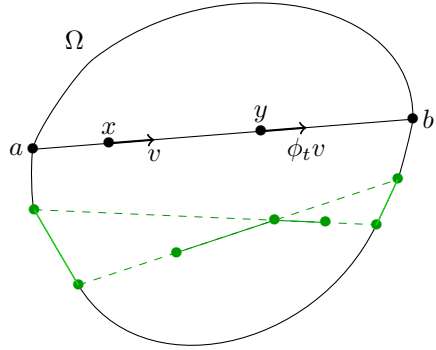
\begin{figure}
\centering
\begin{tikzpicture}[scale=2]
\coordinate (a) at (-0.7,-0.1);
\coordinate (b) at (1.8,0.1);
\coordinate (c) at (-0.4,-1);
\coordinate (d) at (1.56,-0.6);
\coordinate (e) at (-0.3,0.5);
\coordinate (f) at (1.7,-.3);
\coordinate (g) at ($ (f)!.5!(d) +(.001,0) $);
\coordinate (h) at (-.69,-.5);
\coordinate (i) at ($ (h)!.5!(c) +(-.001,0) $);

\coordinate (z) at (intersection of c--f and h--d);
\coordinate (zp) at ($ (z)!.5!(d) $);
\coordinate (zm) at ($ (c)!.5!(z) $);

\coordinate (x) at ($ (a)!0.2!(b) $);
\coordinate (y) at ($ (a)!0.6!(b) $);
\coordinate (v) at ($ (x)!0.3!(y) $);
\coordinate (phitv) at ($ (y)!.3!(b) $);


\length{e,a,h,i,c,d,g,f,b}
\cvx{e,a,h,i,c,d,g,f,b}{1}

\draw (a)--(b);

\draw (a) node{$\bullet$} node[left]{$a$};
\draw (x) node{$\bullet$} node[above]{$x$};
\draw (y) node{$\bullet$} node[above]{$y$};
\draw (b) node{$\bullet$} node[right]{$b$};
\draw (c) node[green!60!black]{$\bullet$};
\draw (d) node[green!60!black]{$\bullet$};
\draw (f) node[green!60!black]{$\bullet$};
\draw (h) node[green!60!black]{$\bullet$};
\draw (z) node[green!60!black]{$\bullet$};
\draw (zp) node[green!60!black]{$\bullet$};
\draw (zm) node[green!60!black]{$\bullet$};
\draw (e) node[above left]{$\Omega$};
\draw [->, thick] (x) -- (v) node[below]{$v$};
\draw [->, thick] (y) -- (phitv) node[below]{$\phi_tv$};
\draw [green!60!black,dashed,very thin] (c)--(f);
\draw [green!60!black,dashed,very thin] (h)--(d);
\draw [green!60!black] (zm)--(z);
\draw [green!60!black] (z)--(zp);
\draw [green] (d)--(f);
\draw [green] (h)--(c);

\end{tikzpicture}
\caption{The Hilbert metric and the geodesic flow ($t=d_\Omega(x,y)$)}\label{figure_distance}
\end{figure}

Recall that if $\Omega$ is an ellipsoid, then $(\Omega,d_\Omega)$ is the Klein model of the real hyperbolic space of dimension $d$; if $\Omega$ is a $d$-simplex, then $(\Omega,d_\Omega)$ is isometric to $\R^d$ endowed with a hexagonal norm.

Any discrete subgroup $\Gamma\subset\PGL(V)$ of automorphisms of $\Omega$ preserves $d_\Omega$, hence must act properly discontinuously on $\Omega$; therefore the quotient $M=\Omega/\Gamma$ is an orbifold. Furthermore, $M$ is a manifold if the action is free (\ie if $\Gamma$ is torsion-free, by Brouwer's fixed point theorem, applied to the convex hull of a finite orbit of a torsion element). Note that by Selberg's lemma \cite{selberg_lemma}, if $\Gamma$ is finitely generated, then it has a torsion-free finite-index subgroup. We will work in general with $\Gamma$ not necessarily torsion-free, so we set the notation $T^1M=T^1\Omega/\Gamma$. 

The straight lines of $\Omega$ can be parametrised to be geodesics, which are said to be \emph{straight}. However, an interesting feature in the non-strictly convex case is that when there are two coplanar non-trivial segments in $\partial\Omega$, one can construct geodesics which are not straight, see the broken green segment in Figure~\ref{figure_distance}. In order to define the geodesic flow we only take into account straight geodesics: for $v$ in $T^1\Omega$, let $t\mapsto c(t)$ be the parametrisation of the projective line tangent to $v$ such that $c$ is an isometric embedding from $\R$ to $\Omega$ and $c'(0)=v$. For $t\in\R$ we set $\phi_t(v)=c'(t)\in T^1\Omega$. See Figure~\ref{figure_distance}.


The geodesic flow on $T^1\Omega/\Gamma$ is well defined because the two actions of $\Aut(\Omega)$ and $(\phi_t)_{t\in\R}$ on $T^1\Omega$ commute. We denote by $\pi :T^1M\rightarrow M$ and $\pi : T^1\Omega\rightarrow \Omega$ the projections, we define the following metrics:
\begin{align*}
 \forall x,y\in M, \ d_M(x,y) & = \min\{d_\Omega(\xx,\yy): \xx,\yy\in\Omega \text{ lifts of } x,y\},\\
 \forall v,w\in T^1\Omega, \ d_{T^1\Omega}(v,w) & =  \max_{0\leq t\leq 1}d_\Omega(\pi\phi_tv,\pi\phi_tw),\\
 \forall v,w\in T^1M, \ d_{T^1M}(v,w) & = \max_{0\leq t\leq 1}d_\Omega(\pi\phi_tv,\pi\phi_tw) \leq \min\{d_{T^1\Omega}(\vv,\ww): \vv,\ww\in T^1\Omega \text{ lifts of } v,w\}.
\end{align*}

A very useful inequality is provided by the following lemma.

\begin{lemma}[\!{\!\cite[Lem.\,8.3]{crampon}}]\label{crampon}
Let $\Omega$ be a properly convex open subset of $\PR(V)$. Let $c_1$ and $c_2$ be two straight geodesics parametrised with constant speed, but not necessarily with the same speed. Then for all $0\leq t\leq T$,
\[d_{\Omega}(c_1(t),c_2(t))\leq d_\Omega(c_1(0),c_2(0))+d_\Omega(c_1(T),c_2(T)).\]
\end{lemma}

See \cite[\S A]{topmixing} for a proof of the above lemma that fill in a missing detail in the original proof.

\subsection{Terminology on convex sets and duality}\label{duality}


We recall here some terminology on convex sets.

\begin{nota}
 For any subset $X$ of the projective space $\PR(V)$, the closure (\resp interior \resp boundary) of $X$, denoted by $\overline{X}$ (\resp $\interior(X)$ \resp $\partial X$), will always be considered \emph{with respect to $\PR(V)$}.
\end{nota}

Let $K\subset\PR(V)$ be properly convex. 

\begin{itemize}
 \item The \emph{relative interior} (\resp \emph{relative boundary}) of $K$, denoted by $\interior_{rel}(K)$ (\resp $\partial_{rel}K$) is its topological interior (\resp boundary) with respect to the projective subspace it spans. 
 \item For $x\in \overline{K}$, the \emph{open face} of $x$ in $\overline{K}$, denoted by $F_{K}(x)$, consists of the points $y\in \overline{K}$ such that $[x,y]$ is contained in the relative interior of a segment contained in $\overline{K}$. The \emph{closed face} of $x$ is $\overline{F}_K(x)=\overline{F_K(x)}$. 
 \item A point $x\in\partial_{rel}{K}$ is said to be \emph{extremal} (\resp \emph{strongly extremal}) if $F_K(x)=\{x\}$ (\resp $x\not\in\overline{F}_K(y)$ for $y\in\partial_{rel} K\smallsetminus\{x\}$); one says that $K$ is \emph{strictly convex} if all the points in the relative boundary are extremal (and hence strongly extremal).
 \item Assume that $K$ spans $\PR(V)$ and let $\xi\in\partial K$. A \emph{supporting hyperplane} of $K$ at $\xi$ is a hyperplane which contains $\xi$ but does not intersect $\interior (K)$. Note that there always exists such a hyperplane. The point $\xi$ is said to be a \emph{smooth} point of $\partial K$ if there is only one supporting hyperplane of $K$ at $\xi$, denoted by $T_\xi\partial K$.
\end{itemize}
 
Let us recall the notion of duality for properly convex open sets. We identify the dual projective space $\PR(V^*)$ with the set of projective hyperplanes of $\PR(V)$. Let $\Omega$ be a properly convex open subset of $\PR(V)$. The dual of $\Omega$, denoted by $\Omega^*$, is the properly convex open subset of $\PR(V^*)$ defined as the set of projective hyperplanes which do not intersect $\overline{\Omega}$. We naturally identify $\PGL(V)$ and $\PGL(V^*)$, then $\Aut(\Omega)$ identifies with $\Aut(\Omega^*)$, and the attracting (\resp repelling) fixed point of the action on $\PR(V^*)$ of any biproximal element $g\in\PGL(V)$ is $x_g^+\oplus x_g^0$ (\resp $x_g^-\oplus x_g^0$).

Observe that a hyperplane $H$ is a smooth point of $\partial\Omega^*$ if and only if its intersection with $\overline{\Omega}$ is reduced to a singleton.

\subsection{More metrics and upper semi-continuity of balls}

In this section we define two metrics, one on the usual boundary $\partial\Omega$ of any properly convex open set $\Omega\subset\PR(V)$, and one on the closure $\overline\Omega$. These notions are used in the proofs of Lemma~\ref{l ombre grandit a la frontiere} and Proposition~\ref{l'ensemble limite conique a mesure non nulle}.

\begin{defi}\label{les metriques du bord}
Let $x\in\Omega\subset\PR(V)$ be a pointed properly convex open set, and $t>0$ be a positive parameter. 
\begin{itemize}
\item Let $\xi,\eta\in\partial\Omega$. The \emph{simplicial distance} between $\xi$ and $\eta$ is defined as follows (and it is possibly infinite).
\[d_{\mathrm{spl}}(\xi,\eta):=\inf\{k : \exists a_0,\dots,a_k\in\partial\Omega : a_0=\xi, a_k=\eta, \forall 0\leq i<k, [a_i,a_{i+1}]\subset\partial\Omega\}.\]
\item Let $\xi,\eta\in\overline\Omega$. When $\xi$ and $\eta$ are on the same open face $F$ of $\overline\Omega$, we set $d_{\overline{\Omega}}(\xi,\eta):=d_F(\xi,\eta)$, and otherwise we set $d_{\overline{\Omega}}(\xi,\eta)$ to be infinite.
\end{itemize}
\end{defi}

The open balls of radius $R$ and centred at $\xi$ are respectively denoted by $B_{\spl}(\xi,R)$ and $B_{\overline{\Omega}}(\xi,R)$.

Observe that $d_{\spl}$ and $d_{\overline{\Omega}}$ are lower semi-continuous. In particular, closed balls for the metric $d_{\overline{\Omega}}$ are upper semi-continuous with respect to the Hausdorff topology (in the sense of Fact~\ref{semi-continuite superieure} below). Given a metrisable locally compact topological space $X$, recall that the \emph{Hausdorff topology} on the set of compact subsets of $X$ is a metrisable topology such that a sequence of compact subsets $(K_n)_n$ converges to $K\subset X$ compact if and only if any converging sequence $(x_n)_n\in \prod_nK_n$ has its limit in $K$, and any point of $K$ is the limit of a sequence $(x_n)_n\in \prod_nK_n$.
 
\begin{fait}\label{semi-continuite superieure}
 Let $\Omega\subset\PR(V)$ be a properly convex open set. For any $R>0$, the map
 \begin{equation*}
  \begin{array}{lllll}
   \overline{B}_{\overline{\Omega}}(\cdot,R) & : & \overline{\Omega} & \longrightarrow & \{\text{compact subsets of }\overline{\Omega}\}\\
   && \xi & \mapsto & \overline{B}_{\overline{\Omega}}(\xi,R)
  \end{array}
 \end{equation*}
 is upper semi-continuous in the following sense: all accumulation points of $\overline{B}_{\overline{\Omega}}(\eta,R)$ when $\eta\to\xi$ must be contained in $\overline{B}_{\overline{\Omega}}(\xi,R)$.
\end{fait}

\subsection{Benz\'ecri's compactness theorem and proper densities}

In this section we recall Benz\'ecri's famous compactness theorem, and we see a first consequence for $\PGL(V)$-equivariant volume form on properly convex open sets. We denote by $\Ecal_V$ (resp. $\Ecal_V^\bullet$) the set of (resp. pointed) properly convex open set of $\PR(V)$ (resp. endowed with the pointed Hausdorff topology).

\begin{fait}[\!{\!\cite[Ch.\,5,\S2,Th.\,2]{benz_varlocproj}}]\label{benz}
 The action of $\PGL(V)$ on $\Ecal_V^\bullet$ is continuous, proper and cocompact.
\end{fait}

We recall the notion of a proper density on the set of properly convex open sets. They prescribe the way to choose a volume form on each properly convex open set in a $\PGL(V)$-equivariant manner. For the whole paper we fix a density $\Vol_{\PR(V)}$ on $\PR(V)$, seen as a measure.

\begin{defi}\label{Definition : densite propre}
 A \emph{proper density} on $\Ecal_V$ is a map of the form $\Omega\mapsto\Vol_\Omega$, where $\Omega\subset\PR(V)$ is a properly convex open set and $\Vol_\Omega$ is a density on $\Omega$ with Radon--Nikodym derivative $f(x,\Omega)>0$ with respect to $\Vol_{\PR(V)}$, satisfying the following three conditions.
 \begin{itemize}
  \item (Continuity) The function $f : \Ecal_V^\bullet\rightarrow \R_{>0}$ is continuous.
  \item (Monotone decreasing) Let $(x,\Omega)$ and $(y,\Omega')\in \Ecal_V$. If $x=y\in\Omega\subset\Omega'$ then 
  \[f(x,\Omega')\leq f(x,\Omega).\]
  \item ($\PGL(V)$-equivariance) For any $T\in\PGL(V)$, 
  \[T_*\Vol_\Omega=\Vol_{T(\Omega)}.\]
 \end{itemize}
\end{defi}

 See \cite{vernicosapprox2017} for more details and examples. We fix for the whole paper a proper density $\Omega\mapsto\Vol_\Omega$ on $\Ecal_V$. One of the key observations that we will need on proper densities is that for any $R>0$, the following quantities are positive and finite (this is a direct consequence of Fact~\ref{benz}).
 \begin{gather}\label{Equation : chi}
  0<\chi_-(R) := \min_{(x,\Omega)\in\Ecal_V^\bullet}\Vol_\Omega(\overline{B}_\Omega(x,R))\leq\chi_+(R) := \max_{(x,\Omega)\in\Ecal_V^\bullet}\Vol_\Omega(\overline{B}_\Omega(x,R))<\infty.
 \end{gather}

\subsection{The critical exponent}\label{Section : Exposant critique}

Let $\Gamma\subset\PGL(V)$ be a discrete subgroup which preserves a properly convex open set $\Omega\subset\PR(V)$. We defined in the introduction (Sections~\ref{Section : Exposant critique} and \ref{Section : la machinerie des densites conformes}) the critical exponent of $\Gamma$ and when $\Gamma$ is divergent. In this section we explain why these definitions do not depend on $\Omega$.

Recall that $V=\R^{d+1}$ has a canonical Euclidean structure. For any $g\in\GL(V)=\GL_{d+1}(\R)$, we denote by $\mu_1(g)\geq\dots\geq\mu_{d+1}(g)$ the singular values of $g$. For any $g\in\PGL(V)$, we set 
\[\kappa(g):=\frac{1}{2}\log\left(\frac{\mu_1(\tilde{g})}{\mu_{d+1}(\tilde{g})}\right),\]
where $\tilde{g}\in\GL(V)$ is any lift of $g$.

\begin{fait}[\!{\!\cite[Prop.\,10.1]{fannycvxcocpct}}]
Let $\Omega\subset\PR(V)$ be a properly convex open subset and let $x\in\Omega$. Then there exists a constant $C>0$ such that for any $g\in\Aut(\Omega)$, one has
\[|d_\Omega(x,gx) - \kappa(g)| \leq C.\]
\end{fait}

As a consequence, if $\Gamma\subset\PGL(V)$ is a discrete subgroup which preserves a properly convex open set $\Omega\subset\PR(V)$, then
\begin{equation}\label{Equation : exposant critique}
\delta_\Gamma=\limsup_{r\to\infty}\frac{1}{r}\log\left(\#\{\gamma\in\Gamma : \kappa(\gamma)\leq r\}\right),
\end{equation}
and  
\begin{equation}\label{Equation : divergent}
\Gamma \text{ is divergent if and only if } \sum_{\gamma\in\Gamma}e^{-\delta_\Gamma \kappa(\gamma)}=\infty.
\end{equation}
If $\Gamma$ does not preserve any properly convex open set, then we take \eqref{Equation : exposant critique} and \eqref{Equation : divergent} as definitions.

We end this section with the following classical fact, which is a consequence for instance of the Tits alternative and of the sub-additivity of $\kappa$. It applies to all strongly irreducible discrete subgroups of $\PGL(V)$, and non-elementary rank-one discrete groups of properly convex open sets.

\begin{fait}\label{exposant critique non nul}
 Let $\Gamma\subset\PGL(V)$ be a discrete subgroup that is not virtually solvable. Then $\delta_\Gamma>0$.
\end{fait}
 
\subsection{Proximal linear transformations}\label{Proximal linear transformations}

In this section we recall the notion of a proximal linear transformation, which was used in the definition of the proximal limit set $\Lambda^{\prox}$ and the biproximal unit tangent bundle $T^1M_{\bip}$.

\begin{nota}\label{oplus}
If $W_1$ and $W_2$ are two subspaces of $V$ such that $W_1\cap W_2=\{0\}$, we write $W_1\oplus W_2\subset V$ for their direct sum and $\PR(W_1)\oplus\PR(W_2)=\PR(W_1\oplus W_2)$ for its projectivisation. In particular, if $x,y\in \PR(V)$ are two distinct points, we write $x\oplus y$ for the projective line through $x$ and $y$.
\end{nota}

\begin{defi}\label{def_prox}
A linear transformation $g\in \End(V)$ is \emph{proximal} if it has exactly one complex eigenvalue with maximal modulus among all eigenvalues, and if this eigenvalue has multiplicity $1$. The associated eigenline in $\PR(V)$ is the attracting fixed point of $g$ and is denoted by $x_g^+$.

An invertible linear transformation $g\in\GL(V)$ is said to be \emph{biproximal} if $g$ and $g^{-1}$ are proximal. The attracting fixed point of $g^{-1}$ is the repelling fixed point of $g$ and is denoted by $x_g^-$. The projective line $x_g^+\oplus x_g^-$ (see Notation~\ref{oplus}) is the \emph{axis} of $g$ and is denoted by $\axis(g)$. The $g$-invariant complementary subspace to the axis of $g$ is denoted by $x_g^0$. Note that the notions of biproximality, attracting/repelling fixed point, and axis, are well defined for the image of $g$ in $\PGL(V)$.
\end{defi}

More generally for any projective transformation $g\in\PGL(V)$, one can define the attracting subspace $x_g^+$ of $g$ as the span of all generalised eigenspaces associated to eigenvalues with maximal norm, and $x_g^0$ and $x_g^-$ can be defined similarly.

\begin{rqq}\label{prox_are_open}
The set of proximal linear transformations is open in $\End(V)$, and the map sending a proximal linear transformation to the pair (attracting fixed point, maximal eigenvalue) is continuous.
\end{rqq}

\begin{rqq}\label{minimality}
 As observed by Benoist \cite[Lem.\,3.6.ii]{BenoistPropAsymp}, for any irreducible subgroup $\Gamma\subset\PGL(V)$ which contains a proximal element, the proximal limit set is the smallest closed $\Gamma$-invariant non-empty subset of $\PR(V)$; in particular, the action of $\Gamma$ on  $\Lambda^{\prox}$ is minimal (\ie any orbit is dense). Indeed, consider any proximal element $\gamma\in\Gamma$, and let $\PR(W)\subset \PR(V)$ be the $\gamma$-invariant complementary subspace  to $x_\gamma^+$. By irreducibility, any closed $\Gamma$-invariant non-empty subset $X\subset\PR(V)$ contains a point $x$ outside $\PR(W)$, and then $x_\gamma^+$, which is the limit of the sequence $(\gamma^nx)_{n\in\N}$, belongs to $X$.
\end{rqq}

\subsection{Periodic straight geodesics and elements of \texorpdfstring{$\Gamma$}{Gamma}}\label{Periodic orbits and elements of Gamma}

In this section we recall the link between periodic geodesics in $T^1\Omega/\Gamma$ and conjugacy classes of $\Gamma$. Let $\Omega\subset \PR(V)$ be a properly convex open set. Let $g\in \Aut(\Omega)$. Then
\begin{equation}\label{longueur de translation}
\ell(g)=\inf\{d_\Omega(x,g\cdot x): \ x\in\Omega\}\geq 0.
\end{equation}
Where $\ell(g)$ was defined in \eqref{ell}. The right-hand side of \eqref{longueur de translation} is called the \emph{translation length} of $g$. See \cite[Prop.\,2.1]{CLT2015cvxisom} for a proof.

Combined with an elementary computation, \eqref{longueur de translation} yields:

\begin{fait}\label{period_is_translation_length} Let $\Omega\subset\PR(V)$ be a properly convex open set and $\Gamma\subset\Aut(\Omega)$ a discrete subgroup; denote $\Omega/\Gamma$ by $M$. Then for any lift in $\Omega$ of any periodic straight geodesic of $M$, there is an automorphism $\gamma\in\Gamma$ which preserves it and acts by positive translation on it. Let $\tilde{\gamma}\in\GL(V)$ be a lift of $\gamma$. The endpoints in $\partial\Omega$ of the geodesic are fixed by $\gamma$, the associated eigenvalues of $\tilde{\gamma}$ are $\lambda_1(\tilde{\gamma})$ and $\lambda_{d+1}(\tilde{\gamma})$, and the length of the geodesic in $M$ is the translation length of $\gamma$. If furthermore these endpoints are extremal (for example if $\gamma$ is rank-one, see Definition~\ref{def rank-one}), then $\gamma$ is biproximal.
\end{fait}

\begin{defi}\label{biprox_orbit}
 Let $\Omega\subset \PR(V)$ be a properly convex open set and $\Gamma\subset\Aut(\Omega)$ a discrete subgroup. Let $\gamma\in\Gamma$ be a biproximal element whose axis meets $\Omega$. Then the periodic geodesic associated to $\gamma$ is said to be \emph{biproximal}, and the vectors along this geodesic are said to be \emph{biproximal} periodic.
\end{defi}

There are cases where $\gamma\in\Gamma$ is biproximal but its axis does not intersect $\Omega$ (e.g.\ when $\Omega$ is a triangle, or is symmetric). Then we cannot make sense of a straight periodic geodesic associated to~$\gamma$.

\subsection{Biproximal vs rank-one periodic geodesics}\label{bip VS rk1}

In this section we explain when a biproximal periodic geodesic is rank-one. We saw in Fact~\ref{period_is_translation_length} that a rank-one periodic geodesic is always biproximal. Conversely, the endpoints of a biproximal periodic geodesic are smooth, but not always strongly extremal.

\begin{fait}[\!{\!\cite[Lem.\,3.1]{topmixing}}]\label{les extremites des geodesiques periodiques biproximales sont lisses}
Let $\Omega\subset \PR(V)$ be a properly convex open set. Let $g$ be a biproximal automorphism of $\Omega$. Then $\axis(g)\cap\Omega$ is non-empty if and only if $x_g^+$ is smooth; in this case $T_{x_g^+}\partial\Omega=x_g^+\oplus x_g^0$.
\end{fait}

The following result is later completed by Corollary~\ref{cor:carac rgun}.

\begin{fait}[\!{\!\cite[Lem.\,3.2]{topmixing},\ \cite[Prop.\,6.3]{islam_rank_one}}]\label{equivalences rang un}
Let $\Omega\subset \PR(V)$ be a properly convex open set. Let $g\in\PGL(V)$ be a biproximal automorphism of $\Omega$. Then the following are equivalent: 
\begin{enumerate}[label=(\alph*)]
\item $g$ is rank-one;
\item $x_g^+,x_g^-\in\partial\Omega$ are smooth and strongly extremal points;
\item $x_g^+$ is strongly extremal;
\item the element $g$ seen as an automorphism of $\Omega^*$ is rank-one;
\item \label{Item : dualite de rg1}the axis of $g$ in $\PR(V)$ intersects $\Omega$, and the axis of $g$ in $\PR(V^*)$ intersects $\Omega^*$;
\item \label{Item : d_spl>2} $d_{\spl}(x_g^+,x_g^-)\geq 3$.
\end{enumerate}
\end{fait}

\begin{proof}
 In \cite{topmixing}, the point $(g)$ was missing, so we briefly check that it is equivalent to the other points.
 \begin{itemize}
  \item $(c)$ implies $(g)$ because it implies that $d_{\spl}(x_g^+,x_g^-)=\infty$.
  \item If $d_{\spl}(x_g^+,x_g^-)\geq 3$, then the axis of $g$ in $\PR(V)$ intersects $\Omega$ so $x_g^+$ is smooth. Furthermore $x_g^0$ does intersect $\partial\Omega$, because otherwise $d_{\spl}(x_g^+,x_g^-)\leq 2$, since $x_g^+\oplus x_g^0$ and $x_g^-\oplus x_g^0$ are supporting hyperplanes of $\Omega$.\qedhere
 \end{itemize}
\end{proof}

These characterisations of rank-one automorphisms yield a characterisation of rank-one manifolds. To see this we need the following fact.

\begin{fait}[\!{\!\cite[Prop.\,1.1]{benoist2000automorphismes}}\ \&\ {\cite[Lem.\,3.6.iv]{BenoistPropAsymp}}]\label{Fait : des bips partout}
 Let $\Gamma\subset \PGL(V)$ be a strongly irreducible subgroup.
\begin{enumerate}[label=(\arabic*)]
\item If $\Gamma$ preserves a properly convex open set $\Omega\subset\PR(V)$, then it contains a proximal element.
\item \label{Item : des bips partout} If $\Gamma$ contains a proximal element, then $\{(x_\gamma^+,x_\gamma^-): \gamma\in \Gamma \text{ biproximal}\}$ is dense in $\Lambda^{\prox}\times\Lambda^{\prox}$.
\end{enumerate}
\end{fait}

\begin{cor}\label{Corollaire : rang 1 = sse non vide}
 Let $\Omega\subset \PR(V)$ be a properly convex open set and $\Gamma\subset\Aut(\Omega)$ a strongly irreducible discrete subgroup. Then $M=\Omega/\Gamma$ is rank-one if and only if $\Lambda^{\prox}$ contains a two points at simplicial distance at least $3$.
\end{cor}

\begin{proof}
 If $M$ is rank-one, then $\Lambda^{\prox}$ contains two strongly extremal points by definition. Conversely, if there exists $\xi,\eta\in\Lambda^{\prox}$ at simplicial distance at least $3$. By lower semi-continuity of $d_{\spl}$ and by Fact~\ref{Fait : des bips partout}.\ref{Item : des bips partout}, there exists $\gamma\in\Gamma$ such that $d_{\spl}(x_\gamma^+,x_\gamma^-)\geq 3$. By Fact~\ref{equivalences rang un}.\ref{Item : d_spl>2}, $\gamma$ is rank-one and $M$ as well.
\end{proof}
%

Exemple~\ref{Exemple : dim3->rg1} below is an elementary application of the previous result. It can also be seen as a consequence of Zimmer's higher-rank rigidity \cite[Th.\,1.4]{zimmer_higher_rank}.

\begin{ex}\label{Exemple : dim3->rg1}
 Any $3$-dimensional irreducible compact convex projective manifold is rank-one.
\end{ex}

\begin{proof}
 Let $M=\Omega/\Gamma$ be an irreducible compact convex projective manifold of dimension $3$. By \cite[Prop.\,3.10.a]{CD4}, $\Lambda^{\prox}=\partial\Omega$, and by \cite[Th.\,1.1.d \& Prop.\,3.8]{CD4}, $\partial\Omega\smallsetminus\partial_{\sse}\Omega$ is contained in the union of countably many properly embedded triangles (PET) (namely PES's of dimension $2$) of $\Omega$. These cannot cover the whole boundary $\partial\Omega$ (for instance because they have Lebesque measure zero), hence there is a point in $\partial_{\sse}\Omega=\partial_{\sse}\Omega\cap\Lambda^{\prox}$. By irreducibility, there is another one, and we  can conclude using Corollary~\ref{Corollaire : rang 1 = sse non vide}.
\end{proof}

Remark~\ref{minimality} and Fact~\ref{Fait : des bips partout} are useful results, and we will need them to hold for certain discrete groups $\Gamma\subset\PGL(V)$ that are not necessarily irreducible. The following fact is exactly what will be needed. Its proof is quite similar to the proof of Fact~\ref{Fait : des bips partout}, yet somehow easier, because rank-one elements are easier to manipulate than general biproximal elements (in particular when there is no invariant properly convex open set).
\begin{fait}[\!{\!\cite{EeERfH+}}]\label{fait:des rk1 partout}
 Let $\Omega\subset\PR(V)$ be a properly convex open set and $\Gamma\subset\Aut(\Omega)$ a non-elementary rank-one discrete subgroup. Then $\Lambda^{\prox}\subset\overline{\Omega}$ is the smallest $\Gamma$-invariant closed subset, and it has no isolated point. Furthermore, $\{(x_\gamma^+,x_\gamma^-)\in\Lambda^{\prox}\times\Lambda^{\prox} : \gamma\in\Gamma \text{ rank-one}\}$ is dense in $\Lambda^{\prox}\times\Lambda^{\prox}$.
\end{fait}

\subsection{Horoboundaries and Patterson--Sullivan densities}\label{rappel sur les densites conformes}

In this section we recall the classical definitions of horofunctions and horoboundary, and how they can be used to define Patterson--Sullivan densities.

\begin{defi}
 Let $(X,d)$ be a proper metric space. The \emph{horofunction} at points $x,y,z\in X$ is defined as follows :
 \[b_z(x,y)=d(x,z)-d(y,z).\]
\end{defi}

We now recall the definition of the horocompactification. By a \emph{compactification} of a topological space $X$ we mean a compact topological space $Y$ together with an embedding $X\hookrightarrow Y$ with open and dense image; then the subset $Y\smallsetminus X=\partial_Y X$ is called the boundary of the compactification. Another compactification $Z$ \emph{dominates} $Y$ if there is a continuous map from $Z$ to $Y$ which is compatible with the embeddings of $X$. Using the Arzel\`a--Ascoli theorem, one can show that the following is well defined.

\begin{defi}[\!{\!\cite[\S3]{GromovVrac}}]\label{horobord}
 Let $(X,d)$ be a proper metric space. The \emph{horocompactification} of $(X,d)$ is the smallest compactification $\overline{X}^{\hor}$ of $X$ such that $z\mapsto b_z(x,y)$ extends continuously to $\overline{X}^{\hor}$ for every $x,y\in X$. The \emph{horoboundary} $\partial_{\hor}X$ of $(X,d)$ is the boundary of $\overline{X}^{\hor}$. 
\end{defi}

Note that $\overline{X}^{\hor}$ is metrisable and the function $(\xi,x,y)\in \overline{X}^{\hor}\times X\times X \mapsto b_\xi(x,y)$ is continuous. Any isometry of $X$ extends continuously to a bi-Lipschitz homeomorphism of $\overline{X}^{\hor}$.

We now recall the definition of Patterson--Sullivan measures on the horoboundary of a proper metric space.

\begin{defi}\label{densites conformes}
 Let $(X,d)$ be a proper metric space. Let $\Gamma < \Isom(X,d)$ be a discrete subgroup. Given $\delta\in\R$, a ($\Gamma$-equivariant) \emph{$\delta$-conformal density} on $\partial_{\hor}X$ is a family of finite measures $(\mu_x)_{x\in X}$ on $\partial_{\hor}X$ such that
 \begin{itemize}
  \item $\mu_y$ is absolutely continuous with respect to $\mu_x$ for all $x,y\in X$, and the Radon--Nikodym derivative is :
  \[\frac{d\mu_y}{d\mu_x}(\xi)=e^{-\delta \,b_\xi(y,x)};\]
  (This implies that the family is entirely determined by $\mu_o$ for any $o\in X$.)
  \item for every $\gamma\in\Gamma$ and $x\in X$ the push-forward by $\gamma$ of $\mu_x$ is :
  \[\gamma_*\mu_x=\mu_{\gamma x}.\]
 \end{itemize}
\end{defi}

Let us recall the classical example of a conformal density, which we will need in this paper. For any measured metric space $(X,d,\mu)$ with infinite mass, the \emph{volume entropy} is, for any $o\in X$,
\begin{equation}\label{Equation : entropie volumique}
\delta_{\mu}:=\limsup_{r\to\infty}\frac{\log\mu(B_X(o,r))}{r}\in\R_{\geq 0}\cup\{\infty\}.
\end{equation}

\begin{fait}[\!{\!\cite[\S3]{Patterson76}}]\label{smear}
 Let $(X,d)$ be a proper metric space, let $o\in X$ be a basepoint, let $\Gamma$ be a non-compact closed subgroup of $\Isom(X,d)$ and let $\Vol$ be a $\Gamma$-invariant Radon measure on $X$. We assume that the volume entropy $\delta_{\Vol}$ is finite. Then there exists a continuous non-decreasing function $h:\R_+\rightarrow\R_{>0}$ such that
 \begin{itemize}
  \item $\int_{x\in X}h(d(o,x))e^{-\delta_{\Vol}d(o,x)}\diff\Vol(x)=\infty$,
  \item for every $\epsilon>0$, there exists $R>0$ such that $h(r+t)\leq e^{\epsilon t}h(r)$ for any $r\geq R$ and $t\geq 0$.
 \end{itemize}
 Furthermore, if we define, for $s>\delta_{\Vol}$, the probability measure $\mu_{o,s}$ on $X$ such that
 \[\mu_{o,s}(A)=\frac{\int_{x\in A}h(d(o,x))e^{-sd(o,x)}\diff \Vol(x)}{\int_{x\in X}h(d(o,x))e^{-sd(o,x)}\diff\Vol(x)},\] 
for any Borel subset $A\subset X$, then any accumulation point of $(\mu_{o,s})_{s\to\delta_{\Vol}}$ in the space $\p(\b{X}^{\hor})$ of probability measures on $\overline{X}^{\hor}$ is supported on $\partial_{\hor}\Omega$ and is a $\delta_{\Vol}$-conformal density.
 \end{fait}

A typical example of $\Gamma$-invariant Radon measure on $X$ is push-forward by an orbital map of the Haar measure on $\Gamma$. 

Let $\Omega\subset\PR(V)$ be a properly convex open set. One can check that the critical exponent $\delta_\Gamma$ of any discrete group $\Gamma\subset\Aut(\Omega)$ is equal to the volume entropy of the push-forward by any orbital map of the counting measure on $\Gamma$, and also that $\delta_\Gamma\leq \delta_{\Vol_\Omega}$.
 
\begin{fait}[\!{\!\cite[Th.\,2]{entropie<d-2}}]\label{entropie<d-2}
  Let $\Omega\subset\PR(V)$ be a properly convex open set. Then $\delta_{\Vol_\Omega}\leq \dim(V)-2$.
  
  In particular, for any discrete subgroup $\Gamma\subset\Aut(\Omega)$, the numbers $\delta_\Gamma$ and $\delta_{\Vol_\Omega}$ are finite.
\end{fait}

This will allow us to apply Fact~\ref{smear} to our convex projective setting.

\subsection{The horoboundary of a properly convex open set} 

In this section we recall results on the horoboundary of a properly convex open set that were mentioned in the introduction. A direct and elementary proof of the following result may be found in \cite[Fact 6.1.2]{these}.

\begin{fait}[\!{\!\cite[Th.\,1.3]{horoboundary}}]\label{walsh}
 Let $\Omega\subset\PR(V)$ be a properly convex open set. Then the horocompactification $\overline{\Omega}^{\hor}$ dominates the projective compactification $\overline{\Omega}$.
\end{fait}

\begin{nota}
 Given a properly convex open set $\Omega\subset\PR(V)$, we denote by $\pi_{\hor}$ the map $\partial_{\hor}\Omega\rightarrow\partial\Omega$.
\end{nota}

\begin{defi}\label{densite conforme sur le bord de Omega}
Let $\Omega\subset\PR(V)$ be a properly convex open set and $\Gamma\subset\Aut(\Omega)$ a discrete subgroup. For $\delta\geq 0$, a \emph{$\delta$-conformal density on $\partial\Omega$} is the push-forward by $\pi_{\hor}$ of a $\delta$-conformal density on $\partial_{\hor}\Omega$. Note that such a family is made of $\Gamma$-quasi-invariant measures.
\end{defi}

If it was not true that $\partial_{\hor}\Omega$ dominates $\partial\Omega$, we could have considered a common refinement of $\partial_{\hor}\Omega$ and $\partial\Omega$, where the conformal densities are also well defined.

Thanks to the following fact, we will from now on abusively identify any smooth point on the projective boundary with its preimage in the horoboundary.

\begin{fait}[\!{\!\cite[Lem.\,3.2]{bray_ergodicity}}]\label{smooth_points_are_in_hor}
 Let $\Omega\subset\PR(V)$ be a properly convex open set. Let $\xi$ be a smooth point of $\partial\Omega$. Then it has only one preimage by $\pi_{\hor}:\partial_{\hor}\Omega\rightarrow\partial\Omega$.
\end{fait}

\subsection{The Hopf decomposition and quotients of measures}\label{Section : decompo de Hopf}

Let us fix for the whole section a locally compact, secound countable and unimodular group $G$ acting measurably on a standard Borel space $X$, and a $G$-invariant and $\sigma$-finite measure $\tilde{m}$ on $X$; fix a Haar measure on $G$, denoted by $\diff g$, and an integrable positive function $\sigma$ on $X$. For any non-negative measurable function $f$ on $X$, we denote by $\int_Gf$ the $G$-invariant measurable function defined by $\int_Gf(x)=\int_Gf(gx)\diff g$; this also denotes the induced function on $G\backslash X$.

A measurable subset $A\subset X$ is said to \emph{wandering} under the action of $G$ if for $\tilde{m}$-almost any $x\in A$, the \emph{transporter} $T(x,A):=\{g\in G : gx\in A\}$ is relatively compact. The following fact is classical, and serves as a definition of the Hopf decomposition.

\begin{fait}
 Let $\mathcal{C}:=\{\int_G\sigma=\infty\}$ and $\mathcal{D}:=\{\int_G\sigma<\infty\}\subset X$. The decomposition $X=\mathcal{C}\sqcup\mathcal{D}$ is a \emph{Hopf decomposition}, in the sense that every wandering subset of $\mathcal{C}$ has $\tilde{m}$-measure zero, and $\mathcal{D}$ is a countable union of wandering subsets of $X$. Any two Hopf decompositions agree on some $\tilde{m}$-full subset of $X$. The dynamical system $(X,G,\tilde{m})$ is said to be \emph{conservative} (\resp \emph{dissipative}) if $\tilde{m}(\mathcal{D})=0$ (\resp $\tilde{m}(\mathcal{C})=0$).
\end{fait}
\begin{proof}
 Let $A\subset \mathcal{C}$ be a measurable subset, and let us prove that $\int_G1_A$ is infinite on $\tilde{m}$-almost every point of $A$. On the one hand, if, for $R>0$, we denote $A_R:=\{\int_G1_A\leq R\}\cap A$, then
 \begin{align*}
  \int_{X\times G}\sigma(gx)1_{A_R}(x)\diff g\diff \tilde{m}(x) & = \infty\cdot \tilde{m}(A_R),
 \end{align*}
 while on the other hand, since $G$ is unimodular and $\tilde{m}$ is $G$-invariant,
 \begin{align*}
  \int_{X\times G}\sigma(gx)1_{A_R}(x)\diff g\diff \tilde{m}(x) &  = \int_{X\times G}\sigma(x)1_{A_R}(gx)\diff g\diff \tilde{m}(x) \\
  & = \int_X \sigma\left(\int_G1_{A_R}\right)\diff\tilde{m} \\
  & \leq R\int_X\sigma\diff\tilde{m}<\infty.
 \end{align*}
 Therefore $\tilde{m}(A_R)=0$ for any $R>0$, hence $\tilde{m}(\cup_RA_R)=0$. For any compact subset $K\subset G$, we consider $B_K:=\{\int_K\sigma > (1/2)\int_G\sigma\}\subset\mathcal{D}$, and observe that it is wandering. Indeed if $x\in B_K$ and $g\in G$ are such that $gx\in B_K$, then $\int_K\sigma(x) + \int_{Kg}\sigma(x) >\int_G\sigma(x)$, thus $K\cap Kg\neq\emptyset$ and $g\in K^{-1}\cdot K$ which is compact. Furthermore $\mathcal{D}=\cup_{n}B_{K_n}$ if $G=\cup_nK_n$.
\end{proof}
Note that if the action of $G$ on $X$ is \emph{smooth} (namely $G\backslash X$ is a standard Borel space) and has compact stabilisers, then $(X,G,\tilde{m})$ is dissipative. In particular, this observation applies when $X$ is a locally compact second countable topological space and the action of $G$ is continuous and proper.

\begin{defi}\label{Definition : quotient de mesure}
 If $(X,G,\tilde{m})$ is dissipative, the \emph{quotient} of $\tilde{m}$ on $G\backslash X$ is defined as
 \[m := \left(\int_G\sigma\right)^{-1}\pi_*(\sigma \tilde{m}),\]
 where $\pi$ denotes the projection $X\rightarrow G\backslash X$. For any non-negative (or integrable) function $f$ on $X$, 
 \begin{equation*}
  \int_Xf\diff \tilde{m} = \int_{G\backslash X}\left(\int_Gf\right)\diff m,
 \end{equation*}
 and $m$ is independent of the choice of $\sigma$ (but it depends on the choice of $\diff g$).
\end{defi}

\begin{fait}\label{Fait : conservatif passe au quotient}
 Suppose $G$ contains a unimodular, normal and closed subgroup $H$ whose action on $X$ is dissipative and smooth; fix a Haar measure on $H$. Set $\pi : X\rightarrow H\backslash X$ and $m$ to be the quotient of $\tilde{m}$ on $H\backslash X$; this measure is $G/H$-invariant. Then the Hopf decomposition of $X$ projects under $\pi$ onto the Hopf decomposition of $H\backslash X$. In particular, $(X,G,\tilde{m})$ is conservative (\resp dissipative) if and only if $(H\backslash X,G/H,m)$ is conservative (\resp dissipative).
\end{fait}

\begin{proof}
 Observe that the Haar measure on $G/H$ is the quotient of the Haar measure on $G$ by the action of $H$. Therefore, if $\sigma$ is a positive integrable function on $X$, then $\int_H\sigma$ is a positive integrable function on $H\backslash X$, and $\int_G\sigma=\int_{G/H}\int_H\sigma$. Moreover, for any $G$-invariant measurable subset $\tilde{A}\subset X$ whose projection $A$ in $H\backslash X$ is measurable, $\tilde{m}(\tilde{A})=\int_{H\backslash X}\left(\int_H1_{\tilde{A}}\right)\diff m= \int_{H\backslash X} \Vert\mathrm{Haar}_H\Vert 1_{A}\diff m$ is zero if and only if $m(A)=0$ (with the convention $\infty\cdot 0=0$).
\end{proof}

Recall that if the action of $G$ is continuous, then a point $x\in X$ is said to be \emph{recurrent} if for any neighbourhood $U$ of $x$, the set $\{g\in G: gx\in U\}$ is not relatively compact; if $G=\R$, then $x$ is call \emph{forward recurrent} (\resp \emph{backward recurrent}) if $\{t>0: \phi_ty\in U\}\subset\R$ (\resp $\{t<0: \phi_ty\in U\}\subset\R$) is unbounded. The following fact is classical and will be used in Section~\ref{Ssection : negligeons les singularites}.

\begin{fait}\label{fait:ergo imp trans}
Assume that $X$ is a locally compact topological space with countable basis, and that the action of $G$ is continuous. 
\begin{enumerate}
 \item \label{item:ergo imp trans:cons} If $\tilde m$ is conservative, then $\tilde m$-almost all points are recurrent;
 \item \label{item:ergo imp trans:consR} if $\tilde m$ is conservative and $G=\R$, then $m$-almost all points are forward and backward recurrent;
 \item \label{item:ergo imp trans:ergo} if $\tilde m$ is ergodic, then $\tilde m$-almost all points have a dense $G$-orbit in $\supp(\tilde m)$;
 \item \label{item:ergo imp trans:ergoR} if $\tilde m$ is ergodic and conservative, and $G=\R$, then $\tilde m$-almost all points have a dense forward orbit and a dense backward orbit in $\supp(\tilde m)$.
\end{enumerate}
\end{fait}

\subsection{Criterions for ergodicity and for mixing}\label{Sous-section : critere ergo et melange}

Inspired by Babillot's proof of mixing \cite{Babillot02}, Coud\`ene stated and proved a criterion for the ergodicity property, and another for the mixing property, both based on the following definition.

\begin{defi}\label{Definition : W-invariant}
 Consider a Borel flow $(\phi_t)_{t\in\R}$ on a metric space $(X,d)$. For each point $x\in X$ we define the \emph{strong stable manifold} 
 \[W^{ss}(x)=\{y\in X: \ d(\phi_tx,\phi_ty)\underset{t\to\infty}{\longrightarrow} 0\}.\]
 We define the \emph{strong unstable manifold} $W^{us}(x)$ to be the strong stable manifold of the time-reversed flow. Consider a $(\phi_t)_{t\in\R}$-invariant $\sigma$-finite measure on $X$. A measurable function $f:X\rightarrow\R$ is said to be \emph{$W^{ss}$-invariant} when there exists a measurable subset $E\subset X$ with full measure such that for all $x$ and $y$ in $E$, if they are on the same strong stable manifold then $f(x)=f(y)$. The notion of \emph{$W^{su}$-invariance} is similarly defined.
\end{defi}

Coud\`ene's criteria are the following. 

\begin{fait}[\!{\!\cite{coudeneHopf}}]\label{Fait : critere d'ergodicite}
 Let $(X,d)$ be a metric space, $(\phi_t)_t$ a measurable flow on it, $m$ a conservative $(\phi_t)_t$-invariant measure, and assume that some $m$-full subset of $X$ is covered by a countable family of open sets with finite $m$-measure. If every $W^{ss}$, $W^{su}$ and $(\phi_t)_t$-invariant measurable function is essentially constant then the flow is ergodic.
\end{fait}

\begin{fait}[\!{\!\cite{coudenemixing}}]\label{coudene}
 Consider a Borel flow preserving a finite measure on a metric space. If every $W^{ss}$- and $W^{su}$-invariant Borel function is essentially constant then the flow is mixing.
\end{fait}

\subsection{Topological entropy}\label{rappels sur l'entropie topologique}

In this section we recall the definition of topological entropy.

\begin{defi}\label{Topological Entropy}
 Let $\phi=(\phi_t)_{t\in\R}$ be a continuous flow on a compact metric space $(X,d)$.
 \begin{itemize}
  \item Let $\epsilon>0$. A subset $S\subset X$ is \emph{$(d,\epsilon)$-separated} if $d(s,s')\geq \epsilon$ for all $s\neq s'$ in $S$. We denote by $N(d,\epsilon)$ the maximal cardinality of such a set $S$.
  \item Let $\epsilon>0$. A subset $S\subset X$ is \emph{$(d,\epsilon)$-spanning} if for any $x\in X$, there exists $s\in S$ with $d(x,s)<\epsilon$. We denote by $S(d,\epsilon)$ the minimal cardinality of such a set $S$.
  \item We take the classical notation $d^{(t)}(x,y):=\max_{0\leq s\leq t}d(\phi_sx,\phi_sy)$ for $t\geq 0$ and $x,y\in X$; this defines a family of metrics on $X$.
  \item The \emph{topological entropy} of $\phi$ on $X$ is:
  \[h_{\topp}(\phi):=\lim_{\epsilon\to0}\limsup_{t\to\infty}\frac{\log N(d^{(t)},\epsilon)}{t}=\lim_{\epsilon\to0}\limsup_{t\to\infty}\frac{\log S(d^{(t)},\epsilon)}{t}.\]
 \end{itemize}
\end{defi}

One defines similarly the topological entropy of a homeomorphism by replacing $t\in\R_{\geq 0}$ by $n\in\Z_{\geq0}$ in Definition~\ref{Topological Entropy}. Note that if $\phi=(\phi_t)_{t\in\R}$ is a continuous flow on a compact metric space $(X,d)$, then the topological entropies of the flow and of the underlying time-one map are the same. Also note that the the reparametrised flow $(\phi_{kt})_{t\in\R}$, where $k\in\R_{\neq0}$, has topological entropy equal to $|k|$ times the topological entropy of $(\phi_t)_{t\in\R}$. 

\begin{rqq}\label{Remarque : distdyn}
 Let $\Omega\subset\PR(V)$ be a properly convex open set and $\Gamma\subset\Aut(\Omega)$ a discrete subgroup; set $M=\Omega/\Gamma$. Then by definition, for any $t\geq 0$, for all $v,w\in T^1M$ and all possible lifts $\tilde{v},\tilde{w}\in T^1\Omega$,
 \begin{align*}
  d_{T^1M}^{(t)}(v,w) = \max_{0\leq s\leq t+1}d_M(\pi\phi_sv,\pi\phi_sw) \leq d_{T^1\Omega}^{(t)}(\tilde{v},\tilde{w}) = \max_{0\leq s\leq t+1}d_\Omega(\pi\phi_s\tilde{v},\pi\phi_s\tilde{w}).
 \end{align*}
 Yet, we will sometimes prefer to work with dynamical balls in the universal cover, and this is why we introduce another notation (which is a bit sloppy unfortunately). For $t\geq 0$ and $v,w\in T^1M$,
 \begin{align*}
  \tilde{d}^{(t)}_{T^1M}(v,w) := \min\{d_{T^1\Omega}^{(t)}(\tilde{v},\tilde{w}) : \tilde{v},\tilde{w}\in T^1\Omega \text{ lifts of }v,w\},
 \end{align*}
 and we denote by $\tilde{B}_{T^1M}^{(t)}(v,R)$ the open ball of radius $R\geq 0$ and centred at $v\in T^1M$ for the metric $\tilde{d}_{T^1M}^{(t)}$. Observe that $\tilde{d}^{(t)}_{T^1M}\geq d_{T^1M}^{(t)}$ and that we do not have equality in general; however we have equality on $\{d_{T^1M}^{(t)}< \inj(M)/2\}$, where $\inj(M)$ is the injectivity radius of $M$, namely
 \begin{equation}\label{Equation : rayon d'injectivite}
 \inj(M):=\inf\{d_\Omega(x,\gamma x) : x\in\Omega,\gamma\in\Gamma\smallsetminus\{\id\}\},
 \end{equation}
\end{rqq}

To conclude this section, we give a relation between the critical exponent, the volume entropy and the topological entropy, which is originally due to Manning \cite{ManningTopEntropy}. Our setting is not exactly the same as his, but the proof works perfectly thanks to Fact~\ref{crampon} and Selberg's Lemma.

\begin{fait}\label{manning}
Let $\Omega\subset\PR(V)$ be a properly convex open set and $\Gamma\subset\Aut(\Omega)$ a convex cocompact, strongly irreducible and discrete subgroup. Let $\Vol$ be a $\Gamma$-invariant Radon measure on $\mathcal{C}^{\core}_\Omega(\Gamma)$. Then 
\[\delta_\Gamma = \delta_{\Vol} \geq  h_{\topp}(T^1M_{\core},(\phi_t)_{t\in\R}) \geq h_{\topp}(T^1M_{\bip},(\phi_t)_{t\in\R}).\]
\end{fait}

\begin{proof}
 The fact that $h_{\topp}(T^1M_{\core},(\phi_t)_{t\in\R}) \geq h_{\topp}(T^1M_{\bip},(\phi_t)_{t\in\R})$ is immediate from the definition of topological entropy.
 
 Let us prove that $\delta_\Gamma=\delta_{\Vol}$. Consider $o\in\mathcal{C}^{\core}$ and $R>0$ large enough such that $\Gamma\cdot B_\Omega(o,R)$ contains $\mathcal{C}^{\core}$ (which is acted on cocompactly by $\Gamma$). Then, on the one hand,
 \[\Vol(B_\Omega(o,r))\leq \sum_{\gamma : d_\Omega(o,\gamma o)\leq r+R}\Vol(B_\Omega(\gamma o,R))\leq \Vol(B_\Omega(o,R)) \#\{\gamma : d_\Omega(o,\gamma o)\leq r+R\}\]
 for any $r>0$, hence $\delta_{\Vol}\leq \delta_\Gamma$, and $\Vol(B_\Omega(o,R))>0$.
 
 On the other hand, consider $C=\#\{\gamma : d_\Omega(o,\gamma o)\leq 2R\}$, so that for any $r>0$, one can find a subset $A$ of $\{\gamma : d_\Omega(o,\gamma o)\leq r\}$ with size greater than $C^{-1}\#\{\gamma : d_\Omega(o,\gamma o)\leq r\} -1$ and whose image by the orbital map $\gamma\mapsto\gamma o$ is $(2R,d_\Omega)$-separated; then
 \begin{align*}
  C^{-1}\#\{\gamma : d_\Omega(o,\gamma o)\leq r\} -1 & \leq \Vol(B_\Omega(o,R))^{-1}\sum_{\gamma \in A}\Vol(B_\Omega(\gamma o,R))\\
  & \leq \Vol(B_\Omega(o,R))^{-1}\Vol(B_\Omega(o,r+R)), 
 \end{align*}
 therefore $\delta_\Gamma\leq\delta_{\Vol}$.
 
 Let us prove that $\delta_\Gamma\geq  h_{\topp}(T^1M_{\core},(\phi_t)_{t\in\R})$. Fix $\epsilon>0$, and fix a maximal $(\epsilon/4,d_\Omega)$-spanning set $A$ of $B_\Omega(o,R)$. Let $B_r$ be the set of vectors based at a point of $A$ and pointing at a point of $\{\gamma a : a\in A, \ d_\Omega(o,\gamma o)\leq r+2R\}$; for each $\tilde{v}\in B_r$ choose $v'\in T^1M_{\core}$ at distance less than $\epsilon/2$ from the projection in $T^1M$ of $\tilde{v}$ (when such a vector exists), and denote by $B_r'$ the collection of $v'$. By Fact~\ref{crampon}, $B_r'$ is a $(\epsilon,d_{T^1M}^{(r)})$-spanning set of $T^1M_{\core}$ for any $r>0$. Furthermore $\#B_r'\leq\#B_r\leq \#A^2\cdot\#\{\gamma : d_\Omega(o,\gamma o)\leq r+2R\}$, hence $\delta_\Gamma\geq  h_{\topp}(T^1M_{\core},(\phi_t)_{t\in\R})$.
\end{proof}

We will see in Proposition~\ref{l'entropie de BM est max} that $h_{\topp}(T^1M_{\core},(\phi_t)_{t\in\R})=\delta_\Gamma$.

\subsection{Measure-theoretic entropy}\label{rappels sur l'entropie mesurable}

\begin{defi}
 Let $f$ be an inversible measurable map of a measurable space $X$, which preserves a probability measure $\mu$. Let $\p$ be a finite measurable partition of $X$.
 \begin{itemize}
  \item The \emph{entropy} of $\p$ is
  \[H_\mu(\p):=-\sum_{P\in\p}\mu(P)\log\mu(P).\]
  \item For each integer $n\geq 1$, we set
  \[\p^{(n)}:=\{P_0\cap f^{-1}P_1\cap\dots\cap f^{-n+1}P_{n-1} : P_k\in \p,\ 0\leq k\leq n-1\}.\]
  \item The entropy of $f$ with respect to $\mu$ and $\p$ is
  \[H_\mu(f,\p):=\lim_{n\to\infty}\frac{H_\mu(\p^{(n)})}{n}.\]
  \item The entropy of $f$ with respect to $\mu$ is
  \[h_\mu(f):=\sup\{H_\mu(f,\mathcal{Q}) : \mathcal{Q} \text{ measurable partition}\}.\]
 \end{itemize}
 Let us now consider a measurable flow $(\phi_t)_{t\in\R}$ on a measurable space $X$ preserving a probability measure $\mu$. The entropy of the flow is defined to be the entropy of the time-one map $\phi_1$.
\end{defi}

\begin{rqq}\label{reparam_entropy}
 The reparametrised  flow $(\phi_{kt})_{t\in\R}$ has measure-theoretic entropy equal to $|k|$ times the measure-theoretic entropy of $(\phi_t)_{t\in\R}$.
\end{rqq}

The relation between topological and measure-theoretic entropies is given by the following famous principle, for more details see \cite[Th.\,4.5.3]{katok}.

\begin{fait}[Variational Principle]\label{le principe variationnel}
 Let $\phi=(\phi_t)_{t\in\R}$ a continuous flow on a compact metric space $(X,d)$. Denote by $\p^{\phi}(X)$ the set of $\phi$-invariant probability measures on $X$. Then
 \[h_{\topp}(\phi)=\sup_{\mu\in\p^{\phi}(X)}h_\mu(\phi).\]
\end{fait}

We recall the definition of entropy-expansive maps from Bowen \cite{entropy-expansive}.

\begin{defi}
 Consider a homeomorphism $f$ of a metric space $(X,d)$ and $\epsilon>0$. The map $f$ is said to be \emph{$(d,\epsilon)$-entropy-expansive} if for each $x\in X$, the action of $f$ on the \emph{Bowen ball}
 \[Z_{\epsilon}(x):=\{y\in X: \forall n\in \Z, d(f^nx,f^ny)\leq\epsilon\}\]
 has zero entropy.
\end{defi}

\begin{rqq}\label{f^n_is_expansive}
 If $f$ is $(d,\epsilon)$-entropy-expansive then $f^n$ is $(d^{(n)},\epsilon)$-entropy-expansive.
\end{rqq}

The following is one of the main property of entropy-expansive maps.

\begin{fait}[\!{\!\cite[Th.\,3.5]{entropy-expansive}}]\label{bowen}
 Let $\epsilon$ be a positive number; let $f$ be a $\epsilon$-entropy-expansive homeomorphism of a metric space $(X,d)$; let $\mu$ be a $f$-invariant probability measure on $X$; let $\p$ be a finite measurable partition all of whose elements have diameter less than $\epsilon$. Then 
 \[h_{\mu}(f)=H_\mu(f,\p).\]
\end{fait}

\begin{rqq}\label{l'entropie est semi-continue}
 Let $\epsilon>0$; let $f$ be an $\epsilon$-entropy-expansive homeomorphism of a metric space $(X,d)$. Fact~\ref{bowen} tells us that the measure-theoretic entropy depends upper semi-continuously on the measure (with respect to the weak* topology), and that there exists a measure of maximal entropy if $X$ is compact.
\end{rqq}

The geodesic flow on a convex projective manifold is entropy-expansive if the injectivity radius is non-zero.
\begin{fait}[\!{\!\cite[Th.\,6.2]{bray_top_mixing}}]\label{geod_is_expansive}
 Let $\Omega\subset\PR(V)$ be a properly convex open set and $\Gamma\subset\Aut(\Omega)$ a discrete subgroup. Let us assume that the injectivity radius of $M=\Omega/\Gamma$ is non-zero. Then the time-one map of the geodesic flow on $T^1M$ is $(d_{T^1M},\epsilon_0/3)$-entropy-expansive.
\end{fait}
It applies in particular if $\Gamma$ acts convex cocompactly on $\Omega$ and \emph{if $\Gamma$ is torsion-free}. However, we will consider cases where $\Gamma$ is not torsion-free, and the geodesic flow on $T^1M$ is \emph{not} entropy-expansive: for instance if $\Omega$ is the Poincar\'e disk and $\Gamma$ is a cocompact triangle group. To overcome this issue, we will use Selberg's lemma (if $\Gamma$ acts convex cocompactly on $\Omega$, then it is finitely generated --- see \cite[Th.\,8.10]{bridsonhaeflige}) and the following elementary observation.

\begin{obs}\label{obs:revetentrop}
 Let $(\phi_t)_t$ be a measurable flow on a measurable space $X$ that preserves a probability measure $\mu$, and $G$ a finite group that acts measurably on $X$ and commutes with $(\phi_t)_t$. We denote by $\pi:X\rightarrow X/G$ the natural projection, and $(\phi_t)_t$ the induced flow on the quotient. Then $h_\mu(X,(\phi_t)_t)=h_{\pi_*\mu}(X/G,(\phi_t)_t)$.
 In particular, if $X$ is a compact topological space and the actions of $(\phi_t)_t$ and $G$ are continuous, then $h_{\topp}(X,(\phi_t)_t)=h_{\topp}(X/G,(\phi_t)_t)$.
\end{obs}

\section{Construction of the Sullivan measures}\label{Construction of the Sullivan measures}

\subsection{The Gromov product}\label{Section : le produit de Gromov}

Let us define the Gromov product. It will be used to define the Sullivan measures. Recall that given a properly convex open set $\Omega\subset\PR(V)$ with three points $x,\xi,\eta\in\Omega$, the Gromov product is defined by 
\begin{equation}
 2\langle \xi,\eta \rangle_x:=d_\Omega(x,\xi)+d_\Omega(x,\eta)-d_\Omega(\xi,\eta)\geq0.
\end{equation}
The following are three immediate properties of the Gromov product. Let $y\in\Omega$.
\begin{gather}
 x\in[\xi,\eta]  \Rightarrow \langle \xi,\eta \rangle_x = 0; \label{Equation : Gromov 1} \\
 2\langle \xi,\eta \rangle_x = 2\langle \xi,\eta \rangle_y + b_\xi(x,y)+b_\eta(x,y); \label{Equation : Gromov 2} \\
 |\langle \xi,\eta \rangle_x - \langle \xi,\eta \rangle_y | \leq d_\Omega(x,y). \label{Equation : Gromov 3}
\end{gather}

We denote by $\Geod_{\hor}(\Omega)$ (\resp $\Geod^{\infty}_{\hor}(\Omega)$) the set of pairs $(\xi,\eta)$ in $(\overline{\Omega}^{\hor})^2$ (\resp $\partial_{\hor}\Omega^2$) such that $[\xi,\eta]\cap\Omega\neq\emptyset$, where $[\xi,\eta]:=[\pi_{\hor}(\xi),\pi_{\hor}(\eta)]$.

\begin{prop}\label{prop:gromprod}
 Let $\Omega\subset\PR(V)$ be a properly convex open set. The map $(\xi,\eta,x)\mapsto \langle \xi,\eta \rangle_x$ defined on $\Omega^3$ extends continuously to $\Geod_{\hor}(\Omega)\times\Omega$, as well as \eqref{Equation : Gromov 1}, \eqref{Equation : Gromov 2} and \eqref{Equation : Gromov 3}.
\end{prop}

\begin{proof}
 Let $A:=\{(\xi,\eta,x,y)\in\Geod_{\hor}(\Omega)\times\Omega^2 : y\in[\xi,\eta]\}$. The projection $(\xi,\eta,x,y) \mapsto (\xi,\eta,x)$ from $A$ to $\Geod_{\hor}(\Omega)\times\Omega$ is continuous, surjective and open. Therefore, it is enough to prove that the function $(\xi,\eta,x,y)\mapsto \langle \xi,\eta \rangle_x$ defined on $A\cap\Omega^4$ extends continuously to $A$. This is a consequence of \eqref{Equation : Gromov 1} and \eqref{Equation : Gromov 2} combined, which yield, for any $(\xi,\eta,x,y)\in A\cap\Omega^4$,
 \[\langle \xi,\eta \rangle_x = b_\xi(x,y) + b_\eta(x,y).\]
 
 Note that \eqref{Equation : Gromov 1} extends to $(\xi,\eta)\in\Geod_{\hor}(\Omega)$ because the set $\{(\xi,\eta,x)\in\Omega^3 : x\in[\xi,\eta]\}$ is dense $\{(\xi,\eta,x)\in\Geod_{\hor}(\Omega)\times\Omega : x\in[\xi,\eta]\}$.
\end{proof}

\subsection{The Hopf coordinates}

Let us define the Hopf parametrisation of the unit tangent bundle $T^1\Omega$ of a properly convex open set $\Omega$, which depends on the choice of a fixed basepoint $o\in\Omega$.

\begin{defi}\label{Definition : Hopf}
 Let $\Omega\subset\PR(V)$ be a properly convex open set and fix a basepoint $o\in\Omega$. The Hopf parametrisation based at $o$ is the $(\phi_t)_t$-equivariant continuous surjective map
 \[\Hopf_o\colon \Geod^{\infty}_{\hor}(\Omega)\times\R\longrightarrow T^1\Omega,\]
 that sends $(\xi,\eta,t)\in\Geod^{\infty}_{\hor}(\Omega)\times\R$ to the vector $\Hopf_o(\xi,\eta,t)$ which is tangent to the projective line $\pi_{\hor}\xi\oplus\pi_{\hor}\eta$ and such that $b_{\eta}(o,\pi\Hopf_o(\xi,\eta,t))=t$.
\end{defi}

When the context is clear, we will simply write $\Hopf$ instead of $\Hopf_o$. Note that if $x\in\Omega$ then for any $(\xi,\eta,t)\in\Geod^\infty_{\hor}(\Omega)\times\R$,
\begin{equation}
 \Hopf_x(\xi,\eta,t)=\Hopf_o(\xi,\eta,t+b_\eta(o,x)).
\end{equation}
Therefore, if we consider the following action of $\Aut(\Omega)$:
\begin{equation}
g\cdot(\xi,\eta,t)=(g\xi,g\eta,t+b_\eta(g^{-1}o,o))
\end{equation}
for $g\in\Aut(\Omega)$ and $(\xi,\eta,t)\in\Geod^\infty_{\hor}(\Omega)$, then $\Hopf_o$ is $\Aut(\Omega)$-equivariant.

Note also that for any $(\xi,\eta,t)\in\Geod^\infty_{\hor}(\Omega)$,
\begin{equation}\label{Equation : Hopflip}
 -\Hopf_o(\xi,\eta,t) = \Hopf_o(\eta,\xi, \langle \xi,\eta \rangle_o -t),
\end{equation}
where for any $v\in T^1\Omega$, the vector $-v$ satisfies $\phi_{\pm\infty}(-v)=\phi_{\mp\infty}v$ and $\pi (-v)=\pi v$.

\subsection{The Sullivan measures}\label{The Sullivan measures}

 In this section, given a conformal density, we construct the associated $(\phi_t)_{t\in\R}$-invariant Sullivan measures on $T^1\Omega$ and on the quotient $T^1M=T^1\Omega/\Gamma$.

\begin{defi} 
 Let $\Omega\subset\PR(V)$ be a properly convex open set and $\Gamma\subset\Aut(\Omega)$ a strongly irreducible discrete subgroup; denote $M=\Omega/\Gamma$. Let $\delta\geq0$ and let $(\mu_x)_x$ be a $\delta$-conformal density on $\partial_{\hor}\Omega$. The \emph{Sullivan measure} $\tilde{m}_{\hor}$ on $\partial_{\hor}\Omega^2\times\R$ induced by $(\mu_x)_x$ is defined by the following.
 \[\diff \tilde{m}_{\hor}(\xi,\eta,t)=e^{2\delta\langle \xi,\eta\rangle_o}1_{\Geod^{\infty}_{\hor}(\Omega)}(\xi,\eta)d\mu_o(\xi)d\mu_o(\eta)dt.\]
 where $o\in\Omega$; the measure $\tilde{m}_{\hor}$ does not depend on $o$ and is $(\phi_t)_t$-invariant. Then we push it forward  via the Hopf coordinates to get the induced \emph{Sullivan measure on $T^1\Omega$}:
 \[\mm := (\Hopf_o)_*\mm_{\hor}.\]
 It does not depend on $o$, and it is $\Gamma$ and $(\phi_t)_t$-invariant. Hence it yields an induced $(\phi_t)_t$-invariant \emph{Sullivan measure on $T^1M$}, denoted by $m$.
\end{defi}

Observe that the formula $e^{2\delta\langle \xi,\eta\rangle_o}1_{\Geod^{\infty}_{\hor}(\Omega)}(\xi,\eta)d\mu_o(\xi)d\mu_o(\eta)$ yields a $\Gamma$-invariant measure on $\Geod_{hor}^\infty(\Omega)$, which is in fact the quotient of the Sullivan measure under the action of $\R$.

\begin{prop}\label{prop:Radonnonnul}
 Let $\Omega\subset\PR(V)$ be a properly convex open set and $\Gamma\subset\Aut(\Omega)$ a discrete subgroup. Let $\delta\geq0$ and let $(\mu_x)_x$ be a $\delta$-conformal density on $\partial_{\hor}\Omega$. Then the Sullivan measures on $\Geod_{\hor}(\Omega)\times\R$, $T^1\Omega$ and $T^1M$ are Radon. If $\Gamma$ is strongly irreducible and $T^1M_{\bip}\neq\emptyset$, or if $M$ is rank-one and non-elementary, then these measures are non-zero.
\end{prop}

\begin{proof}
 The fact that the Sullivan measures are Radon is a direct consequence of the continuity of the Gromov product (Proposition~\ref{prop:gromprod}).
%

 Suppose that $\Gamma$ is strongly irreducible and $T^1M_{\bip}\neq\emptyset$, or that $M$ is rank-one and non-elementary. Then there exists $(\xi,\eta)\in\Geod^\infty(\Omega)\cap(\Lambda^{\prox})^2$. Let $U$ and $V$ be neighbourhoods of $\xi$ and $\eta$ in $\partial\Omega$ such that $U\times V\subset\Geod^\infty(\Omega)$. Since $\nu_o=\pi_{hor*}\mu_o$ is $\Gamma$-quasi-invariant, its support is $\Gamma$-invariant and hence contains $\Lambda^{\prox}$ (see Remark~\ref{minimality} and Fact~\ref{fait:des rk1 partout}). Therefore, $\nu_o^2(U\times V)>0$, and $\mm_{\hor}(\pi_{\hor}^{-1}(U)\times\pi_{\hor}^{-1}(V)\times \R)>0$.
\end{proof}

\section{The Shadow lemma}\label{The Shadow lemma}

In this section we establish the Shadow lemma (Lemma~\ref{shadowlemma}) which consists of estimates on the measures of shadows. The measure is a conformal density on $\partial\Omega$, and shadows are subsets of the projective boundary, defined as follows. Recall that the Shadow lemma is a classical result in the theory of conformal densities, and we adapt here its classical proof to the convex projective setting.

\begin{defi}\label{Definition : ombres}
 Let $\Omega\subset \PR(V)$ be a properly convex open set. Take $x,y\in\Omega$ and $r>0$. Set
 \[\mathcal{O}_r(x,y):=\{\xi\in\partial\Omega : [x,\xi]\cap B_\Omega(y,r)\neq\emptyset\};\]
 \[\mathcal{O}_r^+(x,y):=\{\xi\in\partial\Omega : \exists z\in B_\Omega(x,r) \text{ such that } [z,\xi]\cap B_\Omega(y,r)\neq\emptyset\};\]
 \[\mathcal{O}_r^-(x,y):=\{\xi\in\partial\Omega : \forall z\in B_\Omega(x,r), [z,\xi]\cap B_\Omega(y,r)\neq\emptyset\}.\]
\end{defi}

See in Figure~\ref{fig:lemombre} an example of a shadow.

\begin{lemma}\label{shadowlemma}
Let $o\in\Omega\subset\PR(V)$ be a pointed properly convex open set and $\Gamma\subset\Aut(\Omega)$ a discrete subgroup; set $M=\Omega/\Gamma$. Suppose that $\Gamma$ is strongly irreducible and $T^1M_{\bip}$ is non-empty, or that $M$ is rank-one and non-elementary. Consider $\delta\geq0$ and a $\delta$-conformal density $(\nu_x)_{x\in\Omega}$ on $\partial\Omega$. Then there exists $R_0>0$ such that for any $R\geq R_0$, one can find $C=C(R)>0$ such that for each $\gamma\in\Gamma$,
 \[C^{-1}e^{-\delta d_\Omega(o,\gamma o)}\leq \nu_o(\mathcal{O}_R(o,\gamma o))\leq \nu_o(\mathcal{O}_R^+(o,\gamma o))\leq Ce^{-\delta d_\Omega(o,\gamma o)}.\]
\end{lemma}

We will actually need two more Shadow lemmas: Lemma~\ref{shadowlemma avec defaut} and Corollary~\ref{shadowlemma+}. They both make stronger assumptions on the convex projective manifold $M=\Omega/\Gamma$, and either are consequences of Lemma~\ref{shadowlemma}, or have a very similar proof.

\subsection{Preliminaries}

In this section we prove two classical intermediate lemmas, used in the proof of the Shadow lemma.

\begin{lemma}\label{b=d}
 Let $\Omega\subset\PR(V)$ be a properly convex open set. Let $\xi\in\overline{\Omega}^{\hor}$ and $x,y\in\Omega$. If $y\in [x,\xi]$, then $b_{\xi}(x,y)=d_\Omega(x,y)$. Let $r>0$. If $\pi_{\hor}(\xi)\in \mathcal{O}_r^+(x,y)$ (see Definition~\ref{Definition : ombres}), then
 \[d_\Omega(x,y)-4r\leq b_{\xi}(x,y)\leq d_\Omega(x,y).\]
\end{lemma}

\begin{proof}
 One easily see that $b_{\xi}(x,y) - d_\Omega(x,y) = -2\langle \xi,x \rangle_y$ if $\xi\in\Omega$, and this extends to $\xi\in\overline{\Omega}^{\hor}$ by continuity. This, by \eqref{Equation : Gromov 1}, implies that $b_\xi(x,y)=d_\Omega(x,y)$ if $y\in[x,\xi]$. Assume that $\xi\in\mathcal{O}^+_r(x,y)$. The triangular inequality gives $b_\xi(x,y)\leq d_\Omega(x,y)$; let us establish $b_\xi(x,y)\geq d_\Omega(x,y)-4r$. By definition of $\mathcal{O}^+_r(x,y)$, we can find $x'\in B_\Omega(x,r)$ and $y'\in B_\Omega(y,r)\cap [x',\xi]$. Then
 \begin{align*}
  b_{\xi}(x,y)&=b_{\xi}(x,x')+b_{\xi}(x',y')+b_{\xi}(y',y)\\
  &\geq - d_\Omega(x,x')+d_\Omega(x',y') - d_\Omega(y',y)\\
  &\geq d_\Omega(x,y)-2d_\Omega(x,x')-2d_\Omega(y,y')\\
  &\geq d_\Omega(x,y)-4r.\qedhere
 \end{align*}
\end{proof}

\begin{lemma}\label{l ombre grandit a la frontiere}
Let $\Omega\subset\PR(V)$ be a properly convex open set and $\Gamma\subset\Aut(\Omega)$ a discrete subgroup; set $M=\Omega/\Gamma$. Suppose that $\Gamma$ is strongly irreducible and $T^1M_{\bip}$ is non-empty, or that $M$ is rank-one and non-elementary. Consider a $\Gamma$-quasi-invariant finite Borel measure $\nu$ on $\partial\Omega$. Then
there is $\epsilon>0$ and $R>0$ such that for all $x\in\Omega$, 
 \[\nu(\mathcal{O}_R(x,o))\geq \epsilon.\]
\end{lemma}

\begin{proof}
 By contradiction suppose that there are sequences $(R_n)_n\in\R_{>0}^\N$ and $(x_n)_n\in\Omega$ such that 
 \[R_n\underset{n\to\infty}{\longrightarrow}\infty \text{ and } \nu(\mathcal{O}_{R_n}(x_n,o))\underset{n\to\infty}{\longrightarrow}0.\]
 We can assume that $(x_n)_n$ converges to some $\xi\in\overline{\Omega}$. If $\xi\in\Omega$ then for $n$ such that ${R_n}\geq d_\Omega(o,\xi)+1$ and $d_\Omega(x_n,\xi)<1$, we have $\mathcal{O}_{R_n}(x_n,o)=\partial\Omega$ which is absurd; hence $\xi\in\partial\Omega$. We claim that $\nu(\overline{B}_{\mathrm{spl}}(\xi,1))=1$. It enough to prove that 
 \[\partial\Omega\setminus \overline{B}_{\mathrm{spl}}(\xi,1)\subset \bigcup_n\bigcap_{k\geq n}\mathcal{O}_{R_k}(x_k,o).\]
 See Figure~\ref{fig:lemombre}. Let $\eta\in\partial\Omega\setminus \overline{B}_{\mathrm{spl}}(\xi,1)$. Fixing an affine chart containing $\overline{\Omega}$, we can consider for each $n$ the following points of $\Omega$:
\[y:=\frac{1}{2}(\xi-\eta)+\eta \text{ and } y_n:=\frac{1}{2}(x_n-\eta)+\eta,\]
where the map $x\mapsto (x-\eta)/2+\eta$ is defined on the affine chart as the homothety centred at $\eta$ and with ratio one half. Then we can find $n$ large enough so that $B_\Omega(o,{R_n})$ contains a neighbourhood $U$ of $y$, and so that $y_n$ is contained in $U$. This implies that $\eta\in \mathcal{O}_{R_n}(x_n,o)$, which conludes the proof of the claim that $\nu(\overline{B}_{\mathrm{spl}}(\xi,1))=1$.
 
 Since $\Lambda^{\prox}$ is the smallest $\Gamma$-invariant closed subset of $\overline\Omega$ (Remark~\ref{minimality} and Fact~\ref{fait:des rk1 partout}), and $\nu$ is $\Gamma$-quasi-invariant, we deduce that $\Lambda^{\prox}$ is contained in the support of $\nu$. In order to get a contradiction, let us prove that $\Lambda^{\prox}$ is not contained in $\overline{B}_{\mathrm{spl}}(\xi,1)$. By assumption we can find $\eta,\eta'\in\Lambda^{\prox}$ such that $(\eta\oplus\eta')\cap\Omega\neq\emptyset$. Using again that $\Lambda^{\prox}\subset\overline\Omega$ is minimal for the action of $\Gamma$, we deduce the existence of a sequence $(\gamma_n)_n\in\Gamma^{\N}$ such that $(\gamma_n\xi)_n$ converges to $\eta$. But then $(\overline{B}_{\mathrm{spl}}(\gamma_n\xi,1))_n$ sub-converges to $\overline{B}_{\mathrm{spl}}(\eta,1)$ (\ie any accumulation point for the Hausdorff topology is contained in $\overline{B}_{\mathrm{spl}}(\eta,1)$); hence $\overline{B}_{\mathrm{spl}}(\gamma_n\xi,1)$ does not contain $\eta'$ for $n$ large enough.
 
\end{proof}

 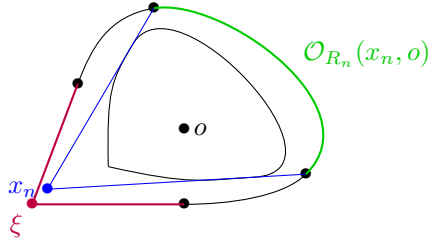
\begin{figure}
  \centering
  \begin{tikzpicture}[scale=2]
   \coordinate (xi) at (0,0);
   \coordinate (a) at (1,0);
   \coordinate (b) at (1.8,.2);
   \coordinate (c) at (.8,1.3);
   \coordinate (d) at (.3,.8);
   \coordinate (x) at (.1,.1);
   \coordinate (o) at (1,.5);
   \coordinate (y) at ($(x)!.8!(b)$);
   \coordinate (z) at ($(x)!.8!(c)$);
   \coordinate (xi') at ($(xi)!.5!(o)$);
   \coordinate (Ta) at ($.3*(a)-.3*(xi)$);
   \coordinate (Tb) at (.3,.3);
   \coordinate (Tc) at (-.3,-.05);
   \coordinate (Td) at ($.3*(xi)-.3*(d)$);
   \coordinate (Ty) at ($(b)-(y)$);
   \coordinate (Tz) at ($(z)-(c)$);
   \coordinate (Txi') at ($.5*(a)-.5*(xi')+(Ta)$);
   \coordinate (Tmxi') at ($.5*(xi')-.5*(d)+(Td)$);
   
   \draw (xi) node[purple]{$\bullet$} node[below left,purple]{$\xi$};
   \draw (a) node{$\bullet$};
   \draw (b) node{$\bullet$};
   \draw (c) node{$\bullet$};
   \draw (d) node{$\bullet$};
   \draw (o) node{$\bullet$} node[right]{$o$};
   \draw (x) node[blue]{$\bullet$} node[left,blue]{$x_n$};
   
   \draw [thick, purple] (xi)--(a);
   \draw (a) .. controls ($(a)+(Ta)$) and ($(b)-.5*(Tb)$) .. (b);
   \draw [thick, green!80!black] (b).. controls ($(b)+1.5*(Tb)$) and ($(c)-1.5*(Tc)$) ..(c);
   \draw (c).. controls ($(c)+(Tc)$) and ($(d)-(Td)$) ..(d);
   \draw [thick, purple] (d)--(xi);
   \draw (y)..controls ($(y)+2*(Ty)$) and ($(z)-2*(Tz)$)..(z);
   \draw (z)..controls ($(z)+(Tz)$) and ($(xi')-(Tmxi')$)..(xi');
   \draw (xi')..controls ($(xi')+(Txi')$) and ($(y)-(Ty)$)..(y);
   \draw [blue] (x)--(b);
   \draw [blue] (x)--(c);
   
   \draw (2.2,1) node[green!80!black]{$\mathcal{O}_{R_n}(x_n,o)$};
  \end{tikzpicture}
  \caption{The sequence of increasing shadows in the proof of Lemma~\ref{l ombre grandit a la frontiere}}
  \label{fig:lemombre}
 \end{figure}

We will need exactly twice a refined version of the previous lemma, which bounds from below the measure of \emph{scarce shadows}. It will be used to prove the refined version of the Shadow lemma for scarce shadows (Lemma~\ref{shadowlemma avec defaut}), and to prove Proposition~\ref{l'ensemble limite conique a mesure non nulle}. This version needs $M$ to be rank-one, and the statement is a bit more technical. The proof is almost the same.

\begin{lemma}\label{l'ombre en defaut grandit a la frontiere}
 Let $\Omega\subset\PR(V)$ be a properly convex open set and $\Gamma\subset\Aut(\Omega)$ a discrete subgroup. Suppose $M=\Omega/\Gamma$ is rank-one and non-elementary. Consider a $\Gamma$-quasi-invariant Borel finite measure $\nu$ on $\partial\Omega$. Then
there exist $\epsilon>0$ and $R_0>0$ such that $\nu(\mathcal{O}_R^-(x,o))\geq \epsilon$ for all $R\geq R_0$ and $x\in\Omega\smallsetminus B_\Omega(o,2R)$.
\end{lemma}

\begin{proof}
 By contradiction we suppose the existence of sequences $(R_n)_n\in\R_{>0}^{\N}$ and $(x_n)_n\in\Omega^{\N}$ such that for all $n\in\N$, 
\[R_n\underset{n\to\infty}{\longrightarrow}\infty, \text{ while } d_\Omega(x_n,o)-R_n\underset{n\to\infty}{\longrightarrow}\infty \text{ and } \nu(\mathcal{O}_{R_n}^-(x_n,o))\underset{n\to\infty}{\longrightarrow}0.\]
 We can assume, up to extracting, that $(x_n)_n$ converges to some point $\xi\in\partial\Omega$. We claim that $\nu(\overline{B}_{\mathrm{spl}}(\xi,2))=1$ (see Definition~\ref{les metriques du bord}). It is enough to prove that 
 \[\partial\Omega\setminus \overline{B}_{\mathrm{spl}}(\xi,2)\subset \bigcup_n\bigcap_{k\geq n}\mathcal{O}_{R_k}^-(x_k,o).\]
 Let $\eta\in\partial\Omega\setminus \overline{B}_{\mathrm{spl}}(\xi,2)$. Fixing an affine chart containing $\overline{\Omega}$, we can consider for each $n$ the following compact subsets of $\Omega$:
\[K:=\frac{1}{2}(\overline{B}_{\mathrm{spl}}(\xi,1)-\eta)+\eta \text{ and } K_n:=\frac{1}{2}(\overline{B}_\Omega(x_n,{R_n})-\eta)+\eta,\]
where the map $x\mapsto (x-\eta)/2+\eta$ is defined on the affine chart as the homothety centred at $\eta$ and with ratio one half. We observe that all accumulation points of the sequence $(K_n)_n$ are contained in $K$. Indeed, any accumulation point of $(\overline{B}_\Omega(x_n,{R_n}))_n$ is convex, contains $\xi$ and is moreover contained in $\partial\Omega$ since $x_n$ goes faster to infinity than $R_n$; hence it is contained in $\overline{B}_{\mathrm{spl}}(\xi,1)$. Therefore we can find $n$ large enough so that $B_\Omega(o,{R_n})$ contains a neighbourhood $U$ of $K$, and so that $K_n$ is contained in $U$.
 
 Since $\Lambda^{\prox}$ is the smallest $\Gamma$-invariant closed subset of $\overline\Omega$ (Fact~\ref{fait:des rk1 partout}), and $\nu$ is $\Gamma$-quasi-invariant, we deduce that $\Lambda^{\prox}$ is contained in the support of $\nu$. In order to get a contradiction, let us prove that $\Lambda^{\prox}$ is not contained in $\overline{B}_{\mathrm{spl}}(\xi,2)$. By assumption we can find two distinct points $\eta,\eta'$ in $\Lambda^{\prox}\cap\partial_{\sse}\Omega$. Using again that $\Lambda^{\prox}\subset\overline\Omega$ is minimal for the action of $\Gamma$, we deduce the existence of a sequence $(\gamma_n)_n\in\Gamma^{\N}$ such that $(\gamma_n\xi)_n$ converges to $\eta$. But then $(\overline{B}_{\mathrm{spl}}(\gamma_n\xi,2))_n$ sub-converges to $\overline{B}_{\mathrm{spl}}(\eta,2)=\{\eta\}$; hence $\overline{B}_{\mathrm{spl}}(\gamma_n\xi,2)$ does not contain $\eta'$ for $n$ large enough.
\end{proof}

\subsection{Proof of Lemma~\ref{shadowlemma} and another Shadow lemma}

\begin{proof}[Proof of Lemma~\ref{shadowlemma}]
We compute for $\alpha\in\{\emptyset,+,-\}$:
 \begin{align*}
  \nu_o(\mathcal{O}_R^\alpha(o,\gamma o)) &= \mu_o(\pi_{\hor}^{-1}(\mathcal{O}_R^\alpha(o,\gamma o)))\\
  &= \mu_{\gamma^{-1}o}(\pi_{\hor}^{-1}(\mathcal{O}_R^\alpha(\gamma^{-1}o,o)))\\
  &= \int_{\pi_{\hor}^{-1}(\mathcal{O}_R^\alpha(\gamma^{-1}o,o))}e^{-\delta b_{\xxi}(\gamma^{-1}o,o)}d\mu_o(\xxi).
 \end{align*}
 
 Hence on one hand $b_{\xxi}(\gamma^{-1}o,o)\geq d_\Omega(o,\gamma^{-1}o)-4R$  because of Lemma~\ref{b=d}, so 
 \begin{align*}
  \nu_o(\mathcal{O}_R^+(o,\gamma o))&\leq \int_{\pi_{\hor}^{-1}(\mathcal{O}_R^+(\gamma^{-1}o,o))}e^{-\delta d_\Omega(o,\gamma o)+4\delta R}d\mu_o(\xxi)\\
  &\leq e^{4\delta R} e^{-\delta d_\Omega(o,\gamma o)}\nu_o(\mathcal{O}_R^+(\gamma^{-1}o,o))\\
  &\leq e^{4\delta R} e^{-\delta d_\Omega(o,\gamma o)}\nu_o(\partial\Omega).
 \end{align*}

 On the other hand, we can use Lemma~\ref{l ombre grandit a la frontiere}, to obtain $\epsilon>0$ and $R_0$ such that for $R\geq R_0$ and for $\gamma\in\Gamma$ such that $d_\Omega(o,\gamma o)\geq R_0$,
 \begin{align*}
  \nu_o(\mathcal{O}_R(o,\gamma o))&\geq \int_{\pi_{\hor}^{-1}(\mathcal{O}_R(\gamma^{-1}o,o))}e^{-\delta d_\Omega(o,\gamma o)}d\mu_o(\xxi)\\
  &\geq \nu_o(\mathcal{O}_R(\gamma^{-1}o,o)) e^{-\delta d_\Omega(o,\gamma o)}\\
  &\geq \epsilon e^{-\delta d_\Omega(o,\gamma o)}.\qedhere
 \end{align*}
\end{proof}

As for Lemma~\ref{l ombre grandit a la frontiere}, there exists a refined version of the Shadow lemma (Lemma~\ref{shadowlemma}) for scarce shadows. We will only need it once, for the proof of Proposition~\ref{l'ensemble limite conique a mesure non nulle}. Its proof is exactly the same as that of Lemma~\ref{shadowlemma}, except that we use instead Lemma~\ref{l ombre grandit a la frontiere}.

\begin{lemma}\label{shadowlemma avec defaut}
Let $o\in\Omega\subset\PR(V)$ be a pointed properly convex open set and $\Gamma\subset\Aut(\Omega)$ a discrete subgroup. Suppose $M=\Omega/\Gamma$ is rank-one and non-elementary. Consider $\delta\geq0$ and a $\delta$-conformal density $(\nu_x)_{x\in\Omega}$ on $\partial\Omega$. Then there exists $R_0>0$ such that for any $R\geq R_0$, one can find $C=C(R)>0$ such that for each $\gamma\in\Gamma$ satisfying $d_\Omega(o,\gamma o)\geq 2R$,
 \[C^{-1}e^{-\delta d_\Omega(o,\gamma o)}\leq \nu_o(\mathcal{O}_R^-(o,\gamma o)).\]
\end{lemma}

\begin{proof}
We make the same computation as in the proof of Lemma~\ref{shadowlemma}:
 \begin{align*}
  \nu_o(\mathcal{O}_R^-(o,\gamma o)) = \int_{\pi_{\hor}^{-1}(\mathcal{O}_R^-(\gamma^{-1}o,o))}e^{-\delta b_{\xxi}(\gamma^{-1}o,o)}d\mu_o(\xxi).
 \end{align*}

 We can use Lemma~\ref{l'ombre en defaut grandit a la frontiere}, to obtain $\epsilon>0$ and $R_0$ such that for $R\geq R_0$ and for $\gamma\in\Gamma$ such that $d_\Omega(o,\gamma o)\geq 2R$,
 \begin{align*}
  \nu_o(\mathcal{O}_R^-(o,\gamma o))&\geq \int_{\pi_{\hor}^{-1}(\mathcal{O}_R^-(\gamma^{-1}o,o))}e^{-\delta d_\Omega(o,\gamma o)}d\mu_o(\xxi)\\
  &\geq \nu_o(\mathcal{O}_R^-(\gamma^{-1}o,o)) e^{-\delta d_\Omega(o,\gamma o)}\\
  &\geq \epsilon e^{-\delta d_\Omega(o,\gamma o)}.\qedhere
 \end{align*}
\end{proof}

\subsection{First consequences}

In this section we deduce from the Shadow lemma that there is no conformal density with parameter $0\leq\delta<\delta_\Gamma$. We also prove that open faces of the conical limit set are neglected by conformal densities.

\begin{prop}\label{delta<deltaGamma}
 Let $o\in\Omega\subset\PR(V)$ be a pointed properly convex open set and $\Gamma\subset\Aut(\Omega)$ a discrete subgroup; denote $M=\Omega/\Gamma$. Suppose that $\Gamma$ is strongly irreducible and $T^1M_{\bip}$ is non-empty, or that $M$ is rank-one and non-elementary. Consider $\delta\geq0$ such that there exists a $\delta$-conformal density on $\partial\Omega$. Then $\delta\geq\delta_\Gamma$, and there is some constant $C>0$ such that 
 \[\#\{\gamma\in\Gamma : d_\Omega(o,\gamma o)\leq r\}\leq C e^{\delta_\Gamma r}.\]
\end{prop}

\begin{proof}
 Let $(\nu_x)_{x\in\Omega}$ be a $\delta$-conformal density on $\partial\Omega$. We consider $R$ and $C>0$ from the Shadow lemma (Lemma~\ref{shadowlemma}), such that for each automorphism $\gamma\in\Gamma$, $\nu_o(\mathcal{O}_R(o,\gamma o))\geq C^{-1}e^{-\delta d_\Omega(o,\gamma o)}$. For each $r>0$ we give ourselves a maximal $(1+4R)$-separated subset of $\Gamma\cdot o\cap B_\Omega(o,r+1)\smallsetminus B_\Omega(o,r)$. One can easily see that the shadows $(\mathcal{O}_R(o,x))_{x\in F_r}$ are pairwise disjoint, therefore
\begin{align*}
 1&\geq \sum_{x\in F_r}\nu_o(\mathcal{O}_R(o,x))\\
&\geq C^{-1}\sum_{x\in F_r}e^{-\delta d_\Omega(o,x)}\\
&\geq C^{-1}e^{-\delta}e^{-\delta r}\#F_r\\
&\geq C^{-1}e^{-\delta}\#\{\gamma : d_\Omega(o,\gamma o)\leq 1+4R\}^{-1}e^{-\delta r}\#\{\gamma : \gamma o\in B_\Omega(o,r+1)\smallsetminus B_\Omega(o,r)\}.
\end{align*}
This implies that $\#\{\gamma : d_\Omega(o,\gamma o)\leq r\}\leq C'e^{\delta r}$ for any $r>0$, for some $C'>0$ independent of $r$. By definition, this implies that $\delta\geq \delta_\Gamma$, and since by Fact~\ref{smear} there exists a $\delta_\Gamma$-conformal density, $\#\{\gamma : d_\Omega(o,\gamma o)\leq r\}\leq C''e^{\delta_\Gamma r}$ for any $r>0$, for some $C''>0$ independent of $r$.
\end{proof}

\begin{prop}\label{Proposition : faces coniques negligees}
 Let $\Omega\subset\PR(V)$ be a properly convex open set and $\Gamma\subset\Aut(\Omega)$ a discrete subgroup; set $M=\Omega/\Gamma$. Suppose that $\Gamma$ is strongly irreducible and $T^1M_{\bip}$ is non-empty, or that $M$ is rank-one and non-elementary. Consider $\delta\geq0$ and a $\delta$-conformal density $(\nu_x)_{x\in\Omega}$ on $\partial\Omega$. Then $\nu_x(F_\Omega(\xi))=0$ for all $x\in\Omega$ and $\xi\in\Lambda^{\con}$.
\end{prop}

\begin{proof}
 Let $o\in\Omega$. Observe that the open face $F_\Omega(\xi)$ is contained in $\Lambda^{\con}$ by \cite[Cor.\,4.10]{fannycvxcocpct}. It's enough to prove that $\nu_o(B_{\overline{\Omega}}(\xi,R))=0$ for any $R>0$. By definition, there are sequences $(\gamma_n)_n\in\Gamma^\N$ and $(x_n)_n\in[o,\xi)^\N$ going to infinity such that $(d_\Omega(\gamma_no,x_n)_n$ is bounded; denote $R'=\sup_nd_\Omega(\gamma_no,x_n)$. Using the Shadow lemma (Lemma~\ref{shadowlemma}), we can find $C>0$ such that for any $n$,
 \[\nu_o(B_{\overline{\Omega}}(\xi,R))\leq \nu_o(\mathcal{O}_R(o,x_n))\leq\nu_o(\mathcal{O}_{R+R'}(o,\gamma_no))\leq Ce^{-\delta d_\Omega(o,\gamma_no)}.\]
 This last term goes to zero as $n$ tends to infinity since $\delta\geq\delta_\Gamma>0$ (Fact~\ref{exposant critique non nul} and Proposition~\ref{delta<deltaGamma}).
\end{proof}

\section{The convergent case of the HTSR dichotomy}\label{Section : HTSR convergent}

In this section, we establish the convergent case of the HTSR dichotomy (Theorem~\ref{Thm : The weak Hopf--Tsuji--Sullivan--Roblin dichotomy}.\ref{Item : cas convergent}).

\subsection{The conical limit set has zero measure}\label{ensemble conique est nul}

We prove that, in the convergent case of the HTSR dichotomy, any $\delta$-conformal density gives zero measure to the conical limit set.

\begin{prop}\label{Prop : ensemble conique est nul}
 Let $\Omega\subset\PR(V)$ be a properly convex open set and $\Gamma\subset\Aut(\Omega)$ a discrete subgroup. Suppose that $\Gamma$ is strongly irreducible and $T^1M_{\bip}$ is non-empty, or that $M$ is rank-one and non-elementary. Let $\delta\geq\delta_\Gamma$ with $\sum_{\gamma\in\Gamma} e^{-\delta d_\Omega(o,\gamma o)}$ finite, and consider a $\delta$-conformal density $(\nu_x)_{x\in\Omega}$ on $\partial\Omega$. Then $\nu_x(\Lambda^{\con})=0$ for any $x\in\Omega$.
\end{prop}

\begin{proof}
 We consider $R>0$ and prove that $\nu_o(\Lambda^{\con}_R)=0$. By definition, for any $r>0$, the set $\Lambda^{\con}_R$ is contained in the union, over all $\gamma\in \Gamma$ such that $d_\Omega(o,\gamma o)\geq r$, of the shadows $\mathcal{O}_R(o,\gamma o)$. As a consequence, by the Shadow lemma (Lemma~\ref{shadowlemma}), we can find a constant $C>0$ such that 
\[\nu_o(\Lambda^{\con}_R)\leq C\sum_{\gamma : d_\Omega(o,\gamma o)\geq r} e^{-\delta d_\Omega(o,\gamma o)}\]
 for any $r>0$, and this last quantity converges to zero as $r$ goes to infinity.
\end{proof}

\subsection{Proof of Theorem~\ref{Thm : The weak Hopf--Tsuji--Sullivan--Roblin dichotomy}.\ref{Item : cas convergent}}

Let $\mu_o$ be a $\delta_\Gamma$-conformal density on $\partial_{\hor}\Omega$ such that $(\pi_{\hor})_*\mu_o=\nu_o$, and let $\mm$ be the Sullivan measure on $T^1\Omega$ induced by $\mu_o$. According to Proposition~\ref{delta<deltaGamma}, we have $\delta\geq \delta_\Gamma$. Let us assume that $\sum_\gamma e^{-\delta d_\Omega(o,\gamma o)}$ is convergent.

By Fact~\ref{Fait : conservatif passe au quotient}, in order to prove that $(T^1M,\phi_t,m)$ and $(\Geod^\infty(\Omega),\Gamma,\nu_o^2)$ are dissipative, it is enough to prove that $(T^1\Omega,\Gamma\times \R,\tilde{m})$ is dissipative. If by contradiction it is not the case, then the conservative part contains a compact subset $K$ of positive measure. This means that for almost any vector $v\in K$, using notations from Section~\ref{Section : decompo de Hopf},
\begin{equation*}
 \infty=\int_{\Gamma\times\R}1_K(v)= \sum_{\gamma\in\Gamma}\int_{-\infty}^\infty1_{K}(\gamma\phi_t v)\diff t,
\end{equation*}
hence there exist diverging sequences $(\gamma_n)_n\in\Gamma^\N$ and $(t_n)_n\in\R^\N$ such that $\gamma_n\phi_{t_n}v\in K$; if $(t_n)_n$ tends to $\infty$ (\resp $-\infty$), then $\phi_{\infty}v$ (\resp $\phi_{-\infty}v$) is in $\Lambda^{\con}$, which contradicts Proposition~\ref{Prop : ensemble conique est nul} and the definition of $\tilde{m}$ since for almost every vector $v$ in $T^1\Omega$, the endpoints $\phi_{\pm\infty} v$ are not in $\Lambda^{\con}$.
 
Suppose by contradiction that $(T^1M,\phi_t,m)$ is ergodic. Let $K\subset T^1\Omega$ be a compact neighbourhood of a vector $v$ in $T^1\Omega_{\bip}:=\pi_{\Gamma}^{-1}T^1M_{\bip}$. By Fact~\ref{fait:ergo imp trans}.\ref{item:ergo imp trans:ergo}, almost any $(\phi_t)_t$ orbit is dense in the support of $m$, which contains $T^1M_{\bip}$ (this was basically proved in Proposition~\ref{prop:Radonnonnul}). By dissipativity (and because $m$ is Radon), almost any orbit passes a finite time in $\pi_\Gamma K$. Hence there exists a vector $w\in T^1\Omega$ and a time $T>0$ such that $\pi_\Gamma (K)\cap T^1M_{\bip}\subset \phi_{[-T,T]}\pi_\Gamma w$; in other words, as the action of $\Gamma$ on $T^1\Omega$ is properly discontinuous, we may find a finite subset $A\subset\Gamma$ such that $K\cap T^1\Omega_{\bip}\subset A\phi_{[-T,T]}w$. Thus, $\phi_{\infty}(K\cap T^1\Omega_{\bip})$, which is a neighbourhood in $\Lambda^{\prox}$ of $\phi_{\infty}v$, is contained in $\phi_{\infty}A\phi_{[-T,T]}w$, which is equal to $A\phi_{\infty}w$ and hence is finite. This contradicts the fact that $\Lambda^{\prox}$ has no isolated point (Remark~\ref{minimality} and Fact~\ref{fait:des rk1 partout}).
 

\section{The divergent case of the HTSR dichotomy}\label{Section : HTSR divergent}

In this section we adapt some proofs of Roblin \cite[p.\,19-23]{roblin_smf} to our convex projective setting, in order to establish Theorem~\ref{Thm : The weak Hopf--Tsuji--Sullivan--Roblin dichotomy}. We will fix a Patterson--Sullivan density, and need several time to prove that some $\Gamma$-invariant subset $A$ of $\partial\Omega$ is given full measure by the Patterson--Sullivan density. Thanks to the following elementary observation, it is often enough to show that $A$ has non-zero measure.

\begin{obs}\label{obs:obsconf}
 Let $\Omega\subset\PR(V)$ be a properly convex open set, $\Gamma\subset\Aut(\Omega)$ a closed subgroup, $A\subset\partial\Omega$ a $\Gamma$-invariant measurable subset, and $\delta\geq 0$. Then for any $\delta$-conformal density $(\nu_x)_{x\in\Omega}$ on $\partial\Omega$, the restrictions $(\nu_x{}_{|A})_{x\in\Omega}$ also form a $\delta$-conformal density. As a consequence, if $\nu_x(A)>0$ for any $x\in\Omega$ and any (non-zero) $\delta$-conformal density $(\nu_y)_{y\in\Omega}$, then $\nu_x(\partial\Omega\smallsetminus A)=0$ for any $x\in\Omega$ and any $\delta$-conformal density $(\nu_y)_{y\in\Omega}$.
\end{obs}

\subsection{The conical limit set has full measure}\label{ssection:bcp de coniques}

Let $o\in\Omega\subset\PR(V)$ be a pointed properly convex open set and $\Gamma\subset\Aut(\Omega)$ a discrete subgroup. Recall that the conical limit set, denoted by $\Lambda^{\con}$, is the union, over $R>0$, of the sets $\Lambda^{\con}_{R}(\Gamma,\Omega,o)=\Lambda^{\con}_R$, consisting of points $\xi\in\partial\Omega$  for which there exists a sequence $(\gamma_n)_{n\in\N}\in\Gamma^\N$ going to infinity such that $d_\Omega([o,\xi),\gamma_no)<R$ for each $n\in\N$.


Thanks to Observation~\ref{obs:obsconf}, we only need to prove that the conical limit set has non-zero measure. Roblin's proof of this \cite[Item\,$(f)$ p.\,19]{roblin_smf} in CAT($-1$) geometry actually works verbatim in the present context; we rewrite it down for the convenience of the (non-french-speaking) reader.

\begin{prop}\label{l'ensemble limite conique a mesure non nulle}
 Let $\Omega\subset\PR(V)$ be a properly convex open set and $\Gamma\subset\Aut(\Omega)$ a discrete subgroup. Suppose $\Gamma$ is divergent, and $M=\Omega/\Gamma$ is rank-one and non-elementary. Consider a $\delta_\Gamma$-conformal density $(\nu_x)_{x\in\Omega}$ on $\partial\Omega$. Then $\nu_x(\partial\Omega\smallsetminus\Lambda^{\con})=0$ for any $x\in\Omega$.
\end{prop}

The idea of the proof of Proposition~\ref{l'ensemble limite conique a mesure non nulle} is to find a compact subset $K\subset T^1M$ such that the set of vectors $v\in K$ whose geodesic come back infinitely often to $K$ has positive $m$-measure where $m$ is a Sullivan measure induced by $(\nu_x)_{x\in\Omega}$. There are two main ingredients. The first one is a generalisation of the Borel--Cantelli lemma, due to R\'enyi. One can find a proof of it in \cite[Lem.\,2]{rational_ergodicity}; it is the consequence of the following estimate : for any non-negative square-integrable function $g$ on a probability space, for any $0<a<E(g)$, where $E(g)$ is the expectation of $g$, one has
\[P(g>a) \geq \frac{E(g)^2}{E(g^2)}\left(1+\frac{a}{E(g)-a}\right)^{-2}.\]

\begin{fait}[\!{\!\cite[p.\,391]{renyi}}]\label{Borel-Cantelli}
 Let $(X,\mu)$ be a measurable space equipped with a finite positive measure. Let $(A_t)_{t\geq 0}$ be a family of subsets of $X$ such that the function $(x,t)\in X\times\R_{\geq0}\mapsto 1_{A_t}(x)$ is measurable. Let us assume that $\int_0^\infty \mu(A_t)\diff t = \infty$, and that, for some constant $C>0$, 
\begin{equation}\label{Equation : hypothese renyi}
\int_0^T \int_0^T\mu(A_t\cap A_s)\diff t\diff s \leq C \left(\int_0^T\mu(A_t)\diff t\right)^2,
\end{equation}
for $T$ large enough. Then the set of $x\in X$ such that $\int_0^\infty 1_{A_t}(x)\diff t=\infty$ has $\mu$-measure greater than or equal to $1/C$.
\end{fait}

This lemma will be applied to $A_t=K\cap\phi_{-t}K$. The second ingredient consists of estimates that will allow us to check that the assumptions of Fact~\ref{Borel-Cantelli} are satisfied.

\begin{lemma}\label{Lemme : estimees moches}
 Let $o\in\Omega\subset\PR(V)$ be a pointed properly convex open set and $\Gamma\subset\Aut(\Omega)$ a discrete subgroup. Suppose $M=\Omega/\Gamma$ is rank-one and non-elementary. Consider a $\delta_\Gamma$-conformal density $(\nu_x)_{x\in\Omega}$ on $\partial\Omega$, with induced Sullivan measure $m$ on $T^1M$. Let $v_o\in T^1_o\Omega$. For $R>0$ large enough, if we denote by $K$ the projection in $T^1M$ of $\overline{B}_{T^1\Omega}(v_o,R)$, then there exist constants $C>0$ and $T_0$ such that for any $T>T_0$, we have the estimates:
 \begin{equation}\label{Equation : minoration moche}
  \int_0^T m(K\cap\phi_{-t}K)\diff t \geq C^{-1}\sum_{\substack{g\in\Gamma \\ d_\Omega(o,go)\leq T}}e^{-\delta_\Gamma d_\Omega(o,go)},
 \end{equation}
 \begin{equation}\label{Equation : majoration moche}
  \int_0^T\int_0^T m(K\cap\phi_{-t}K\cap\phi_{-t-s}K)\diff s\diff t \leq C \left(\sum_{\substack{g\in\Gamma \\ d_\Omega(o,go)\leq T}}e^{-\delta_\Gamma d_\Omega(o,go)}\right)^2.
 \end{equation}
\end{lemma}

\begin{proof}
We assume that $\nu_o$ is a probability measure. Let $(\mu_x)_{x\in\Omega}$ be a $\delta_\Gamma$-conformal density on $\partial_{\hor}\Omega$ that induces $(\nu_x)_{x\in\Omega}$ and $m$; we denote by $\tilde{m}$ the induced Sullivan measure on $T^1\Omega$. We fix $R>0$ large enough so that we can apply the Shadow lemma (Lemmas~\ref{shadowlemma} and \ref{shadowlemma avec defaut}) and Lemmas~\ref{l ombre grandit a la frontiere} and \ref{l'ombre en defaut grandit a la frontiere} to it and to $R' := (R-2)/6$. One can find a constant $C_1>0$ such that 
\begin{gather*}
 m(K\cap\phi_{-t}K)\geq C_1^{-1} \sum_{g\in\Gamma}\tilde{m}(\tilde{K}\cap\phi_{-t}g \tilde{K}), \quad \text{and} \\
 m(K\cap\phi_{-t}K\cap\phi_{-t-s}K)\leq \sum_{g,h\in\Gamma}\tilde{m}(\tilde{K}\cap\phi_{-t}g \tilde{K}\cap \phi_{-t-s}h\tilde{K})
\end{gather*}
for all $t,s\geq 0$. Indeed, we just have to recall the definition of $m$, which is the quotient of $\tilde{m}$ under the action of $\Gamma$ (Definition~\ref{Definition : quotient de mesure}), then
\begin{align*}
 \sum_{g\in\Gamma}\tilde{m}(\tilde{K}\cap\phi_{-t}g\tilde{K}) & = \int_{T^1\Omega}\sum_{g\in\Gamma}1_{\tilde{K}}1_{g\phi_{-t}\tilde{K}}\diff \tilde{m} \\
 & = \int_{T^1M}\sum_{h\in\Gamma}\sum_{g\in\Gamma}(1_{\tilde{K}}\circ h)\cdot (1_{g\phi_{-t}\tilde{K}}\circ h)\diff m \\
 & = \int_{T^1M}\underbrace{\left(\sum_{h\in\Gamma}1_{h\tilde{K}}\right)}_{\leq C_1 1_K}\cdot \underbrace{\left(\sum_{g\in\Gamma}1_{g\phi_{-t}\tilde{K}}\right)}_{\leq C_1 1_{\phi_{-t}K}}\diff m,
\end{align*}
where $C_1>0$ is a constant which is independent of $t$. Similarly,
\begin{align*}
 \sum_{g,h\in\Gamma}\tilde{m}(\tilde{K}\cap\phi_{-t}g\tilde{K}\cap\phi_{-t-s}h\tilde{K}) & = \int_{T^1M}\sum_{k\in\Gamma}\sum_{g,h\in\Gamma}(1_{\tilde{K}}\circ k)\cdot (1_{g\phi_{-t}\tilde{K}}\circ k)\cdot(1_{h\phi_{-t-s}\tilde{K}}\circ k)\diff m \\
 & = \int_{T^1M}\underbrace{\left(\sum_{k\in\Gamma}1_{k\tilde{K}}\right)}_{\geq 1_K}\cdot \underbrace{\left(\sum_{g\in\Gamma}1_{g\phi_{-t}\tilde{K}}\right)}_{\geq 1_{\phi_{-t}K}}\cdot \underbrace{\left(\sum_{h\in\Gamma}1_{h\phi_{-t-s}\tilde{K}}\right)}_{\geq 1_{\phi_{-t-s}K}}\diff m
\end{align*}

Therefore, in order to prove Lemma~\ref{Lemme : estimees moches}, it is enough to find a constant $C_{2}>0$ such that
\[a:=\int_{0}^{T}\sum_{g\in \Gamma}\tilde{m}(\tilde{K}\cap \phi_{-t}g\tilde{K})\diff t \geq C_{2}^{-1}\sum_{g : d_{\Omega}(o,go)\leq T}e^{-\delta_\Gamma d_{\Omega}(o,go)},\]
and
\[b:=\int_{0}^{T}\int_{0}^{T}\sum_{g,h\in \Gamma}\tilde{m}(\tilde{K}\cap \phi_{-t}g\tilde{K}\cap \phi_{-t-s}h\tilde{K})\diff s\diff t \leq C_{2}\left(\sum_{g : d_{\Omega}(o,go)\leq T}e^{-\delta_\Gamma d_{\Omega}(o,go)}\right)^2.\]
We first establish the estimate of $b$. For all $0\leq t,s\leq T$ and $g,h\in\Gamma$, for any triple $(\xi,\eta,\tau)$ in $\Hopf^{-1}(\tilde{K}\cap \phi_{-t-s}h\tilde{K})$, one observes that $\eta\in\pi_{\hor}^{-1}\mathcal{O}^+_{R}(o,ho)$ and $d_\Omega (o,\pi \Hopf (\xi,\eta,\tau))\leq R$; this last inequality implies that $\vert \tau \vert \leq R$ and  $\langle \xi, \eta \rangle_{o} \leq R$; then by the Shadow lemma (Lemma~\ref{shadowlemma}), and by definition of $\tilde{m}$,
\[\tilde{m}(\tilde{K}\cap \phi_{-t}g\tilde{K}\cap \phi_{-t-s}h\tilde{K})\leq 2Re^{2\delta_\Gamma R}\nu_{0}(\mathcal{O}^+_{R}(o,ho))\leq 2Re^{2\delta_\Gamma R}Ce^{-\delta_\Gamma d_{\Omega}(o,ho)}.\]
Furthermore, by triangular inequality, if $\tilde{K}\cap \phi_{-t}g\tilde{K}\cap \phi_{-t-s}h\tilde{K}$ is non-empty, then
\begin{gather*}
 \vert d_{\Omega}(o,go)-t\vert, \vert d_{\Omega}(go,ho)-s\vert, \vert d_{\Omega}(o,ho)-t-s\vert \leq 2R, \quad \text{and}\\
 d_{\Omega}(o,go)+d_{\Omega}(go,ho)\leq d_{\Omega}(o,ho)+6R.
\end{gather*}
Combining these estimates we obtain:
\begin{align*}
b & \leq \int_{0}^{T}\int_{0}^{T}\sum_{g,h\in \Gamma}2Re^{2\delta_\Gamma R}Ce^{-\delta_\Gamma d_\Omega (o,ho)}1_{ \Big\{ \substack{\vert d_{\Omega}(o,go)-t\vert, \vert d_{\Omega}(go,ho)-s\vert  \leq 2R \\ d_{\Omega}(o,go)+d_{\Omega}(go,ho)\leq d_{\Omega}(o,ho)+6R} \Big\} } \diff s\diff t \\
& \leq 2R2R2Re^{2\delta_\Gamma R}C\sum_{\substack{d_\Omega (o,go) \leq T+2R \\ d_\Omega (o,g^{-1}ho)\leq T+2R}}e^{-\delta_\Gamma (d_\Omega (o,go) + d_{\Omega}(o,g^{-1}ho) -6R)} \\
& \leq 8R^3Ce^{8\delta_\Gamma R}\left( \sum_{d_\Omega (o,go)\leq T+2R}e^{-\delta_\Gamma d_\Omega (o,go)})\right)^2.
\end{align*}
We end the estimation of $b$ by noting that for all $T\geq 0$ and $A\geq 0$, using again the Shadow lemma (Lemma~\ref{shadowlemma}),
\begin{align*}
\sum_{T\leq d_\Omega (o,go) \leq T+A}e^{-\delta_\Gamma d_\Omega (o,go)} & \leq C \sum_{T\leq d_\Omega (o,go) \leq T+A}\nu_{o}(\mathcal{O}_{R}(o,go)) \\
& \leq C\int_{\partial\Omega} \sum_{T\leq d_\Omega (o,go) \leq T+A}1_{\mathcal{O}_{R}(o,go)}(\xi)\diff \nu_{o}(\xi) \\
& \leq C \# \{g : d_\Omega (o,go) \leq 4R+2A\}.
\end{align*}

We now proceed to the minoration of $a$. Take $g\in \Gamma$, take $t\geq 0$ at distance less than $R'$ from $d_\Omega(o,go)$; let us prove that 
\[\Hopf\left(\pi_{\hor}^{-1}(\mathcal{O}^-_{R'}(go,o))\times \pi_{\hor}^{-1}(\mathcal{O}^-_{R'}(o,go))\times [0,R'] \right) \subset \tilde{K}\cap \phi_{-t}g\tilde{K}.\]
Take $(\xi,\eta,\tau)\in \pi_{\hor}^{-1}(\mathcal{O}^-_{R'}(go,o))\times \pi_{\hor}^{-1}(\mathcal{O}^-_{R'}(o,go))\times [0,R']$. According to Observation~\ref{LR<OR*OR} below, there exists $s_1<s_2$ such that $d_\Omega(\pi\Hopf(\xi,\eta,s_1),o)\leq R'$ and $d_\Omega(\pi\Hopf(\xi,\eta,s_2),go)\leq R'$. This implies that $|s_1|=|b_\eta(\pi\Hopf(\xi,\eta,s_1),o)|\leq R'$, hence 
\[d_\Omega(\pi\Hopf(\xi,\eta,\tau),o)\leq |\tau|+|s_1|+d_\Omega(\pi\Hopf(\xi,\eta,s_1),o)\leq 3R'.\]
Finally $d_{T^1\Omega}(\Hopf(\xi,\eta,\tau),v_o)\leq 3R'+2\leq R$, which means that $\Hopf(\xi,\eta,\tau)\in \tilde{K}$. 

In order to prove the inclusion in $\phi_{-t}g\tilde{K}$, we note that since $s_2-s_1$ is the distance between $\pi\Hopf(\xi,\eta,s_2)$ and $\pi\Hopf(\xi,\eta,s_1)$, then by triangular inequality $|s_2-d_\Omega(o,go)|\leq 3R'$; therefore
\begin{align*}
d_\Omega(\pi\Hopf(\xi,\eta,t+\tau),go)&\leq |\tau|+|t-d_\Omega(o,go)|+|d_\Omega(o,go)-s_2|+d_\Omega(\pi\Hopf(\xi,\eta,s_2),go)\\
&\leq 6R'.
\end{align*}
Finally $d_{T^1\Omega}(\Hopf(\xi,\eta,t+\tau),gv_o)\leq 6R'+2=R$, which means that $\Hopf(\xi,\eta,\tau)\in \phi_{-t}g\tilde{K}$.

As a consequence, if $d_\Omega(o,go)\geq 2R'$, then by the Shadow lemma (Lemma~\ref{shadowlemma avec defaut}) and Lemma~\ref{l'ombre en defaut grandit a la frontiere},
\[\tilde{m}(\tilde{K}\cap\phi_{-t}g\tilde{K})\geq \nu_o(\mathcal{O}^-_{R'}(o,go))\nu_o(\mathcal{O}^-_{R'}(go,o))R'\geq R'C^{-2}e^{-\delta_\Gamma d_\Omega(o,go)}.\]
We conclude that for $T\geq R'$,
\[\int_{0}^{T}\sum_{g\in \Gamma}\tilde{m}(\tilde{K}\cap \phi_{-t}g\tilde{K})\diff t \geq R'{}^2C^{-2}\left(\sum_{g : d_\Omega(o,go)\leq T}e^{-\delta_\Gamma d_\Omega(o,go)}-\sum_{g : d_\Omega(o,go)\leq 2R'}e^{-\delta_\Gamma d_\Omega(o,go)}\right).\qedhere\]
\end{proof}

\begin{obs}\label{LR<OR*OR}
Let $n\geq 1$ be an integer, $A$ and $B$ be two non-empty disjoint compact subsets of $\R^n$, and $\xi,\eta\in\R^n$ be two points. If for all $a\in A$ and $b\in B$, the intersections $[a,\eta]\cap B$ and $[b,\xi]\cap A$ are non-empty, then one can find $a\in A$ and $b\in B$ such that $\xi,a,b,\eta$ are aligned in this order.
\end{obs}

\begin{proof}[Proof of Proposition~\ref{l'ensemble limite conique a mesure non nulle}]
We let $m$ be a Sullivan measure associated to $(\nu_x)_{x\in\Omega}$ on $T^1M$. By definition of $m$ and of the conical limit set, it is enough to find a compact set $K\subset T^1M$ large enough so that the set of vectors $v\in K$ such that $\int_0^\infty 1_K(\phi_t v)\diff t=\infty$ has non-zero $m$-measure.

Let $R>0$ be large enough so that we can apply Lemma~\ref{Lemme : estimees moches}, and let $K$ be the projection in $T^1M$ of $\overline{B}_{T^1\Omega}(v_o,R)$, where $v_o\in T^1_o\Omega$ and $o\in\Omega$. We want to apply Fact~\ref{Borel-Cantelli} for $(X,\mu)=(T^1M,m_{|K})$ and $A_t:= K\cap \phi_{-t}K$ for all $t\geq 0$. The measure $\mu$ is finite because $m$ is Radon and $K$ is compact. Since $\Gamma$ is divergent (this is important since it is the only place where we need this assumption), and by the estimate \eqref{Equation : minoration moche} of Lemma~\ref{Lemme : estimees moches}, the integral $\int_0^\infty \mu(A_t)\diff t$ diverges. It remains to check that \eqref{Equation : hypothese renyi} is satisfied, but this is a direct consequence of the estimates \eqref{Equation : minoration moche} and \eqref{Equation : majoration moche} of Lemma~\ref{Lemme : estimees moches}, and of the fact that
\[\int_0^T \int_0^T\mu(A_t\cap A_s)\diff t\diff s \leq 2\int_0^T \int_0^T\mu(A_t\cap A_{t+s})\diff t\diff s.\qedhere\]\end{proof}

\subsection{The geodesic flow is conservative}\label{conservativite}

In this section we prove, that, in the setting of Theorem~\ref{Thm : The weak Hopf--Tsuji--Sullivan--Roblin dichotomy}.\ref{Item : cas divergent}, the geodesic flow is conservative.

\begin{prop}\label{Proposition : convervativite}
 Let $\Omega\subset\PR(V)$ be a properly convex open set and $\Gamma\subset\Aut(\Omega)$ a divergent discrete subgroup. Suppose $M=\Omega/\Gamma$ is rank-one and non-elementary. Let $(\nu_x)_{x\in\Omega}$ be a $\delta_\Gamma$-conformal density on $\partial\Omega$ and $m$ an induced Sullivan measure on $T^1M$. Then $m$ is conservative under the action of $(\phi_t)_t$.
\end{prop}

\begin{proof}
 Let $(\mu_x)_{x\in\Omega}$ be a $\delta_\Gamma$-conformal density on $\partial_{\hor}\Omega$ such that $(\pi_{\hor})_*\mu_x=\nu_x$ for any $x\in\Omega$, and let $\mm$ be the induced Sullivan measure on $T^1\Omega$. 
 
 According to Fact~\ref{Fait : conservatif passe au quotient}, it is enough to prove that $(T^1\Omega,\Gamma\times \R,\tilde{m})$ is conservative. Let $\sigma$ be an integrable, positive and continuous function on $T^1\Omega$, and let us prove that $\int_{\Gamma\times\R}\sigma=\infty$ almost surely (recall the notation from Section~\ref{Section : decompo de Hopf}). For almost any vector $v\in T^1\Omega$, we have $\phi_\infty v \in \Lambda^{\con}$ by Proposition~\ref{l'ensemble limite conique a mesure non nulle}, and by definition of $\tilde{m}$; let $v$ be such a vector. Then we can find $R>0$ such that $\phi_\infty v\in\Lambda_R^{\con}$; in other words there exist $(\gamma_n)_n\in\Gamma^\N$ and $(t_n)_n\in\R^\N$ such that $\gamma_n\neq\gamma_k$ and $d_\Omega(\pi\phi_{t_n}\gamma_nv,o)\leq R$ for $n\neq k$. Let $\epsilon:=\min\{\sigma(w) : d_\Omega(\pi w,o)\leq R+1\}>0$. Then
 \begin{equation*}
  \int_{\Gamma\times\R}\sigma(v) \geq \sum_{n\geq 0}\int_{t=t_n-1}^{t_n+1}\sigma(\phi_t\gamma v)\diff t \geq \sum_{n\geq 0}2\epsilon =\infty.\qedhere
 \end{equation*}
\end{proof}

\subsection{The smooth and strongly extremal points have full measure}\label{Ssection : negligeons les singularites}

In this section we combine the conservativity of Sullivan's measures with a result of Benzécri to establish that in the divergent case, the smooth and strongly extremal points have full measure. The following result is a particular case of a more general result of Benz\'ecri. Recall that $\Ecal_V^\bullet$ is the space of pointed properly convex open subsets of $\PR(V)$, endowed with the Hausdorff topology, on which $\PGL(V)$ acts properly cocompactly (Fact~\ref{benz}).

\begin{fait}[\!{\!\cite[Prop.\,5.3.9]{benz_varlocproj}}]\label{Fait : triangle limite}
 Suppose that $\dim(V)=3$. Let $(x,\Omega)\in\Ecal_V^\bullet$ and $\xi\in\partial_{\sing}\Omega$. Let $T\in \Ecal_V$ be a triangle and $o\in T$. Then we have the following convergence in $\Ecal_V^\bullet/\PGL(V)$.
 \[[y,\Omega] \underset{\substack{y\to\xi\\y\in[x,\xi)}}{\longrightarrow} [o,T].\]
\end{fait}

\begin{proof}
 We briefly recall the proof of this fact, which is very easy in this particular case. Consider a projective line $\PR(W)$ that does not intersect $[x,\xi]$, and for each $y\in[x,\xi)$, denote by $g_y\in\PGL(V)$ the unique element that fixes $\xi$ and $\PR(W)$, and such that $g_yy=x$. Let $p : \PR(V)\smallsetminus\{\xi\}\rightarrow \PR(W)$ be the stereographic projection. Since $\xi$ is a singular point of $\partial\Omega$, the image $p(\Omega)$ is a properly convex open subset of $\PR(W)$, \ie the interior of a segment. As $y$ tends to $\xi$, the properly convex open set $g_y\Omega$ converges to the convex hull of $p(\Omega)$ and $\xi$, which is a triangle containing $x$.
\end{proof}


\begin{lemma}\label{Lemme : recurrent + rg1 faible -> lisse}
 Let $\Omega\subset\PR(V)$ be a properly convex open set and $\Gamma\subset\Aut(\Omega)$ a discrete subgroup; set $M=\Omega/\Gamma$. Let $v\in T^1\Omega$ with a forward recurrent projection in $T^1M$. If $\phi_\infty v\not\in\partial_{sse}\Omega$, then $d_{\spl}(\phi_{-\infty}v,\phi_{\infty}v)=2$.
\end{lemma}

\begin{proof}
 Observe that the inequality $d_{\spl}(\phi_{-\infty}w,\phi_{\infty}w)\geq2$ holds for any vector $w$.
 Since the projection of $v$ in $T^1M$ is forward recurrent, there exist diverging sequences $(t_n)_n\in[0,\infty)^\N$ and $(\gamma_n)_n\in\Gamma^\N$ such that $(\gamma_n\phi_{t_n}v)_n$ tends to $v$.
 
 Suppose that $\phi_\infty v$ is singular. Then there exists a projective plane $\PR(W)\subset\PR(V)$ which contains $v$ and such that $\xi$ is a singular point of $\partial\Omega\cap\PR(W)$. Up to extraction, we can assume that $(\gamma_n\PR(W))_n$ converges to some $\PR(W')$. By construction, $(\gamma_n\pi\phi_{t_n}v,\Omega\cap\gamma_n\PR(W))_n$ converges to $(\pi v,\Omega\cap\PR(W')$. Since $\phi_\infty v$ is singular, we can apply Fact~\ref{Fait : triangle limite}, and we obtain that $\Omega\cap\PR(W')$ is a triangle that contains $\phi_{\pm\infty}v$, hence $d_{\spl}(\phi_{-\infty}v,\phi_\infty v)\leq 2$.
 
 Suppose there exists $\xi\in\partial\Omega\smallsetminus\{\phi_\infty v\}$ such that $[\xi,\phi_\infty v]\subset\partial\Omega$. We can take $\xi$ extremal. Up to extraction we can assume that $(\gamma_n\xi)_n$ converges to some $\xi'\in\partial\Omega$. Observe that $[\phi_\infty v,\xi']\subset\partial\Omega$, since it is the limit of the sequence of segments $([\gamma_n\phi_\infty v,\gamma_n \xi])_n$ that are contained in the boundary (see Figure~\ref{fig:rec+rk1->lisse}). Since $\xi$ is extremal, the Hilbert distance of $\pi\phi_{t_n}v$ to $[\phi_{-\infty}v,\xi]\cap\Omega$ tends to infinity with $n$; this implies that $[\phi_{-\infty}v,\xi']\subset\partial\Omega$. Thus $d_{\spl}(\phi_{-\infty}v,\phi_\infty v)\leq 2$.
 \begin{figure}
\centering
\begin{tikzpicture}[scale=2]
\coordinate (phiminfv) at (-.3,-1);
\coordinate (phiinfv) at (.2,1.05);
\coordinate (v) at ($(phiminfv)!.5!(phiinfv)$);
\coordinate (xi) at (1,.2);
\coordinate (gphiminfv) at (-0.7,-.8);
\coordinate (gphiinfv) at (-.35,1.1);
\coordinate (gv) at ($(gphiminfv)!.45!(gphiinfv)$);
\coordinate (gxi) at (-1.5,.5);
\coordinate (xi') at (-1.6,.1);
\coordinate (seg) at ($(xi)!.5!(phiinfv) + (.0001,.0001)$);
\coordinate (gseg) at ($(gxi)!.5!(gphiinfv) + (-.0001,.0001)$);
\coordinate (phiv) at ($(v)!.5!(phiinfv)$);

\length{phiminfv,xi,seg,phiinfv,gphiinfv,gseg,gxi,xi',gphiminfv}
\cvx{phiminfv,xi,seg,phiinfv,gphiinfv,gseg,gxi,xi',gphiminfv}{1}

\draw (phiminfv)--(phiinfv);
\draw (phiminfv)--(xi);
\draw [->,very thick,red] (v)--($(v)!.2!(phiinfv)$);
\draw [->,very thick,red] (phiv)--($(phiv)!.4!(phiinfv)$);
\draw (gphiminfv)--(gphiinfv);
\draw [->,very thick,orange!50!red] (gv)--($(gv)!.2!(gphiinfv)$);
\draw (gphiminfv)--(gxi);
\draw [thick,blue] (xi)--(phiinfv);
\draw [thick,blue] (gxi)--(gphiinfv);

\draw (phiinfv) node{$\bullet$} node[above right]{$\phi_\infty v$};
\draw (gphiinfv) node{$\bullet$} node[above]{$\gamma_n\phi_\infty v$};
\draw (gxi) node{$\bullet$} node[above left]{$\gamma_n\xi$};
\draw (xi) node{$\bullet$} node[above right]{$\xi$};
\draw (phiminfv) node{$\bullet$} node[below]{$\phi_{-\infty}v$};
\draw (gphiminfv) node{$\bullet$} node[below left]{$\gamma_n\phi_{-\infty}v$};
\draw (xi') node{$\bullet$} node[left]{$\xi'$};
\draw (v) node[red,right]{$v$};
\draw (phiv) node[red,right]{$\phi_{t_n}v$};
\draw (gv) node[orange!50!red,left]{$\gamma_n\phi_{t_n}v$};
\end{tikzpicture}
\caption{Illustration of the second part of the proof of Lemma~\ref{Lemme : recurrent + rg1 faible -> lisse}}\label{fig:rec+rk1->lisse}
\end{figure}
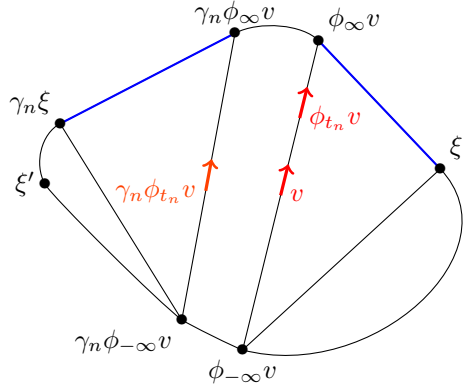
\end{proof}

Observe that the following corollary of Lemma~\ref{Lemme : recurrent + rg1 faible -> lisse} generalises with a shorter proof a result of Islam \cite[Prop.\,6.3]{islam_rank_one}. It can also be deduced from Lemma~\ref{lem:realell}.\ref{item:realellpasr1}.


\begin{cor}\label{cor:carac rgun}
 Let $\Omega\subset \PR(V)$ be a properly convex open set and $g\in\Aut(\Omega)$ with $\ell(g)>0$ and such that $g$ fixes two points  $\xi,\eta\in\partial\Omega$ with $d_{\spl}(\xi,\eta)>2$. Then $g$ is rank-one and $\{\xi,\eta\}=\{x_g^+,x_g^-\}$.
\end{cor}
\begin{proof}
 We apply Lemma~\ref{Lemme : recurrent + rg1 faible -> lisse} to any vector tangent to $[\xi,\eta]$, whose projection in the quotient of $T^1\Omega$ by the group generated by $g$ is periodic, and hence forward recurrent.
\end{proof}

We now combine Lemma~\ref{Lemme : recurrent + rg1 faible -> lisse} with Poincar\'e's recurrence theorem (Fact~\ref{fait:ergo imp trans}.\ref{item:ergo imp trans:consR}) to establish the following.

\begin{prop}\label{Proposition : sse est plein}
 Let $\Omega\subset\PR(V)$ be a properly convex open set, and $\Gamma\subset\Aut(\Omega)$ a divergent discrete subgroup with $M=\Omega/\Gamma$ rank-one and non-elementary. Then any $\delta_\Gamma$-conformal density on $\partial\Omega$ gives full measure to $\partial_{\sse}\Omega$.
\end{prop}

\begin{proof}
 Let $o\in\Omega$ and $(\nu_x)_{x\in\Omega}$ a $\delta_\Gamma$-conformal density on $\partial\Omega$. Observe that $\partial_{\sse}\Omega$ is $\Gamma$-invariant and measurable, hence, by Observation~\ref{obs:obsconf}, in order to show that $\partial_{\sse}\Omega$ has full $\nu_o$-measure, it is enough to prove that $\nu_o(\partial_{\sse}\Omega)>0$. Let us assume by contradiction that $\nu_o(\partial_{\sse}\Omega)=0$. Let $m$ be an induced Sullivan measure on $T^1M$. By assumption, $m$ gives zero measure to the set of vectors in $T^1M$ who have a lift $v\in T^1\Omega$ such that $\phi_\infty v$ is smooth and strongly extremal. By Proposition~\ref{Proposition : convervativite} and Fact~\ref{fait:ergo imp trans}.\ref{item:ergo imp trans:consR}, $m$ gives full measure to the set of forward recurrent vectors. Let $K\subset T^1\Omega$ denote the closed set of vectors $v\in T^1\Omega$ such that $d_{\spl}(\phi_{-\infty}v,\phi_\infty v)=2$. By Lemma~\ref{Lemme : recurrent + rg1 faible -> lisse}, $\tilde m$ gives full measure to the set of vectors $v\in K$ with a forward recurrent projection in $T^1M$. Let $v\in T^1\Omega$ be a periodic rank-one vector; it is in the support of $\tilde m$ and in the open set $T^1\Omega\smallsetminus K$, so $\tilde m(T^1\Omega\smallsetminus K)>0$. This is a contradiction.
\end{proof}

\subsection{Smooth points and strong stable manifolds}

Let $\Omega\subset\PR(V)$ be a properly convex open set. The following fact describes the strong stable manifolds (Definition~\ref{Definition : W-invariant}) of the action of the geodesic flow on $T^1\Omega$.

\begin{fait}[\!{\!\cite[Cor.\,4.5]{topmixing}}] \label{strongstablemanifolds}
 Let $\Omega\subset \PR(V)$ be a properly convex open set. Let $v,w\in T^1\Omega$ be such that $\phi_\infty v=\phi_\infty w$ is a smooth point of $\partial\Omega$ and $b_{\phi_\infty v}(\pi v,\pi w)=0$ (see Figure~\ref{figure_strongstablemanifold}). Then $(d_{T^1\Omega}(\phi_tv,\phi_tw))_t$ converges non-increasingly to zero as $t$ goes to infinity, \ie $w\in W^{ss}(v)$.
\end{fait}

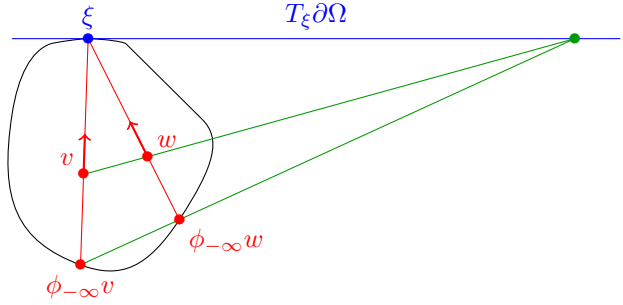
\begin{figure}
\centering
\begin{tikzpicture}
 \coordinate (xi) at (0,0);
 \coordinate (h) at (7,0);
 \coordinate (h') at (-1,0);
 \coordinate (phiv) at (-0.1,-3);
 \coordinate (phiw) at (1.2,-2.4);
 \coordinate (v) at ($(xi)!0.6!(phiv)$);
 \coordinate (v') at ($(v)!0.3!(xi)$);
 \coordinate (x) at (intersection of xi--h and phiv--phiw);
 \coordinate (w) at (intersection of xi--phiw and v--x);
 \coordinate (w') at ($(w)!0.3!(xi)$);
 \coordinate (A) at (0.5,-0.07);
 \coordinate (B) at (1,-0.57);
 \coordinate (C) at (1.5,-1.07);
 \coordinate (D) at (-1,-1);
 \coordinate (E) at (-0.5,-0.07);
  
 \length{xi,E,D,phiv,phiw,C,B,A}
 \cvx{xi,E,D,phiv,phiw,C,B,A}{1.5}
 
 \draw [blue] (h')--(h) node[midway,above,blue] {$T_\xi\partial\Omega$};
 \draw [red] (phiv)--(xi);
 \draw [red] (phiw)--(xi);
 \draw [red,thick,->] (v)--(v');
 \draw [red,thick,->] (w)--(w');
 \draw [green!60!black] (phiv)--(x);
 \draw [green!60!black] (v)--(x);
 \draw (x) node[green!60!black]{$\bullet$};
 \draw (xi) node[blue]{$\bullet$} node[above,blue]{$\xi$};
 \draw (v) node[red]{$\bullet$} node[above left,red]{$v$};
 \draw (w) node[red]{$\bullet$} node[above right,red]{$w$};
 \draw (phiv) node[red]{$\bullet$} node[below,red]{$\phi_{-\infty}v$};
 \draw (phiw) node[red]{$\bullet$} node[below right,red]{$\phi_{-\infty}w$};
\end{tikzpicture}
\caption{$v$ and $w$ in the same strong stable manifold}\label{figure_strongstablemanifold}
\end{figure}

Let us examine how Fact~\ref{strongstablemanifolds} can be rephrased using the Hopf parametrisation (Definition~\ref{Definition : Hopf}). Fix $o\in\Omega$. Let $(\xi,\eta,t)\in\Geod^\infty_{\hor}(\Omega)\times\R$. If $\pi_{\hor}(\eta)$ is a smooth point of $\partial\Omega$, then by Fact~\ref{strongstablemanifolds},
\begin{equation}\label{eq:varstab}
\Hopf_o(\xi',\eta,t)\in W^{ss}(\Hopf_o(\xi,\eta,t))
\end{equation}
for any $\xi'\in\partial_{\hor}\Omega$ such that $(\xi',\eta)\in\Geod^\infty_{\hor}(\Omega)$. The parametrisation of the unstable manifolds is more subtle. Using \eqref{Equation : Hopflip} and the fact that the unstable manifolds are the stable manifolds of the reversed flow, one can see that if $\pi_{\hor}\xi$ is smooth, then
\begin{equation}\label{eq:varinstab}
\Hopf_o(\xi,\eta',t+\rho^o_{\xi,\eta}(\eta'))\in W^{su}(\Hopf_o(\xi,\eta,t))
\end{equation}
for any $\eta'\in\partial_{\hor}\Omega$ such that $(\xi,\eta')\in\Geod_{\hor}(\Omega)$, where
\begin{equation}
 \rho^o_{\xi,\eta}(\eta') := 2\langle \xi,\eta' \rangle_o - 2\langle \xi,\eta \rangle_o.
\end{equation}

\subsection{A cross-ratio on the boundary}

Following the article \cite{birapport}, whose setting is that of negatively curved manifolds, let us define a cross-ratio, denoted by $B$, for four points on the boundary of a properly convex open set $\Omega\subset\PR(V)$. It should not be confused with the cross-ratio of four aligned points of the projective space, denoted with brackets, and used to defined the Hilbert metric $d_\Omega$ (see \eqref{Equation : Hilbert}). 

Recall that the cross-ratio of four points $\xi,\xi',\eta,\eta'\in\Omega$ is defined as\begin{equation}\label{Equation : birapport}
 B(\xi,\xi',\eta,\eta'):=\dd(\xi,\eta)+\dd(\xi',\eta')-\dd(\xi,\eta')-\dd(\xi',\eta).
\end{equation}
Fix $o\in\Omega$. One can check that
\begin{equation}\label{Equation : birapport 2}
 B(\xi,\xi',\eta,\eta') = \rho_{\xi,\eta}^o(\eta') + \rho^o_{\xi',\eta'}(\eta),
\end{equation}
and this implies that the $B$ extends continuously to the set of quadruples $(\xi,\xi',\eta,\eta')\in(\overline{\Omega}^{\hor})^4$ such that $(\xi,\eta)$, $(\xi',\eta)$, $(\xi,\eta')$ and $(\xi',\eta')$ belong to $\Geod_{\hor}\Omega$. See Figure~\ref{fig:birapport} for a geometric interpretation using horospheres (which also works with spheres when $\xi,\xi',\eta,\eta'$ belong to $\Omega$).

The next lemma is classical and relates special values of $B$, called \emph{periods}, with \eqref{longueur de translation}.

\begin{lemma}\label{birapport=longueur de translation}
 Let $\Omega\subset\PR(V)$ be a properly convex open set. Let $g\in\Aut(\Omega)$ be a biproximal automorphism whose axis intersect $\Omega$. Let $\xi\in\overline{\Omega}{}^{\hor}$ be such that $(\xi,x_g^+)$ and $(\xi,x_g^-)$ are in $\Geod_{\hor}\Omega$ (recall that $x_g^\pm$ is a smooth point of $\partial\Omega$ by Fact~\ref{les extremites des geodesiques periodiques biproximales sont lisses}, so it identifies with its preimage in $\partial_{\hor}\Omega$). Then $B(x_g^+,x_g^-,\xi,g\xi)=2\ell(g)$.
\end{lemma}

\begin{proof}
 By continuity of the cross-ratio, we can assume that $\xi\in\Omega$. Then using \eqref{Equation : birapport} or \eqref{Equation : birapport 2}, one can show that
\begin{equation*}
B(x_g^+,x_g^-,\xi,g\xi) = b_{x_g^+}(\xi,g\xi)+b_{x_g^-}(g\xi,\xi).
\end{equation*}
Now consider $x\in\Omega$ on the axis of $g$. We compute
\begin{align*}
b_{x_g^+}(\xi,g\xi) & = b_{x_g^+}(\xi,x)+b_{x_g^+}(x,gx)+b_{x_g^+}(gx,g\xi)\\
& = d_\Omega(x,gx)+b_{x_g^+}(\xi,x)+b_{g^{-1}x_g^+}(x,\xi)\quad \text{(by Lemma~\ref{b=d})}\\ 
& = \ell(g).
\end{align*}
By taking the inverse we get:
\begin{equation*}
b_{x_g^-}(g\xi,\xi)=b_{x_{g^{-1}}^+}(g\xi,g^{-1}g\xi)=\ell(g^{-1})=\ell(g).\qedhere
\end{equation*}
\end{proof}

\subsection{\texorpdfstring{$W^{ss}$}{Wss} and \texorpdfstring{$W^{su}$}{Wsu}-invariant functions are essentially constant}\label{ssection:ergo+melange}

In this section we prove that the $W^{ss}$ and $W^{su}$-invariant functions on $T^1M$ are essentially constant, and we derive as a corollary that the flow is ergodic and mixing. We will need the following result, about the local non-arithmeticity of the length spectrum; it is similar to \cite[Prop.\,4.1]{topmixing} but does not need the assumption of strong irreducibility.
\begin{fait}[\!{\!\cite{EeERfH+}}]\label{non arithmeticite locale}
  Let $\Omega\subset \PR(V)$ be a properly convex open set and $\Gamma\subset\Aut(\Omega)$ a discrete subgroup with $M=\Omega/\Gamma$ rank-one and non-elementary. Then the set of translation lengths $\ell(\gamma)$ of rank-one elements $\gamma\in\Gamma$ generates a dense subgroup of $\R$.
\end{fait}

\begin{prop}\label{Proposition : invariant sur les varstab implique constant}
 Let $\Omega\subset\PR(V)$ be a properly convex open set and $\Gamma\subset\Aut(\Omega)$ a discrete subgroup with $M=\Omega/\Gamma$ divergent rank-one and non-elementary. Let $m$ be the Sullivan measure on $T^1M$ induced by a $\delta_\Gamma$-conformal density $(\nu_x)_{x\in\Omega}$. Then any $W^{ss}$ and $W^{su}$-invariant function on $T^1M$ is essentially constant with respect to $m$.
\end{prop}

\begin{proof}
 We fix $o\in\Omega$, and may assume that $\nu_o(\partial\Omega)=1$.
 Let $f$ be a $W^{ss}$ and $W^{su}$-invariant function on $T^1M$. By Proposition~\ref{Proposition : sse est plein}, the measure $m$ gives full measure to the (measurable) set $T^1M_{\sse}\subset T^1M$ of vectors $v$ with $\phi_{\infty}v,\phi_{-\infty}v\in\partial_{\sse}\Omega$, hence it suffices to show that the restriction $f_{|T^1M_{\sse}}$ is essentially constant. We lift $f_{|T^1M_{\sse}}$, via the Hopf parametrisation, to a function $\tilde f$ on $\{(\xi,\eta)\in\partial_{\sse}\Omega :\xi\neq\eta\}\times\R$ that we extend to $\partial_{\sse}\Omega^2\times\R$ by setting $\tilde f(\xi,\xi)=0$ for any $\xi$ (recall that we abusively identify $\partial_{\sse}\Omega$ with its preimage in $\partial_{\hor}\Omega$). Let us show that $\tilde f$ is $\nu_o^2$ times Lebesgue-essentially constant. Recall that $\nu_o$ is non-atomic by Propositions~\ref{Proposition : faces coniques negligees} and \ref{l'ensemble limite conique a mesure non nulle}, thus for $\nu_o$-almost all $\xi,\eta\in\partial_{\sse}\Omega$ we have $\xi\neq\eta$.
 
 The function $f$ is $W^{ss}$ and $W^{su}$-invariant, so by \eqref{eq:varstab} and \eqref{eq:varinstab}, for $\nu_o$-almost $\xi,\xi',\eta,\eta'\in\partial_{\sse}\Omega$ and Lebesgue-almost any $t\in\R$, the quantity $\rho_{\xi,\eta}(\eta'))$ is well-defined and
 \[\ff(\xi,\eta,t)=\ff(\xi',\eta,t)=\ff(\xi,\eta',t+\rho_{\xi,\eta}(\eta')).\]
 This implies in particular that there exist $\xi_0\neq\eta_0\in\partial_{\sse}\Omega$ such that, if we denote $g(t):=\tilde f(\xi_0,\eta_0,t)$ for any $t\in\R$, then for $\nu_o$-almost all $\xi,\eta\in\partial_{\sse}\Omega$ and Lebesgue-almost any $t\in\R$, the quantity $\rho_{\xi_0,\eta_0}(\eta)$ is well-defined and
 \[\ff(\xi,\eta,t+\rho_{\xi_0,\eta_0}(\eta))=g(t).\]
 
 It is enough to establish that the measurable function $g$ is essentially constant with respect to the Lebesgue measure. We denote by $H$ the additive real subgroup consisting of numbers $\tau$ such that $g(t+\tau)=g(t)$ for Lebesgue-almost every $t\in\R$. A classical result says that $H$ is a closed subgroup of $\R$. (To see this first reduce to the case where $g$ is bounded and then note that $H$ is the stabiliser of $g(t)\diff t$ for the continuous action of $\R$ on the space of Radon measures on $\R$.) To finish the proof of Proposition~\ref{Proposition : invariant sur les varstab implique constant} it is enough, according to Fact~\ref{non arithmeticite locale}, to prove that $2\ell(\gamma)\in H$ for any rank-one element $\gamma\in\Gamma$. To this end we observe that many cross-ratios belong to $H$: for $\nu_o$-almost all $\xi,\eta,\xi',\eta'\in\partial_{\sse}\Omega$ and Lebesgue-almost every $t\in\R$, the quantities $\rho_{\xi_0,\eta_0}(\eta)$, $\rho_{\xi,\eta}(\eta')$, $\rho_{\xi',\eta'}(\eta)$ and $B(\xi,\xi',\eta,\eta')$ are well-defined and we have the following serie of equalities (see Figure~\ref{fig:birapport} for a geometric interpretation).
 \begin{align}\label{eq:birapport}
  g(t) & = \tilde{f}(\xi,\eta,t+\rho_{\xi_0,\eta_0}(\eta)) \nonumber \\
  & = \tilde{f}(\xi,\eta',t+\rho_{\xi_0,\eta_0}(\eta)+\rho_{\xi,\eta}(\eta')) \nonumber \\
  & = \tilde{f}(\xi',\eta',t+\rho_{\xi_0,\eta_0}(\eta)+\rho_{\xi,\eta}(\eta')) \nonumber \\
  & = \tilde{f}(\xi',\eta,t+\rho_{\xi_0,\eta_0}(\eta)+\rho_{\xi,\eta}(\eta')+\rho_{\xi',\eta'}(\eta)) \nonumber \\
  & = \tilde{f}(\xi,\eta,t+\rho_{\xi_0,\eta_0}(\eta)+B(\xi,\xi',\eta,\eta')) \nonumber \\
  & = g(t+B(\xi,\xi',\eta,\eta')).
 \end{align}
 Consider a rank-one element $\gamma\in\Gamma$ and a point $\xi\in\partial_{\sse}\Omega\cap\supp(\nu_o)\smallsetminus\{x_g^-,x_g^+\}$. Then $2\ell(\gamma)$ is equal to $B(x_\gamma^+,x_\gamma^-,\xi,\gamma\xi)$ by Lemma~\ref{birapport=longueur de translation}, and furthermore this quantity belongs to $H$ because $H$ is closed, $B$ is continuous and $x_\gamma^+,x_\gamma^-,\xi,\gamma\xi$ are in $\supp(\mu_o)$.
 
 \begin{figure}
  \centering
  \begin{tikzpicture}[scale=3]

   \coordinate (xi) at (0,0);
   \coordinate (eta') at (1,.1);
   \coordinate (xi') at (.9,1.1);
   \coordinate (eta) at (0,1);
   
   \coordinate (a) at ($(xi)!.3!(eta)$);
   \coordinate (b) at ($(xi)!.4!(eta')$);
   \coordinate (c) at ($(eta')!.6!(xi')$);
   \coordinate (d) at ($(xi')!.35!(eta)$);
   \coordinate (e) at ($(eta)!.5!(xi)$);
   
   \draw (xi)--(eta')--(xi')--(eta)--(xi);
   
   \coordinate (Txi) at (.4,-.2);
   \coordinate (Teta') at (.3,.3);
   \coordinate (Txi') at (-.37,.25);
   \coordinate (Teta) at (-.19,-.4);
   \coordinate (Ta) at (-.2,-.05);
   \coordinate (Tb) at (-.01,.2);
   \coordinate (Tc) at (-.24,.01);
   \coordinate (Td) at (.03,-.2);
   \coordinate (Te) at (.2,.01);
   
   \draw (xi) node{$\bullet$} node[below left]{$\xi$};
   \draw (eta') node{$\bullet$} node[below right]{$\eta'$};
   \draw (xi') node{$\bullet$} node[above right]{$\xi'$};
   \draw (eta) node{$\bullet$} node[above left]{$\eta$};
   \draw (a) node{$\bullet$} node[below right]{$x$};
   \draw (b) node{$\bullet$};
   \draw (c) node{$\bullet$};
   \draw (d) node{$\bullet$};
   \draw (e) node{$\bullet$} node[above right]{$y$};
   
    \draw [thick, blue!70!black , ->] (a)--++($(0,0)-(Ta)$);
   
   \draw (xi)..controls ($(xi)+(Txi)$) and ($(eta')-(Teta')$)..(eta');
   \draw (eta')..controls ($(eta')+(Teta')$) and ($(xi')-(Txi')$)..(xi');
   \draw (xi')..controls ($(xi')+(Txi')$) and ($(eta)-(Teta)$)..(eta);
   \draw (eta)..controls ($(eta)+(Teta)$) and ($(xi)-(Txi)$)..(xi);
   
   \draw [blue!70!black](xi)..controls ($(xi)+.5*(Txi)$) and ($(b)-(Tb)$)..(b);
   \draw [very thick,blue!70!black] (b)..controls ($(b)+(Tb)$) and ($(a)-(Ta)$)..(a);
   \draw [blue!70!black] (a)..controls ($(a)+(Ta)$) and ($(xi)-.5*(Txi)$)..(xi);
   \draw [blue!70!black] (eta')..controls ($(eta')+.5*(Teta')$) and ($(c)-(Tc)$)..(c);
   \draw [very thick,blue!70!black] (c)..controls ($(c)+(Tc)$) and ($(b)+(Tb)$)..(b);
   \draw [blue!70!black] (b)..controls ($(b)-(Tb)$) and ($(eta')-.5*(Teta')$)..(eta');
   \draw [blue!70!black] (xi')..controls ($(xi')+.5*(Txi')$) and ($(d)-(Td)$)..(d);
   \draw [very thick,blue!70!black] (d)..controls ($(d)+(Td)$) and ($(c)+(Tc)$)..(c);
   \draw [blue!70!black] (c)..controls ($(c)-(Tc)$) and ($(xi')-.5*(Txi')$)..(xi');
   \draw [blue!70!black] (eta)..controls ($(eta)+.5*(Teta)$) and ($(e)-(Te)$)..(e);
   \draw [very thick,blue!70!black] (e)..controls ($(e)+(Te)$) and ($(d)+(Td)$)..(d);
   \draw [blue!70!black] (d)..controls ($(d)-(Td)$) and ($(eta)-.5*(Teta)$)..(eta);
  \end{tikzpicture}
  \caption{Illustration of computations \eqref{eq:birapport}, \\ where $d_\Omega(x,y)=B(\xi,\xi',\eta,\eta')$}
  \label{fig:birapport}
 \end{figure}
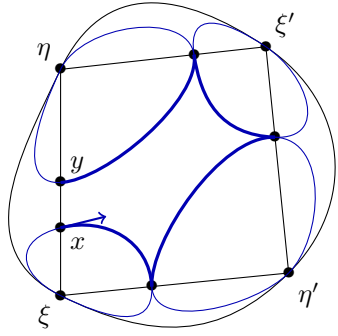

\end{proof}

\begin{cor}\label{Corollaire : ergo et melange}
 In the setting of Proposition~\ref{Proposition : invariant sur les varstab implique constant}, if $\Gamma$ is divergent, then $m$ is ergodic under the action of the geodesic flow. Furthermore, the $\Gamma$-quasi-invariant measures $\nu_o^2$ and $\nu_o$ on $\partial\Omega^2$ and $\partial\Omega$ are ergodic under the action of $\Gamma$. In particular the $\delta_\Gamma$-conformal density is unique up to a scalar multiple.
 
 If moreover $m$ is finite, then it is mixing.
\end{cor}

\begin{proof}
 The ergodicity of $m$ is a direct consequence of Fact~\ref{Fait : critere d'ergodicite}, Proposition~\ref{Proposition : convervativite} and Proposition~\ref{Proposition : invariant sur les varstab implique constant}, which can be applied since $m$ is Radon. 
 
 The ergodicity of $\nu_o^2$ is a direct consequence of that of $m$, since there is correspondence between $\Gamma$-invariant measurable subsets of $\Geod^\infty\!\Omega$ (which has full $\nu_o^2$-measure) and $(\phi_t)_t$-invariant measurable subsets of $T^1M$, sending sets with full (\resp null) measure to sets with full (\resp null) measure. For similar reasons, the ergodicity of $\nu_o$ is a direct consequence of that of $\nu_o^2$.
 
 Let $(\nu'_x)_{x\in\Omega}$ be another $\delta_\Gamma$-conformal density such that $\nu'_o(\partial\Omega)=\nu_o(\partial\Omega)$. Then the family $(\nu_x+\nu_x')_{x\in\Omega}$ is also a conformal density. The measures $\nu_o$ and $\nu_o'$ are absolutely continuous with respect $\nu_o+\nu_o'$; the Radon--Nikodym derivatives are $\Gamma$-invariant, by definition of conformal densities, and hence constant $(\nu_o+\nu_o')$-almost surely by ergodicity. Thus $\nu_o=\nu_o'$.
 
 If $m$ is finite, then it is mixing by Fact~\ref{coudene} and Proposition~\ref{Proposition : invariant sur les varstab implique constant}.
\end{proof}

\subsection{The support of conformal densities}\label{Section : supp(conf)}

Let us establish that the support of the $\delta_\Gamma$-conformal density, for $\Gamma$ divergent and $M=\Omega/\Gamma$ rank-one, is exactly the proximal limit set. This is a consequence of conservativity and ergodicity of the Bowen--Margulis measure, Fact~\ref{fait:ergo imp trans}.\ref{item:ergo imp trans:ergoR}, and the fact that $T^1M_{\bip}$ is maximal among the flow-invariant closed subsets of $T^1M$ on which the geodesic flow is forward topologically transitive (\ie has a dense forward orbit); this last result was proved in \cite{topmixing}.

\begin{prop}\label{Myrberg est plein}
 Let $\Omega\subset\PR(V)$ be a properly convex open set and $\Gamma\subset\Aut(\Omega)$ a divergent discrete subgroup with $M=\Omega/\Gamma$ rank-one and non-elementary. Then $\Lambda^{\prox}$ is the support of any $\delta_\Gamma$-confomal density and $T^1M_{\bip}$ is the support of any Bowen--Margulis measure.
\end{prop}

\begin{proof}
 Let $(\nu_x)_{x\in\Omega}$ be the $\delta_\Gamma$-conformal density on $\partial\Omega$, and $o\in\Omega$. Since $\Lambda^{\prox}$ is the smallest $\Gamma$-invariant closed subset of $\overline\Omega$ (Fact~\ref{fait:des rk1 partout}) and $\nu_o$ is $\Gamma$-quasi-invariant, we have $\Lambda^{\prox}\subset\supp(\nu_o)$ and $T^1M_{\bip}\subset\supp(m)$. Since $m$ is conservative and ergodic by Proposition~\ref{Proposition : convervativite} and Corollary~\ref{Corollaire : ergo et melange}, Fact~\ref{fait:ergo imp trans}.\ref{item:ergo imp trans:ergoR} implies that the geodesic flow is forward topologically transitive on $\supp(m)$ (\ie has a dense forward orbit). According to \cite[Th.\,1.2]{topmixing}, this implies that $\supp(m)=T^1M_{\bip}$.
 
 Suppose by contradiction there is a point $\xi\in\supp(\nu_o)\smallsetminus\Lambda^{\prox}$. Since $M$ is rank-one, we can find a strongly extremal point $\eta\in\Lambda^{\prox}$. Then we can find a neighbourhood $U\times V\subset\Geod^\infty(\Omega)$ of $(\xi,\eta)$ such that $U\cap\Lambda^{\prox}$ is empty. Note that $\nu_o(U)\nu_o(V)>0$ since $\xi,\eta\in\supp(\nu_o)$. Hence $m$ gives non-zero measure to the set of vectors $v$ with lifts $\tilde{v}\in T^1\Omega$ such that $(\phi_{-\infty}v,\phi_\infty v)\in U\times V$; but this set is not contained in $T^1M_{\bip}$: this is a contradiction.
\end{proof}

\subsection{Proof of Theorem~\ref{Thm : The weak Hopf--Tsuji--Sullivan--Roblin dichotomy}.\ref{Item : cas divergent}}

Theorem~\ref{Thm : The weak Hopf--Tsuji--Sullivan--Roblin dichotomy}.\ref{Item : cas divergent} is a direct consequence of Propositions~\ref{delta<deltaGamma}, \ref{Proposition : faces coniques negligees}, \ref{Proposition : convervativite}, \ref{Proposition : sse est plein} and \ref{Myrberg est plein}, Corollary~\ref{Corollaire : ergo et melange} and Fact~\ref{Fait : conservatif passe au quotient}. Observe that $(\partial\Omega^2,\nu_o^2,\Gamma)$ being conservative and ergodic is due to the fact that $\Geod^\infty(\Omega)\subset\partial\Omega^2$ has full $\nu_o^2$-measure, since $\nu_o$ is non-atomic and gives full measure to the set of strongly extremal points

\section{Preparatory results on convex cocompact projective manifolds}\label{Preparatifs pour la partie Knieper}

In real hyperbolic geometry, many properties of compact manifolds are actually also true for a broader class of manifolds: the convex cocompact ones, and the proofs of these properties generally work verbatim. This observation remains valid in convex projective geometry, thanks to the recent notion of convex cocompactness developped by Danciger, Gu\'eritaud and Kassel in this setting (see Definition~\ref{Def : convexe cocompacite}). Note that most of the results of the remainder of the paper concern convex cocompact actions. The present section gathers results on convex cocompact actions, and in particular more precise Shadow lemmas. An important result of Danciger--Gu\'eritaud--Kassel that we will need is the following.
\begin{fait}[\!{\!\cite[Prop.\,5.10]{fannycvxcocpct}}]\label{Fait : fannycvxcocpct}
 Let $\Omega\in\PR(V)$ be a properly convex open set and $\Gamma\subset\PGL(V)$ a discrete subgroup. Suppose that $\Gamma$ acts convex cocompactly on $\Omega$ and $\Omega^*$. Then $\overline\Omega\smallsetminus\Lambda^{\orb}_\Omega(\Gamma)$  \emph{has bisaturated boundary}, in the sense that $[\xi,\eta]\cap\Omega$ is non-empty for all $\xi\in\Lambda^{\orb}$ and $\eta\in\overline\Omega\smallsetminus\Lambda^{\orb}$.
\end{fait}

Observe that, given a strongly irreducible discrete subgroup $\Gamma\subset\PGL(V)$ that acts convex cocompactly on a properly convex open $\Omega\subset\PR(V)$, if we want to understand the dynamics of the geodesic flow on $(T^1\Omega/\Gamma)_{\core}$, then we can assume that $\Gamma$ also acts convex cocompactly on $\Omega^*$, since otherwise we may take another $\Gamma$-invariant properly convex open set $\Omega'$ such that $\Gamma$ acts convex cocompactly on both $\Omega'$ and $(\Omega')^*$, and we have $(T^1\Omega/\Gamma)_{\core}=(T^1\Omega'/\Gamma)_{\core}$ (see \cite[Prop.\,5.10 \& Cor.\,4.17]{fannycvxcocpct}).

\subsection{Proof of Proposition~\ref{Proposition : cvxcocpct -> divergent}}\label{divergence}

Consider $o\in\CC^{\core}_\Omega(\Gamma)$. We set $\Vol:=\sum_{\gamma\in\Gamma}\mathcal{D}_{\gamma o}$, where $\mathcal{D}_x$ denotes the Dirac mass at $x\in\Omega$; this defines a $\Gamma$-invariant Radon measure on $\Omega$, which is supported on $\mathcal{C}^{\core}_\Omega(\Gamma)$. We apply Fact~\ref{smear} to $\Vol$, in order to obtain a $\delta_\Gamma$-conformal density $\nu_o$ on $\partial\Omega$ which is supported on $\Lambda^{\orb}$. Recall that $\Lambda^{\orb}=\Lambda^{\con}$ (see Section~\ref{Ssection : intro MME}), hence $\nu_o(\Lambda^{\con})=1$ and therefore $\Gamma$ is divergent by Theorem~\ref{Thm : The weak Hopf--Tsuji--Sullivan--Roblin dichotomy}. By the same theorem, the Bowen--Margulis measure $m$ is supported on $T^1M_{\core}$. Since $m$ is Radon and $T^1M_{\core}$ is compact, $m$ must be finite.

\subsection{A Shadow lemma for convex cocompact actions}\label{Sous-section : Lemme de l'ombre convex cocompact}

In this section we prove a Shadow lemma for when $\Gamma$ satisfies the stronger assumption that it acts convex cocompactly on $\Omega$. The main interest of the new Shadow lemma is that it estimates the measure of \emph{small} shadows.

 We denote $\mathrm{ray}_x(A)=\bigcup_{\xi\in A}[x,\xi)$ for $x\in \Omega$ and $A\subset\partial \Omega$.

\begin{lemma}\label{shadowlemma+ sans cvxcpcpt}
Let $\Omega\subset\PR(V)$ be a properly convex open set and $\Gamma\subset\Aut(\Omega)$ a discrete subgroup. Suppose that $\Gamma$ is strongly irreducible and $T^1M_{\bip}$ is non-empty, or that $M$ is rank-one and non-elementary. Consider $\delta\geq 0$ and a $\delta$-conformal density $(\nu_x)_{x\in\Omega}$ on~$\partial\Omega$. Let $K\subset\Omega$ be a compact subset and $\mathcal{C}=\Gamma\cdot K$. Then there exists $R_0>0$ such that for any $R>0$, we can find a constant $C=C(R)>0$ such that for all $x,y\in\mathcal{C}$,
\[C^{-1}e^{-\delta d_\Omega(x,y)}\leq \nu_x(\mathcal{O}_{R_0+R}(x,y))\leq \nu_x(\mathcal{O}^+_{R_0+R}(x,y))\leq Ce^{-\delta d_\Omega(x,y)},\]
and if furthermore $y\in\mathrm{ray}_x(\supp(\nu_o))$, then
\[\nu_y(\mathcal{O}_R(x,y))\geq C^{-1} \ \text{and} \ \nu_x(\mathcal{O}_R(x,y))\geq C^{-1}e^{-\delta d_\Omega(x,y)}.\]
\end{lemma}

\begin{proof}
 Let $R>0$. By definition of the conformal density, $\nu_x\leq e^{\delta d_\Omega(x,y)}\nu_y$ for any $x,y\in\Omega$. Using the triangular inequality and Fact~\ref{crampon}, it is elementary to see that for all $x,y,x',y'\in\Omega$,
 \[\mathcal{O}^\alpha_R(x,y)\subset \mathcal{O}^\alpha_{R+d_\Omega(x,x')+d_\Omega(y,y')}(x',y') \text{ for } \alpha=\emptyset \text{ or } +.\]
 By $\Gamma$-equivariance it is enough to prove the lemma for $x\in K$. Let $D$ be the diameter of $K$ for the Hilbert metric, and fix $o\in K$. For any $y\in\mathcal{C}$, there exists $\gamma\in\Gamma$ such that $d_\Omega(y,\gamma o)\leq D$, and then for any $x\in K$, 
 \[e^{-\delta_\Gamma D}\nu_o(\mathcal{O}_{R}(o,\gamma o))\leq \nu_x(\mathcal{O}_{R+2D}(x,y))\leq \nu_x(\mathcal{O}^+_{R+2D}(x,y))\leq e^{\delta_\Gamma D}\nu_o(\mathcal{O}^+_{R+4D}(o,\gamma o)).\]
 Now we can use the Shadow lemma (Lemma~\ref{shadowlemma}) to bound the right-most and left-most terms when $R$ is greater than some $R_0$, and we obtain the first estimate.

 Let us show the second estimate. We set
 \[\epsilon:=\inf\{\nu_x(\mathcal{O}_R(y,x')) : x\in K,\ y\in\Omega,\ x'\in K\cap\mathrm{ray}_y(\supp\nu_o)\}>0.\]
 We then make the same computations as in the Shadow lemma (Lemma~\ref{shadowlemma}). We take a $\delta$-conformal density $(\mu_x)_{x\in\Omega}$ on $\partial_{\hor}\Omega$ such that $(\pi_{hor})_*\mu_x=\nu_x$ for any $x\in\Omega$. Let $x\in K$ and $y\in\mathrm{ray}_x(\supp(\nu_o))\cap\mathcal{C}$; by definition of $K$ there exists $\gamma\in\Gamma$ such that $\gamma^{-1}y\in K$. Then

 \begin{align*}
  \nu_x(\mathcal{O}_R(x,y)) &= \mu_x(\pi_{hor}^{-1}(\mathcal{O}_R(x,y)))\\
  &= \mu_{\gamma^{-1}x}(\pi_{hor}^{-1}(\mathcal{O}_R(\gamma^{-1}x,\gamma^{-1}y)))\\
  &= \int_{\pi_{hor}^{-1}(\mathcal{O}_R(\gamma^{-1}x,\gamma^{-1}y))}e^{-\delta b_{\xxi}(\gamma^{-1}x,\gamma^{-1}y)}d\mu_x(\xxi) \\
  &\geq \int_{\pi_{hor}^{-1}(\mathcal{O}_R(\gamma^{-1}x,\gamma^{-1}y))}e^{-\delta d_\Omega(x,y)}d\mu_x(\xxi)\\
  &\geq \nu_x(\mathcal{O}_R(\gamma^{-1}x,\gamma^{-1}y)) e^{-\delta d_\Omega(x,y)}\\
  &\geq \epsilon e^{-\delta d_\Omega(x,y)}.\qedhere
 \end{align*}
\end{proof}

As an immediate corollary, we obtain the following.

\begin{cor}\label{shadowlemma+}
Let $\Omega\subset\PR(V)$ be a properly convex open set and $\Gamma\subset\Aut(\Omega)$ a convex cocompact and discrete subgroup. Suppose that $\Gamma$ is strongly irreducible and $T^1M_{\bip}$ is non-empty, or that $M$ is rank-one and non-elementary. Consider $\delta\geq 0$ and a $\delta$-conformal density $(\nu_x)_{x\in\Omega}$ on $\partial\Omega$. Then there exists $R_0>0$ such that for any $R>0$, we can find a constant $C=C(R)>0$ such that for all $x,y\in\mathcal{C}^{\core}$,
\[C^{-1}e^{-\delta d_\Omega(x,y)}\leq \nu_x(\mathcal{O}_{R_0+R}(x,y))\leq Ce^{-\delta d_\Omega(x,y)},\]
and if furthermore $y\in\mathrm{ray}_x(\supp(\nu_o))$, then
\[\nu_y(\mathcal{O}_R(x,y))\geq C^{-1} \ \text{and} \ \nu_x(\mathcal{O}_R(x,y))\geq C^{-1}e^{-\delta d_\Omega(x,y)}.\]
\end{cor}

\subsection{Non-straight closed geodesics}

In this section we investigate non-rank-one automorphisms of properly convex open sets which realise their translation length.

\subsubsection{Some standard facts}

The following fact says that the only case where a non-straight geodesic can appear is the one shown in Figure~\ref{figure_distance}.
\begin{fait}\label{fait:geodcassee}
 Let $\Omega\subset\PR(V)$ be a properly convex open set. For $x,y\in\Omega$ distinct, consider $(a_{xy},b_{xy})$ in $\partial\Omega^2$ such that $a_{xy},x,y,b_{xy}$ are aligned in this order. Let $x,y,z\in\Omega$ be pairwise distinct. Then $d_{\Omega}(x,z)=d_\Omega(x,y)+d_\Omega(y,z)$ if and only if $a_{xz}\in[a_{xy},a_{yz}]$ and $b_{xz}\in[b_{xy},b_{yz}]$.
\end{fait}

This property can be used to prove the following.
\begin{fait}[\!{\!\cite{foertsch2005hilbert}}]\label{fait:limgeodcassee}
 Let $\Omega\subset\PR(V)$ be a properly convex open set, $I\subset\R$ a non-trivial interval and $c:I\rightarrow\Omega$ an isometric embedding. For all $t<s\in I$, consider $(a_{ts},b_{ts})$ in $\partial\Omega^2$ such that $a_{ts},c(t),c(s),b_{ts}$ are aligned in this order. Let $F_+$ (\resp $F_-$) be the smallest closed face of $\Omega$ that contains $\{b_{ts} : t<s\in I\}$ (\resp $\{a_{ts} : t<s\in I\}$). Then $F_+$ and $F_-$ are proper faces of $\Omega$, whose dimension is the dimension of the convex hull of $c(I)$ minus $1$.
 
 Moreover, if $\sup I=\infty$ (\resp $\inf I=-\infty$), then $(c(t))_t$ converges to a point of $F_+$ (\resp $F_-$) when $t$ goes to $+\infty$ (\resp $-\infty$).
\end{fait}

The following generalises the fact that $x_g^+\in\partial\Omega$ and $x_g^+\oplus x_g^0\cap\Omega=\emptyset$ for any biproximal automorphism $g$ of a properly convex open set $\Omega$.

\begin{fait}\label{fait:spec gap}
 Let $\Omega\subset\PR(V)$ be a properly convex open set. Let $g\in\PR(\End(V))$ be in the closure of $\Aut(\Omega)$, then its kernel and its image intersect $\overline\Omega$ but not $\Omega$.
\end{fait}

\subsubsection{Analysis of an automorphism that attains its translation length}

\begin{lemma}\label{lem:realell}
 Let $\Omega$ be a properly convex open set. Let $g\in\Aut(\Omega)$  with $\ell(g)>0$ and suppose there exists $x\in\Omega$ such that $d_\Omega(x,gx)=\ell(g)$. Consider $a,b\in\partial\Omega$ such that $a,x,gx,b$ are aligned in this order. Then
 \begin{enumerate}[label=(\arabic*)]
 \item \label{item:realellgeod} the path $c:\R\rightarrow \Omega$, defined by $c(t)\in [g^nx,g^{n+1}x]$ and $d_\Omega(c(t),g^nx)=t-n$ for all $n\in\Z$ and $t\in[n,n+1]$, is a geodesic;
 \item \label{item:realelldiag} the restriction of $g$ to $x_g^+\oplus x_g^-$ is diagonalisable over $\C$;
 \item \label{item:realellbip} the restriction of $g$ to the span $\PR(W)$ of $\{g^nx\}_n$ is biproximal;
 \item \label{item:realellface} $x_g^+\oplus x_g^0\cap\PR(W)\cap\overline\Omega$ (\resp $x_g^-\oplus x_g^0\cap\PR(W)\cap\overline\Omega$) is the smallest $g$-invariant closed face of $\Omega$ that contains $b$ (\resp $a$);
 \item \label{item:realellpasaligne} if $x,gx,g^2x$ are not aligned, then $x_g^0\cap\overline\Omega$ is non-empty;
 \item \label{item:realellpasr1} if $g$ is not rank-one, then $d_{\spl}(a,b)=2$.
 \end{enumerate}
\end{lemma}
\begin{proof}
 Let us check that \ref{item:realellgeod} holds. Consider three real numbers $r\leq s\leq t$, pick three integers $k\leq n\leq m$ such that $k\leq r\leq k+1$ and $n\leq s\leq n+1$ and $m\leq t\leq m+1$. By \eqref{longueur de translation}, we have 
 \begin{align*}
  d_\Omega(c(r),c(t)) & \geq d_\Omega(g^kx,g^{m+1}x) - d_\Omega(g^kx,c(r)) -d_\Omega(c(t),g^{m+1}x) \\
  & \geq \ell(g^{m+1-k})- d_\Omega(g^kx,c(r)) -d_\Omega(c(t),g^{m+1}x)\\
  & = \sum_{i=k}^m d_\Omega(g^ix,g^{i+1}x)- d_\Omega(g^kx,c(r)) -d_\Omega(c(t),g^{m+1}x) \\
  &  = d_\Omega(c(r),g^{k+1}x) + \sum_{i=k+1}^{m-1}d_\Omega(g^ix,g^{i+1}x) + d_\Omega(g^mx,c(t))\\
  & \geq d_\Omega(c(r),c(s))+d_\Omega(c(s),c(t)).
 \end{align*}

 For all distinct pair of points $(y,z)\in\overline\Omega^2$, let us denote by $\stereo_y^z\in\partial\Omega$ the point such that $y,z,\stereo_y^z$ are aligned in this order. Let $\PR(W)\subset \PR(V)$ be the span of $\{g^nx\}_{n\in\Z}$, with dimension $k\geq 1$; we set $\Omega'=\Omega\cap\PR(W)$. By Fact~\ref{fait:limgeodcassee}, the smallest closed face of $\Omega$ that contains $\{g^nb\}_{n\in\Z}$ (\resp $\{g^na\}_{n\in\Z}$), denoted by $F_+$ (\resp $F_-$), is proper, and therefore its span $\PR(W_+)$ (\resp $\PR(W_-)$) has dimension $k-1$. Moreover $\stereo_{g^nx}^{g^mx}\in F_+$ (\resp $F_-$) for all $m>n$ (\resp $m<n$), and $(g^nx)_n$ converges to some point $\xi_+\in F_+\cap x_g^+$ (\resp $\xi_-\in F_-\cap x_g^-$) which is fixed by $g$, when $n$ goes to $+\infty$ (\resp $-\infty$).
 
 Consider a lift $\tilde g\in\GL(V)$ of $g$ that preserves one connected component $C\subset V\smallsetminus\{0\}$ of the preimage of $\Omega$, and such that $\lambda_1(\tilde g)=1$. Let us examine the Jordan normal form of $\tilde g$: there exists $\ell\geq0$ and a decomposition
 \begin{equation}\label{eq:decomp g}
 \tilde g=u_1+\dots+u_\alpha - (u_1'+\dots+u'_{\alpha'}) + r_1^{\theta_1}v_1+\dots+r_\beta^{\theta_\beta}v_\beta + h,
 \end{equation}
 which satisfies the following. The product of any two matrices in this sum is zero. The integers $\alpha,\beta\geq 0$ are not both zero. The sequence $(\frac{1}{n^\ell}h^n)_n$ tends to zero.  The matrices $u_1,\dots,u_\alpha,u_1',\dots,u'_{\alpha'}$ are all conjugate to the matrix with zeros everywhere except on the upper-left block of size $\ell+1$, which is the exponential of the upper-triangular matrix whose $(i,j)$-entry is $1$ if $j=i+1$ and zero otherwise. For $\theta\in \R$ and $1\leq i\leq \beta$, the matrices $r_i^\theta$ and $v_i$ are simultaneously conjugate to the matrices with zeros everywhere except on the upper-left block of size $2\ell+2$, where they are respectively of the form
 \[\begin{pmatrix}
   \mathrm{rot}^\theta & & \\
   & \ddots & \\
   & & \mathrm{rot}^\theta \\
  \end{pmatrix}
   \ \text{and} \ \exp
  \begin{pmatrix}
   0 & I_2 & & \\
   & \ddots & \ddots & \\
   & & \ddots & I_2 & \\
   & & &0 \\
  \end{pmatrix}, \ \text{with} \quad \mathrm{rot}^\theta=\begin{pmatrix} \cos\theta & \sin\theta \\ -\sin\theta & \cos\theta \end{pmatrix} \ \text{and} \ I_2=\mathrm{rot}^0.\]
 Finally $\theta_1,\dots,\theta_\beta\in\R\smallsetminus \pi\Z$. Let $\bar u_i=\lim_{n\to\infty}\frac{\ell!}{n^\ell}u_i^n$ and $\bar u_i'=\lim_{n\to\infty}\frac{\ell!}{n^\ell}u_i'{}^n$ for $1\leq i\leq \alpha$ and $ \bar v_i=\lim_{n\to\infty}\frac{\ell!}{n^\ell}v_i^n$ for $1\leq i\leq \beta$; the set of accumulation points of $(r_i^{n\theta_i})_n$ is $\{r_i^{\theta}:\theta\in\overline{\theta_i\Z+2\pi\Z}\}$ for $1\leq i\leq \beta$. Let $\tilde x\in C$ be a lift of $x\in\Omega$. The accumulation points of $(\frac{\ell!}{n^\ell}\tilde g^n\tilde x)_n$ are 
 $$\left\{ \sum_i\bar u_i\tilde x + (-1)^{m}\sum_i\bar u_i'\tilde x + \sum_ir_i^{\theta_i'}\bar v_i\tilde x : m\in\{0,1\}, \ \theta_i'\in\overline{\theta_i(2\Z+m)+2\pi\Z}\right\},$$
 which are non-zero by Fact~\ref{fait:spec gap}. Since $(g^nx)_n$ converges to $\xi^+$, these accumulation points are all in $\xi^+\cap C$, and this imply that $\bar u_i'\tilde x=0$ for $1\leq i\leq \alpha'$ and $\bar v_i\tilde x=0$ for $1\leq i\leq \beta$. Up to considering another basis, we can assume that $\bar u_i\tilde x=0$ for any $2\leq i\leq\alpha$. 
 
 The element $\tilde g$ commutes with every matrix of its decomposition \eqref{eq:decomp g}, and hence also with $\{\bar u_i\}_{1\leq i\leq \alpha}$, and $\{\bar u'_i\}_{1\leq i\leq \alpha'}$, and $\{r_i^\theta\}_{1\leq i\leq \beta,\ \theta\in\R}$ and $\{\bar v_i\}_{1\leq i\leq \beta}$. Thus $\{\bar u_i\tilde g^n\tilde x\}_{2\leq i\leq \alpha}$,  $\{\bar u'_i\tilde g^n\tilde x\}_{1\leq i\leq \alpha'}$, and $\{\bar v_i\tilde g^n\tilde x\}_{1\leq i\leq \beta}$ are zero for any $n\geq0$.
 By construction of $W$, all elements $\{\bar u_i\}_{2\leq i\leq \alpha}$, and $\{\bar u'_i\}_{1\leq i\leq \alpha'}$, and $\{\bar v_i\}_{1\leq i\leq \beta}$ are zero on $W$.
 
 Suppose by contradiction that the restriction of $g$ to $x_g^+$ is not diagonalisable over $\C$, \ie that $\ell>0$. Then $\bar u_1 \tilde \xi_+=0$ for any lift $\tilde \xi_+\in W$ of $\xi_+$. This, and the fact that $\bar u_1\tilde x\neq 0$, imply that $\bar u_1\tilde y\neq 0$ for any lift $\tilde y\in C$ of $y:=\stereo_{\xi_+}^x\in F_-$. As a consequence, $(g^ny)_n$ converges to $\xi_+$. Since $F_-$ is closed, it contains $\xi_+$, as well as $x$: contradiction. For the same reasons, the restriction of $g$ to $x_g^-$ is diagonalisable over $\C$, and this concludes the proof of \ref{item:realelldiag}.
 
 The sequence of restrictions $(\tilde g^n_{|W})_n$ converges to $\bar u_{1|W}$ which is proximal because it is a projector with rank $1$. Therefore $\tilde g^n_{|W}$ is proximal for $n$ large enough, and so is $\tilde g_{|W}$. For the same reasons, $\tilde g_{|W}^{-1}$ is proximal, and this concludes the proof of \ref{item:realellbip}.
 
 In order to establish \ref{item:realellface}, it is enough to prove that $F_+$ (\resp $F_-$) is contained in $x_g^+\oplus x_g^0$ (\resp $x_g^-\oplus x_g^0$), since its dimension is $k-1$. Pick $\xi\in F_+$; the sequence $(g^n\xi)_{n\geq 0}$ cannot converge to $\xi_-$ since this point does not belong to $F_+$. That $g_{|\PR(W)}$ is biproximal implies that $\xi\in x_g^+\oplus x_g^0$. The proof of $F_-\subset x_g^-\oplus x_g^0$ is similar.
 
 Suppose that $x,gx,g^2x$ are not aligned. Then $b\in F_+\smallsetminus \{\xi_+\}\subset x_g^+\oplus x_g^0\smallsetminus x_g^+$, hence $(g^{-n} b)_n$ accumulates in $x_g^0$. This proves \ref{item:realellpasaligne}.
 
 Suppose that $g$ is not rank-one. Since $x\in[a,b]\cap\Omega$, the simplicial distance between $a$ and $b$ is at least $2$. If $x,gx,g^2x$ are aligned then $a=\xi_-$ and $b=\xi_+$ are fixed by $g$, and $d_{\spl}(a,b)\leq 2$ by Corollary~\ref{cor:carac rgun}. Otherwise, there is $\xi\in x_g^0\cap\partial\Omega'$, and $d_{\spl}(a,b)\leq d_{\spl}(a,\xi)+d_{\spl}(\xi,b)\leq 2$ since $\xi\in F_+\cap F_-$. This proves \ref{item:realellpasr1}.
\end{proof}

\subsection{Non-straight closed geodesics on convex cocompact manifolds}

Let $\Gamma\subset\PGL(V)$ be a discrete subgroup that preserves a properly convex open set $\Omega\subset\PR(V)$; set $M:=\Omega/\Gamma$. In Section~\ref{Periodic orbits and elements of Gamma}, we have associated to each periodic geodesic of $M$ a conjugacy class of $\Gamma$. When $\Gamma$ acts convex cocompactly on $\Omega$, one can produce a weak converse of this; to each automorphism of $\gamma\in\Gamma$, we are able to associate a geodesic segment of length $\ell(\gamma)$, which is closed on $M$ and is freely homotopic to $\gamma$, but which is not necessarily closed in $T^1M$.

\begin{fait}[\!{\!\cite[Prop.\,10.4]{fannycvxcocpct}}]\label{une geodesique presque periodique dans le coeur convexe}
Consider a discrete subgroup $\Gamma\subset\PGL(V)$ that acts convex cocompactly on a properly convex open sets $\Omega\subset\PR(V)$. Then for any $\gamma\in\Gamma$, there exists $x\in\mathcal{C}^{\core}_\Omega(\Gamma)$ such that $d_\Omega(x,\gamma x)=\ell(\gamma)$.
\end{fait}

The following corollary ensures the closed curve can be taken tangent to $T^1M_{\core}$.

\begin{cor}\label{Corollaire : real d'ell dans core}
 Let $\Gamma\subset\PGL(V)$ be a convex cocompact discrete subgroup. Consider a $\Gamma$-invariant properly convex open set $\Omega\subset\PR(V)$ such that $\Gamma$ acts convex cocompactly on $\Omega$ and $\Omega^*$; set $M=\Omega/\Gamma$. Let $\gamma\in\Gamma$ and $v\in T^1\Omega$ be such that $\ell(\gamma)>0$ and $\gamma\pi v=\pi\phi_{\ell(\gamma)}v$. Then $\phi_{\pm\infty}v\in\Lambda^{\orb}$ and $\pi_\Gamma v\in T^1M_{\core}$.
\end{cor}

\begin{proof}
According to Fact~\ref{Fait : fannycvxcocpct}, the properly convex set $\overline\Omega\smallsetminus \Lambda^{\orb}$ has bisaturated boundary (\ie $[\xi,\eta]\cap\Omega\neq\emptyset$ for all $\xi\in\Lambda^{\orb}$ and $\eta\in\overline\Omega\smallsetminus\Lambda^{\orb}$). By Lemma~\ref{lem:realell}, $\phi_{\pm\infty}v\in x_\gamma^\pm\oplus x_\gamma^0\cap\partial\Omega$, hence the segment between $\phi_{\pm\infty}v$ and the limit of $(\gamma^{\pm n}\pi v)_n$ is contained in $\partial\Omega$. This implies that $\phi_{\pm\infty}v\in\Lambda^{\orb}$ since $\overline\Omega\smallsetminus \Lambda^{\orb}$ has bisaturated boundary and since $\lim_{n\to\infty}\gamma^n\pi v\in\Lambda^{\orb}$. Thus $v$ belongs to $ T^1M_{\core}$.
\end{proof}

The following corollary ensures that when $\Gamma$ is convex cocompact, if $(T^1\Omega_{\max}/\Gamma)_{\bip}$ is non-empty then we can find $\Omega\subset\Omega_{\max}$ such that $\Omega/\Gamma$ is rank-one. This corollary generalises results of Islam \cite[Lem.\,6.4]{islam_rank_one} and Zimmer \cite[Th.\,7.1]{zimmer_higher_rank} for compact convex projective manifolds; one may also compare it to a remark of Islam \cite[Rem.\,A.1.C]{islam_rank_one} which concerns not necessarily compact convex projective manifolds with a compact convex core.

\begin{cor}\label{Corollaire : bip* => rg1}
 Let $\Gamma\subset\PGL(V)$ be a discrete subgroup that acts convex cocompactly on a properly convex open set $\Omega\subset\PR(V)$.
 \begin{itemize}
  \item Any biproximal element of $\Gamma$ whose dual axis in $\PR(V^*)$ intersects $\Omega^*$ is rank-one; in particular, if $T^1(\Omega^*/\Gamma)_{\bip}\neq\emptyset$, then $\Omega/\Gamma$ is rank-one.
  \item Suppose that $\Gamma$ acts convex cocompactly on $\Omega^*$, then any biproximal element of $\Gamma$ whose axis intersects $\Omega$ is rank-one; in particular, if $(T^1\Omega/\Gamma)_{\bip}\neq\emptyset$, then $\Omega/\Gamma$ is rank-one.
 \end{itemize}
\end{cor}

\begin{proof}
 The first point is an immediate consequence of Lemma~\ref{lem:realell}.\ref{item:realellpasaligne} and Fact~\ref{une geodesique presque periodique dans le coeur convexe}. Let us establish the second point. Let $\gamma\in\Gamma$ be biproximal with $\axis(\gamma)\cap\Omega\neq\emptyset$. Since $\Gamma$ acts convex cocompactly on $\Omega^*$, the first point implies that $\gamma$, seen as a rank-one automorphism of $\Omega^*$, is rank-one. Thus, $\gamma$ is a rank-one automorphism of $\Omega$ by Fact~\ref{equivalences rang un}.
\end{proof}

\subsection{A closing lemma}

In this section we state a closing lemma, generalising \cite[Th.\,4.4]{bray_top_mixing} and weaker than the classical one from Anosov \cite[Lem.\,13.1]{anosov}. We briefly recall the idea: whenever a geodesic segment comes back sufficiently close to its starting point (no matter how long it is), we can find closed geodesic which tracks it. The following version a more geometrical formulation of the closing lemma. We state the dynamical version below.

\begin{lemma}\label{lem:closinggen}
 Let $\Omega\subset\PR(V)$ be a properly convex open set, $x\in\Omega$ and $(\xi_-,\xi_+)\in\partial\Omega^2$. Assume that $d_{\spl}(\xi_+,\overline F_\Omega(\xi_-))\geq 2$ and $d_{\spl}(\xi_-,\overline F_\Omega(\xi_+))\geq 2$. Then there exists $R>0$ such that for any neighbourhood $W$ of $\overline B_{\overline \Omega}(\xi_-,R)\times \overline B_{\overline\Omega}(\xi_+,R)$ in $\overline \Omega^2$, there exists a neighbourhood $U$ of $(\xi_-,\xi_+)$ such that for any $g\in\Aut(\Omega)$, if $(g^{-1}x,gx)\in U$, then $g$ is biproximal and $(x_g^-,x_g^+)\in W$.
\end{lemma}

To prove the previous result, we need the following generalisation of the fact that projective transformations with an attracting fixed point are proximal.
\begin{fait}[\!{\!\cite{Birkhoff_perronfrob}}]\label{Lemme : contraction -> prox}
 Let $g\in\PGL(V)$ contract a properly convex open set, in the sense that there exists a properly convex open set $\Omega\subset\PR(V)$ such that $g(\overline\Omega)\subset\Omega$. Then $g$ is proximal.
\end{fait}

\begin{proof}[Proof of Lemma~\ref{lem:closinggen}]
 The case where $\dim(V)=2$ is trivial, and we assume that $\dim(V)\geq 3$.
 \begin{enumerate}[label=(\arabic*)]
  \item \label{item:closing1} By assumption, we can find $R>0$ large enough so that $\mathcal{O}_R(\xi_\pm,x)$ contains $\overline F_{\Omega}(\xi_\mp)$.
  \item \label{item:closing2} By lower semi-continuity of $d_{\overline\Omega}$, we may find a neighbourhood $U_\pm\subset\overline\Omega\smallsetminus \overline B_\Omega(x,R)$ of $\xi_\pm$ such that $\mathcal{O}_R(\xi,x)$ contains $\overline{\mathcal{O}}_R(x,y)$ for all $\xi\in U_\pm$ and $y\in\Omega\cap U_\mp$.
  \item \label{item:closing3} Our assumption ensures that $d_{\spl}(\xi,\overline F_\Omega(\xi_{\mp}))\geq 2$ for any $\xi\in \overline B_{\overline\Omega}(\xi_\pm,R)$, so we can find $R'\geq R$ large enough so that $\mathcal{O}_{R'}(\xi,x)$ contains $\overline F_{\Omega}(\xi_\mp)$ for any $\xi\in \overline B_{\overline\Omega}(\xi_\pm,R)$.
  \item \label{item:closing4} By lower semi-continuity of $d_{\overline\Omega}$, we may find a neighbourhood $W_\pm\subset\overline\Omega\smallsetminus\overline B_\Omega(x,R')$ of $\overline B_{\overline\Omega}(\xi_\pm,R)$ such that $\mathcal{O}_{R'}(\xi,x)$ contains $\overline{\mathcal{O}}_{R'}(\xi,y)$ for all $\xi\in W_\pm$ and $y\in\Omega\cap W_\mp$, and such that $W_-\times W_+\subset W$.
  \item \label{item:closing5} Take a neighbourhood $U_\pm'\subset U_\pm$ of $\xi_\pm$ such that $\overline{\mathcal{O}}_R(x,y)$ is contained in $W_\pm$ for any $y\in \Omega\cap U_\pm'$.
 \end{enumerate}

 Consider $g\in\Aut(\Omega)$ such that $g^{\pm1}x\in U_\pm'$, and let us show that $g$ is biproximal with $(x_g^-,x_g^+)$ in $W$. By \ref{item:closing2} and since $gx\in U_+$ and $g^{-1}x\in U_-$, we have 
 $$g\overline{\mathcal{O}}_R(g^{-1}x,x)=\overline{\mathcal{O}}_R(x,gx)\subset \mathcal{O}_R(g^{-1}x,x).$$
 Hence $g$ fixes some point $\eta_+\in\overline{\mathcal{O}}_R(x,gx)$ by the Brouwer fixed point theorem. Symmetrically, $g$ fixes some point $\eta_-\in\overline{\mathcal{O}}_R(x,g^{-1}x)$.
 
 By \ref{item:closing5}, the point $\eta_+$ lies in $W_+$, and $\eta_-$ lies in $W_-$. By \ref{item:closing4}, this implies that
 $$g\overline{\mathcal{O}}_{R'}(\eta_-,x)=\overline{\mathcal{O}}_{R'}(\eta_-,gx)\subset \mathcal{O}_{R'}(\eta_-,x).$$
 Therefore, according to Fact~\ref{Lemme : contraction -> prox}, the projection $g'\in\PGL(V/\eta_-)$ of $g$ is proximal, and its attracting fixed point corresponds in $\PR(V)$ to a line of the form $\eta_-\oplus\zeta_+$, where $\zeta_+\in\overline{\mathcal{O}}_{R'}(\eta_-,gx)$ is fixed by $g$.
 
 By Fact~\ref{fait:spec gap}, since $\ell(g)\geq \ell(g')>0$ and since $\eta_-\oplus\zeta_+$ intersects $\Omega$, we either have $(\eta_-,\zeta_+)\in x_g^-\times x_g^+$ or $(\eta_-,\zeta_+)\in x_g^+\times x_g^-$. The latter case contradicts the fact that $\dim(V)\geq 3$ and $g'$ is proximal. Hence $\eta_-\in x_g^-$ and $\zeta_+\in x_g^+$, and $g$ is proximal with $\zeta_+=x_g^+$. Symmetrically, $g^{-1}$ is also proximal and $\eta_+\in x_g^+$. We have proved that $g$ is biproximal with $(x_g^-,x_g^+)=(\eta_-,\eta_+)\in W$.
\end{proof}

\begin{cor}\label{Corollaire : fermeture}
 Let $\Omega\subset\PR(V)$ be a properly convex open set. Let $x\in\Omega$, $(\xi_-,\xi_+)\in\partial\Omega^2$ and $W$ a neighbourhood of $(\xi_-,\xi_+)$. Then there exists a neighbourhood $U$ of $(\xi_-,\xi_+)$ such that for any $g\in\Aut(\Omega)$ with $(g^{-1}x,gx)\in U$,
 \begin{itemize} 
  \item if $d_{\spl}(\xi_-,\xi_+)\geq 3$, then $g$ is rank-one;
  \item if $\xi_-$ and $\xi_+$ are extremal and $d_{\spl}(\xi_-,\xi_+)\geq 2$, then  $g$ is biproximal and $(x_g^-,x_g^+)\in W$;
  \item if $\xi_-$ and $\xi_+$ are distinct and strongly extremal, then $g$ is rank-one and $(x_g^-,x_g^+)\in W$.
 \end{itemize}
\end{cor}

Given a convex projective orbifold $M$, we use $\tilde{B}_{T^1M}^{(t)}$ to denote the open balls for the metric $\tilde d_{T^1\Omega}^{(t)}$ (see Section~\ref{rappels sur l'entropie topologique}).

\begin{lemma}\label{closinglemma}
 Let $\Omega\subset\PR(V)$ be a properly convex open set and $\Gamma\subset\Aut(\Omega)$ a discrete subgroup; denote $M=\Omega/\Gamma$. Consider $\alpha>0$ and $v_0\in T^1M$ such that the endpoints $\phi_{-\infty}\vv_0$ and $\phi_\infty \vv_0$ of any lift $\vv_0$ are extremal (\resp strongly extremal). Then there exists $\epsilon>0$ satisfying the following. For any $v\in B_{T^1M}(v_0,\epsilon)$, and for any time $t>\alpha$, if $\phi_tv\in B_{T^1M}(v_0,\epsilon)$, then one can find a biproximal (\resp rank-one) periodic vector $w\in \tilde{B}_{T^1M}^{(t)}(v,\alpha)$ with period in $[t-\alpha,t+\alpha]$.
\end{lemma}

\begin{proof}
 Let $W$ be a neighbourhood of $(\phi_{-\infty}\vv_0,\phi_{\infty}v_0)$ such that $[\xi_-,\xi_+]\cap B_{T^1\Omega}(\phi_t\vv_0,\alpha/8)\neq\emptyset$ for any $0\leq t\leq 1$. By Corollary~\ref{Corollaire : fermeture}, we can find a neighbourhood $U=U_-\times U_+\subset W$ of $\phi_{\pm\infty}\vv_0$ such that for any $\gamma\in\Gamma$, if $(\gamma^{-1}\pi\vv_0,\gamma\pi\vv_0)\in U$ then $\gamma$ is biproximal (\resp rank-one) and $(x_\gamma^-,x_\gamma^+)\in W$. Let $t_0>0$ and $\epsilon_1<\alpha/8$ be such that $\overline B_\Omega(\pi\phi_{\pm t}\vv,\epsilon_1)\subset U_\pm$ for any $\vv\in \overline B_{T^1\Omega}(\vv_0,\epsilon_1)$ and $t\geq t_0$.
 
 Now consider $t\geq t_0$ and $v\in B_{T^1M}(v_0,\epsilon_1)$ such that $\phi_tv\in B_{T^1M}(v_0,\epsilon_1)$. We can find a lift $\vv\in B_{T^1\Omega}(\vv_0,\epsilon_1)$ and an element $\gamma\in\Gamma$ such that $\phi_t\vv\in B_{T^1\Omega}(\gamma\vv_0,\epsilon_1)$. Then $(\gamma^{-1}\pi\vv_0,\gamma\pi\vv_0)\in U$, hence $\gamma$ is biproximal (\resp rank-one) and $(x_\gamma^-,x_\gamma^+)\in W$. Be definition of $W$, we can find $\ww\in B_{T^1\Omega}(\vv_0,\alpha/8)$ tangent to the axis of $\gamma$. Then $d_{T^1\Omega}(\ww,\vv)\leq \alpha/4$ and $d_{T^1\Omega}(\gamma\ww,\phi_t\vv)\leq\alpha/4$; since $\gamma\ww=\phi_{\ell(\gamma)}\ww$, by triangular inequality we have $|\ell(\gamma)-t|\leq \alpha/2$, and $d_{T^1\Omega}^{(t)}(\ww,\vv)\leq \alpha$.
 
 To finish the prooof, it remains to find $\epsilon<\epsilon_1$ such that for all $\alpha\leq t\leq t_0$ and $v\in B_{T^1M}(v_0,\epsilon)$, if $\phi_tv\in B_{T^1M}(v_0,\epsilon)$, then one can find a biproximal (\resp rank-one) periodic vector $w\in \tilde{B}_{T^1M}^{(t)}(v,\alpha)$ with period in $[t-\alpha,t+\alpha]$. If $v_0$ is not periodic, then we take $\epsilon<\epsilon_1$ small enough so that $B_{T^1M}(v_0,\epsilon)\subset B_{T^1M}^{(t_0)}(v_0,\epsilon_2/2)$, where $\epsilon_2:=\min_{\alpha\leq t\leq t_0}d_{T^1M}(\phi_tv_0,v_0)$. If $v_0$ is periodic, then it is biproximal (\resp rank-one) since $\phi_{\pm\infty}\vv_0$ are extremal (\resp strongly extremal). We then take $\epsilon<\epsilon_1$ small enough so that $B_{T^1M}(v_0,\epsilon)\subset B_{T^1M}^{(t_0)}(v_0,\epsilon_2/2)$, where $\epsilon_2:=\min_{t\in\mathcal{T}}d_{T^1M}(\phi_tv_0,v_0)$ and $\mathcal{T}$ is the set of times $\alpha\leq t\leq t_0$ such that $\phi_sv_0\neq v_0$ for any $t-\alpha^< s< t+\alpha$.
\end{proof}

\section{The measure of maximal entropy}\label{The measure of maximal entropy}

In this section we prove that on a non-elementary rank-one convex projective manifold $M=\Omega/\Gamma$, if $\Gamma$ acts convex cocompactly on $\Omega$, then the \emph{Bowen-Margulis probability measure} (\ie the unique Bowen--Margulis measure with total mass $1$) is the unique measure with maximal entropy.

\subsection{The measure of dynamical balls}\label{the measure of dynamical balls}

In this section we derive from the Shadow lemma (Corollary~\ref{shadowlemma+}) an estimate for the measure of dynamical balls (namely balls for the metrics $(d^{(t)})_{t\geq 0}$ defined in Definition~\ref{Topological Entropy}). The idea is very similar to the computations in the proof of Proposition~\ref{l'ensemble limite conique a mesure non nulle}, which are actually Roblin's computations; the difference here is that our shadows are a priori \emph{small}, and this is why we need the stronger Shadow lemma (Corollary~\ref{shadowlemma+}), which works for small shadows. However, Corollary~\ref{shadowlemma+} only works for usual shadows $\mathcal{O}_R(x,y)$, and not for the scarce shadows of the form $\mathcal{O}^-_R(x,y)$; to overcome this issue we use the following lemma, which is a consequence of Benz\'ecri's compactness theorem (Fact~\ref{benz}).

\begin{lemma}\label{trapeze ecrase}
 For any $\epsilon>0$, there exists $\epsilon'>0$ such that for any properly convex open set $\Omega\subset\PR(V)$, for any $x,y\in\Omega$ at distance at least $1$, if we have two points on the boundary $\xi\in\mathcal{O}_{\epsilon'}(y,x)$ and $\eta\in\mathcal{O}_{\epsilon'}(x,y)$, then the line $\xi\oplus\eta$ intersects both balls $B_\Omega(x,\epsilon)$ and $B_\Omega(y,\epsilon)$.
\end{lemma}
 
\begin{proof}
 By symmetry it is enough to prove that $\xi\oplus\eta$ meets $B_\Omega(x,\epsilon)$. We assume by contradiction that there exist sequence $(\epsilon_n')_{n\in\N}\in\R_{>0}^\N$ converging to $0$ and a sequence $(\Omega_n)_{n\in\N}\in\mathcal{E}_V^\N$ such that for each $n\in\N$ we can find $x_n,y_n\in\Omega_n$ at distance at least $1$ and $\xi_n\in\mathcal{O}_{\epsilon_n'}(y_n,x_n)$ and $\eta_n\in\mathcal{O}_{\epsilon_n'}(x_n,y_n)$ such that $(\xi_n\oplus\eta_n)\cap B_\Omega(x_n,\epsilon)=\emptyset$. 
 
 By Benz\'ecri's compactness theorem (Fact~\ref{benz}), we can assume, up to extraction, that $((\Omega_n,x_n))_{n}$ converges to some pointed properly convex open set $(\Omega,x)\in\mathcal{E}_V^\bullet$, bounded in some affine chart that we fix. Up to extraction we assume that $\Omega_n$ is also bounded in the affine chart for any $n\in\N$, and that $(y_n)_{n\in\N}$ (\resp $(\xi_n)_{n\in\N}$ and $(\eta_n)_{n\in\N}$) converges to $y\in\overline{\Omega}\smallsetminus B_\Omega(x,1)$ (\resp $\xi$ and $\eta\in\partial\Omega$) such that $(\xi\oplus\eta)\cap B_\Omega(x,\epsilon)=\emptyset$. 
 
 By definition of the Hilbert metric one checks that 
 \[B_{\Omega_n}(x_n,\epsilon_n')\subset (1-e^{-2\epsilon_n'})(\Omega_n-x_n)+x_n,\]
 hence $(B_{\Omega_n}(x_n,\epsilon_n'))_{n\in\N}$ converges to the singleton $\{x\}$ for the Hausdorff topology, and similarly $(B_{\Omega_n}(y_n,\epsilon_n'))_{n\in\N}$ converges to the singleton $\{y\}$. This implies that $x\in[\xi,\eta]$, which contradicts $(\xi\oplus\eta)\cap B_\Omega(x,\epsilon)=\emptyset$.
\end{proof}

\begin{lemma}\label{measure_of_dyn_ball sans cvxcocpct}
 Let $\Omega\subset\PR(V)$ be a properly convex open set, and $\Gamma\subset\Aut(\Omega)$ a discrete subgroup; set $M=\Omega/\Gamma$. Suppose that $\Gamma$ is strongly irreducible and $T^1M_{\bip}\neq\emptyset$, or that $M$ is rank-one and non-elementary. Let $m$ be a Sullivan measure on $T^1M$ induced by a $\delta_\Gamma$-conformal density. Then for any compact subset $K\subset T^1M$, for any $r>0$, there exists a constant $C>0$ such that given any time $t>0$, for any $v\in \Gamma\cdot K$ such that $\phi_tv\in\Gamma\cdot K$,
 \[m(\tilde{B}^{(t)}_{T^1M}(v,r))\leq Ce^{-\delta_\Gamma t},\]
 and if $v\in T^1M_{\bip}$, then
 \[C^{-1}e^{-\delta_\Gamma t}\leq m(\tilde{B}^{(t)}_{T^1M}(v,r)).\]
\end{lemma}

\begin{proof}
 Let $(\mu_x)_{x\in\Omega}$ be a $\delta_\Gamma$-conformal density on $\partial_{\hor}\Omega$ which induces a $\delta_\Gamma$-conformal density $(\nu_x)_{x\in\Omega}$ on $\partial\Omega$ and the Sullivan measures $\tilde{m}$ on $T^1\Omega$ and $m$ on $T^1M$. Let $\tilde{v}\in T^1\Omega$ be a lift of $v$. We have
 \[ C_0^{-1} \tilde{m}(B^{(t)}_{T^1\Omega}(\tilde{v},r)) \leq m(\tilde{B}^{(t)}_{T^1M}(v,r)) \leq \tilde{m}(B^{(t)}_{T^1\Omega}(\tilde{v},r))\]
 for any $t\geq 0$, where 
 \[C_0=\max_{\tilde{w}\in\pi^{-1}K}\#\{\gamma : d_\Omega(\pi \tilde{w},\gamma\pi \tilde{w})\leq 4r\}.\]
 Let us prove the upper bound in the lemma. Consider $\tilde{w}\in B^{(t)}_{T^1\Omega}(\tilde{v},r)$. We make the following observations.
 \begin{itemize} 
  \item The Lebesgue measure of the set of times $s\in\R$ such that $\phi_s\tilde{w}\in B^{(t)}_{T^1\Omega}(\tilde{v},r)$ is less than $2r$;
  \item $\phi_\infty \tilde{w}\in \mathcal{O}^+_{r}(\pi \tilde{v},\pi\phi_t\tilde{v})$;
  \item $\langle \xi,\eta \rangle_{\pi \tilde{v}}\leq r$ for all $\xi\in\pi_{\hor}^{-1}(\phi_{-\infty}\tilde{w})$ and $\eta\in\pi_{\hor}^{-1}(\phi_{\infty}\tilde{w})$.
 \end{itemize}
Combined with the definition of $\tilde{m}$ (see Section~\ref{The Sullivan measures}), they yield:
 \[\tilde{m}(B^{(t)}_{T^1\Omega}(\tilde{v},r))\leq e^{2\delta_\Gamma r}\cdot\nu_{\pi\tilde{v}}(\partial\Omega)\cdot \nu_{\pi \tilde{v}}(\mathcal{O}^+_{r}(\pi \tilde{v},\pi\phi_t \tilde{v}))\cdot 2r.\]
 We deduce from this and the Shadow lemma (Lemma~\ref{shadowlemma+ sans cvxcpcpt}) the desired upper bound.
 
 Let us prove the lower bound. We apply Lemma~\ref{trapeze ecrase} to $\epsilon:=r/16$ to obtain $\epsilon'>0$. Then for all $\eta\in \mathcal{O}_{\epsilon'}(\pi\tilde{v},\pi\phi_{t+1}\tilde{v})$ and $\xi\in \mathcal{O}_{\epsilon'}(\pi\phi_{t+1}\tilde{v},\pi\tilde{v})$, one can find $\tilde{w}\in B^{(t)}_{T^1\Omega}(\tilde{v},r/2)$ tangent to $\xi\oplus\eta$. Observe that the Lebesgue measure of the set of times $s\in\R$ such that $\phi_s\tilde{w}\in B^{(t)}_{T^1\Omega}(\tilde{v},r)$ is greater than~$r$. This means (remembering that the Gromov product is always non-negative) that 
 \[ \tilde{m}(B^{(t)}_{T^1\Omega}(\tilde{v},r)) \geq \nu_{\pi v}(\mathcal{O}_{\epsilon'}(\pi\tilde{v},\pi\phi_{t+1}\tilde{v}))\cdot \nu_{\pi v}(\mathcal{O}_{\epsilon'}(\pi\phi_{t+1}\tilde{v},\pi\tilde{v})) \cdot r. \]
 We conclude thanks to the Shadow lemma (Lemma~\ref{shadowlemma+ sans cvxcpcpt}), and the fact that $\Lambda^{\prox}\subset\supp(\nu_o)$.
\end{proof}

In the convex cocompact case, this yields the following.

\begin{cor}\label{measure_of_dyn_ball}
 Let $\Omega\subset\PR(V)$ be a properly convex open set, and $\Gamma\subset\Aut(\Omega)$ a convex cocompact discrete subgroup; set $M=\Omega/\Gamma$. Suppose that $\Gamma$ is strongly irreducible and $T^1M_{\bip}\neq\emptyset$, or that $M$ is rank-one and non-elementary. Let $m$ be a Sullivan measure on $T^1M$ induced by a $\delta_\Gamma$-conformal density. Then for any $r>0$, there exists a constant $C>0$ such that given any time $t>0$, for any $v\in T^1M_{\core}$,
 \[m(\tilde{B}^{(t)}_{T^1M}(v,r))\leq Ce^{-\delta_\Gamma t},\]
 and if $v\in T^1M_{\bip}$, then
 \[C^{-1}e^{-\delta_\Gamma t}\leq m(\tilde{B}^{(t)}_{T^1M}(v,r)).\]
\end{cor}

\subsection{The Bowen--Margulis measure has maximal entropy}\label{ssection:entropie max}

In this section we prove that any Sullivan measure induced by a $\delta_\Gamma$-conformal density has maximal entropy, which is equal to $\delta_\Gamma$. For this we do not need Theorem~\ref{Thm : The weak Hopf--Tsuji--Sullivan--Roblin dichotomy}.

\begin{prop}\label{l'entropie de BM est max}
 Let $\Omega\subset\PR(V)$ be a properly convex open set, and $\Gamma\subset\Aut(\Omega)$ a convex cocompact discrete subgroup with $M=\Omega/\Gamma$ rank-one and non-elementary. Let $m$ be the Bowen--Margulis probability measure on $T^1M$. Then $m$ has maximal entropy on $T^1M_{\core}$; in other words,
 \[h_m(T^1M_{\bip},(\phi_t)_t)=h_{\topp}(T^1M_{\bip},(\phi_t)_t)=h_{\topp}(T^1M_{\core},(\phi_t)_t)=\delta_\Gamma.\]
\end{prop}

\begin{proof}
 We can assume without loss of generality that $\Gamma$ is torsion-free by Observation~\ref{obs:revetentrop} and since for any finite-index subgroup $\Gamma'$ of $\Gamma$, the measure $m$ is the push-forward of the Bowen--Margulis probability measure on $T^1\Omega/\Gamma'$. Let $\epsilon_0$ be the injectivity radius of $M$. Let $\p$ be a finite measurable partition of $T^1M_{\core}$ whose elements have diameter less than $\frac{\epsilon_0}{3}$. According to the definition of the measure-theoretic entropy, to the variational principle (Fact~\ref{le principe variationnel}), and to Fact~\ref{manning}, it is enough to prove that 
 \[H_m(\phi_1,\p)\geq\delta_\Gamma.\]
 
 For $n\geq 1$, we observe that by definition (and by Lemma~\ref{crampon}), any element of $\p^{(n)}$ has diameter less than $\epsilon_0/3$ with respect to the metric $d_{T^1M}^{(n)}$. Therefore by Corollary~\ref{measure_of_dyn_ball}, there exists a constant $C>0$ such that for any $n\geq 1$, any element of $\p^{(n)}$ has an $m$-measure less than $Ce^{-\delta_\Gamma n}$. We now conclude the proof by the following computation.
 \begin{align*}
  H_m(\p^{(n)})&=-\sum_{P\in\p^{(n)}}m(P)\log(m(P))\\
  &\geq -\sum_{P\in\p^{(n)}}m(P)\log(C e^{-\delta_\Gamma n})\\
  &\geq \delta_\Gamma n -\log(C).
 \end{align*}
 The inequality $H_m(\phi_1,\p)\geq \delta_\Gamma$ follows immediately.
\end{proof}

\subsection{Separated sets in dynamical balls}\label{Ssection : separons les boules dyn}

In this section we prove a technical lemma which bound from above the size a separated set in a dynamical ball. To perform this estimate we will use the notion of proper densities (Definition~\ref{Definition : densite propre}).

\begin{cor}\label{ensemble separe d'une boule dynamique}
 Consider $0<r<R$. Let $\Omega\in\mathcal{E}_V$, let $\Gamma\subset\Aut(\Omega)$ be a discrete subgroup, let $M=\Omega/\Gamma$ and let $t\geq0$.
 \begin{enumerate}[label=(\arabic*)]
  \item \label{item:ensemble separe d'une boule dynamique1} For any vector $v\in T^1M$, the cardinality of any $(\tilde{d}_{T^1M}^{(t)},r)$-separated set of $\tilde{B}_{T^1M}^{(t)}(v,R)$ is less than $\chi_+(R+r/4)^2\cdot\chi_-(r/4)^{-2}$ (see \eqref{Equation : chi}).
  \item \label{item:ensemble separe d'une boule dynamique2} Consider a $(\tilde{d}_{T^1M}^{(t)},r)$-separated subset $\{v_1,\dots,v_k\}$ of $T^1M$ (and of size $k$), take a vector $w_i\in \tilde{B}_{T^1M}^{(t)}(v_i,R)$ for each $i=1,\dots,k$. Then one can find a subset $I$ of $\{1,\dots,k\}$ of size greater than $k\cdot \chi_+(2R+r/2)^{-2}\cdot\chi_-(r/4)^{2}$ such that $\{w_i : i\in I\}$ is $(\tilde{d}_{T^1M}^{(t)},r)$-separated.
 \end{enumerate}
\end{cor}

\begin{proof}
 Let us prove \ref{item:ensemble separe d'une boule dynamique1} when $\Gamma$ is trivial.
 Let $A\subset B_{T^1\Omega}^{(t)}(v,R)$ be a $(d_{T^1\Omega}^{(t)},r)$-separated set. We set
 \[B:=\{(\pi w,\pi\phi_tw) : w\in A\}\subset B_\Omega(\pi v,R)\times B_\Omega(\pi\phi_tv,R)\}.\]
 By Lemma~\ref{crampon}, and since $A$ is $(d_{T^1\Omega}^{(t)},r)$-separated, we see that $B$ is $(d,r/2)$-separated for the metric $d$ on $\Omega^2$, defined by $d((x,y),(x',y'))=\max(d_\Omega(x,x'),d_\Omega(y,y'))$ for $(x,y),(x',y')\in\Omega^2$; this exactly means that for all $(x,y)\neq (x',y')\in B$,
 \[(B_\Omega(x,r/4)\times B_\Omega(y,r/4))\cap (B_\Omega(x',r/4)\times B_\Omega(y',r/4))=\emptyset.\]
 As a consequence,
 \begin{align*}
  \#A &= \#B \\
  &\leq \chi_-(r/4)^{-2}\sum_{(x,y)\in B}\Vol_\Omega(B_\Omega(x,r/4))\Vol_\Omega(B_\Omega(y,r/4))\\
  &\leq \chi_-(r/4)^{-2} \Vol_\Omega^2\left(\bigsqcup_{(x,y)\in B}B_\Omega(x,r/4)\times B_\Omega(y,r/4)\right) \\
  &\leq \chi_-(r/4)^{-2} \Vol_\Omega^2(B_\Omega(\pi v,R+r/4)\times B_\Omega(\pi\phi_t v,R+r/4)) \\
  &\leq \chi_+(R+r/4)^2\chi_-(r/4)^{-2}.
 \end{align*}
 
 Let us prove \ref{item:ensemble separe d'une boule dynamique1} when $\Gamma$ is not necessarily trivial. Let $A\subset \tilde{B}_{T^1M}^{(t)}(v,R)$ be a $(\tilde{d}_{T^1M}^{(t)},r)$-separated set. Consider a lift $\vv\in T^1\Omega$ of $v$, and a lift $\tilde{A}\subset B_{T^1\Omega}^{(t)}(\vv,R)$ of $A$. Then $\tilde{A}$ is $(d_{T^1\Omega}^{(t)},r)$-separated, therefore it has cardinality less than $\chi_+(R+r/4)^2\chi_-(r/4)^{-2}$, and so do $A$.
 
 Let us establish \ref{item:ensemble separe d'une boule dynamique2}. We construct $I$ by induction. We set $i_0=0$ and $I_0:=\{1,\dots,k\}$. For $j\geq 1$, if $(I_0,\dots,I_{j-1})$ and $(i_0,\dots,i_{j-1})$ have been constructed, we set $i_j:=\min I_{j-1}>i_{j-1}$ and 
 \[I_j:=\{i\in I_{j-1} : \tilde{d}_{T^1M}^{(t)}(w_{i_j},w_i)\geq r\}\varsubsetneq I_{j-1}.\]
 This process eventually stops, at the $n$-th step for some $n\in\{1,\dots,k\}$ such that $I_n=\emptyset$. The set $\{w_{i_j} : 1\leq j\leq n\}$ is $(\tilde{d}_{T^1M}^{(t)},r)$-separated by construction. In order to prove that $k$ is bounded above by $n\cdot \chi_+(2R+r/2)^{2}\cdot\chi_-(r/4)^{-2} $, it is enough to see that for each $0\leq j\leq n-1$,
 \[\#I_{j+1} \geq \#I_j - \chi_+(2R+r/2)^{2}\cdot\chi_-(r/4)^{-2}.\]
 This is a consequence of \ref{item:ensemble separe d'une boule dynamique1} and of the fact that $I_j\smallsetminus I_{j+1}$ is contained in the set of indices $i$ such that $v_i\in \tilde{B}_{T^1\Omega}^{(t)}(v_{j+1},r+2R)$.
\end{proof}

\subsection{The measure of maximal entropy is unique}\label{ssection:uniciteMME}

In order to prove the uniqueness of the measure of maximal entropy, we will use our estimates on the size of the dynamical balls Corollary~\ref{measure_of_dyn_ball}. However, one of these estimates only holds for balls centered at vectors in $T^1M_{\bip}$, whereas we would like the uniqueness of the measure of maximal entropy on $T^1M_{\core}$. This is why we will need Corollary~\ref{ensemble separe d'une boule dynamique} and the following lemma.

For the rest of this section, given a convex projective manifold $M=\Omega/\Gamma$, we denote by $T^1\Omega_{\core}$ (\resp $T^1\Omega_{\bip}$) the set (depending on $\Gamma$) of vectors $v\in T^1\Omega$ such that $\phi_{\pm\infty}v\in\Lambda^{\orb}_\Omega(\Gamma)$ (\resp $\Lambda^{\prox}(\Gamma)$).

\begin{lemma}\label{le coeur fait bip}
 Let $\Omega\subset\PR(V)$ be a properly convex open set, and $\Gamma\subset\Aut(\Omega)$ a convex cocompact discrete subgroup. Suppose $M=\Omega/\Gamma$ is rank-one and non-elementary. Then there exists $R>0$ such that for every point $\xi\in\Lambda^{\orb}$ we can find $\eta\in\Lambda^{\prox}$ such that $d_{\overline{\Omega}}(\xi,\eta)\leq R$.
 
 In particular, according to Lemma~\ref{crampon}, for any $v\in T^1\Omega_{\core}$, we can find a vector $w\in T^1\Omega_{\bip}$ such that $d_{T^1\Omega}(\phi_tv,\phi_tw)\leq 2R$ for any $t\in\R$.
\end{lemma}

\begin{proof}
 Consider $o$ in the convex hull of $\Lambda^{\orb}$ in $\Omega$ and the $\delta_\Gamma$-conformal density $\nu_o$ on $\partial\Omega$. Let $R>0$ be given by the Shadow lemma (Corollary~\ref{shadowlemma+}). According to Fact~\ref{semi-continuite superieure}, it is enough to show that for any $x\in[o,\xi)$, the shadow $\mathcal{O}_R(o,x)$ intersect $\Lambda^{\prox}$; this is implied by the fact that $\nu_o(\Lambda^{\prox})=1$ (by Theorem~\ref{Thm : The weak Hopf--Tsuji--Sullivan--Roblin dichotomy} and Proposition~\ref{Proposition : cvxcocpct -> divergent}) and $\nu_o(\mathcal{O}_R(o,x))>0$.
\end{proof}

In the case where $M$ is compact, the proof of Theorem~\ref{Thm : proba d'entropie max dans le cas compact} below can be shorten by using the following result; for instance, we do not need Section~\ref{Ssection : separons les boules dyn} and Lemma~\ref{le coeur fait bip} in this case. 
 
\begin{fait}\label{Proposition : prox=bord}
 Let $\Omega\subset\PR(V)$ be a properly convex open set and $\Gamma\subset\Aut(\Omega)$ a non-elementary rank-one discrete subgroup that acts cocompactly on $\Omega$. Then $\Lambda^{\prox}=\partial\Omega$.
\end{fait}

\begin{proof}[Proof of Theorem~\ref{Thm : proba d'entropie max dans le cas compact}]
 It is enough to prove that the Bowen--Margulis probability measure $m$ on $T^1M$ is the unique measure of maximal entropy, since $m$ is mixing by Theorem~\ref{Thm : The weak Hopf--Tsuji--Sullivan--Roblin dichotomy}. 

 We can assume that $\Gamma$ is torsion-free by Observation~\ref{obs:revetentrop} and since for any finite-index subgroup $\Gamma'$ of $\Gamma$, the measure $m$ is the push-forward of the Bowen--Margulis probability measure on $T^1\Omega/\Gamma'$; let $\epsilon_0$ be the injectivity radius of $\Omega/\Gamma$. Consider a $(\phi_t)_{t\in\R}$-invariant probability measure $m'$ on $T^1M_{\core}$ which is different from $m$. Since $m$ is ergodic (Theorem~\ref{Thm : The weak Hopf--Tsuji--Sullivan--Roblin dichotomy} and Proposition~\ref{Proposition : cvxcocpct -> divergent}), $m'$ cannot be absolutely continuous with respect to $m$. By Radon--Nikodym Theorem we can decompose $m'$ into a sum $tm''+(1-t)m$ where $0<t\leq 1$ and $m''$ is $(\phi_t)_{t\in\R}$-invariant and singular with respect to $m$. Then $h_{m'}(\phi)=th_{m''}(\phi)+(1-t)\delta_\Gamma$ (see \cite[Cor.\,4.3.17]{katok}), and we only need to prove that $h_{m''}(\phi)<\delta_\Gamma$. Note that since $\Lambda^{\orb}\smallsetminus\Lambda^{\prox}$ does not intersect $\partial_{\sse}\Omega$. Without loss of generality we assume that $m=m'$ is singular with respect to $m$. Let $A\subset T^1M_{\sse}$ be a flow-invariant measurable subset such that $m(A)=1$ while $m'(A)=0$.
 
 Fix $\epsilon>0$ and let $K_1\subset A\cap T^1M_{\core}$ and $K_2\subset T^1M_{\core}\smallsetminus A$ be compact subsets such that $m(K_1)\geq 1-\epsilon$ and $m'(K_2)\geq 1-\epsilon$. Observe that 
 \begin{equation}\label{K1 loin de K2}
 \min\{\tilde{d}_{T^1M}^{(2t)}(\phi_{-t}v,\phi_{-t}w): v\in K_1, w \in K_2\}\underset{t\to\infty}{\longrightarrow}\infty.
 \end{equation}
 Indeed,otherwise there would exist $v\in K_1$ and $w\in K_2$ such that $\sup_{t>0}\tilde{d}_{T^1M}^{(2t)}(\phi_{-t}v,\phi_{-t}w)<\infty$. Then we could find lifts $\vv,\ww\in T^1\Omega$ such that $\sup_{t>0}{d}_{T^1\Omega}^{(2t)}(\phi_{-t}\vv,\phi_{-t}\ww)<\infty$, which implies, since $\phi_\infty \vv$ and $\phi_{-\infty}\vv$ are extremal, that $\phi_{\pm\infty}\vv=\phi_{\pm\infty}\ww$ hence $w\in \phi_\R v\subset A$, which is a contradiction.
 
 Let $n\geq 1$ and consider a maximal $(d_{T^1M}^{(2n)},\epsilon_0/8)$-separated set $\{v_1,\dots,v_k\}\subset T^1M_{\core}$, which is ordered so that for any $i=1,\dots,k$, the ball $B_{T^1M_{\core}}^{(2n)}(v_i,\epsilon_0/8)$ intersects $\phi_{-n} K_2$ if and only if $i\leq l$, where $1\leq l\leq k$ is some integer.
 
 Construct by induction the finite measurable partition $\p=\{P_1,\dots,P_k\}$ of $T^1M_{\core}$ so that 
 \[B_{T^1M_{\core}}^{(2n)}(v_{i+1},\epsilon_0/16)\subset P_{i+1}:=B_{T^1M_{\core}}^{(2n)}(v_{i+1},\epsilon_0/8)\smallsetminus (P_1\cup\dots\cup P_i)\subset B_{T^1M_{\core}}^{(2n)}(v_{i+1},\epsilon_0/8).\]
 $P_i$ diameter less than $\epsilon_0/3$ with respect to $d_{T^1M}^{(2n)}$ for each $1\leq i\leq k$, therefore we can combine Remarks~\ref{reparam_entropy} and \ref{f^n_is_expansive} with Facts~\ref{geod_is_expansive} and \ref{bowen} to obtain:
 \[h_{m'}(\phi_1)=\frac{1}{2n}h_{m'}(\phi_{2n})=\frac{1}{2n}H_{m'}(\phi_{2n},\p)\leq \frac{1}{2n}H_{m'}(\p).\]
 Now we use the classical fact that for all $q\in\N_{>0}$ and $a_1,\dots,a_q>0$, if $s:=a_1+\dots+a_q\leq 1$ then 
 \[-\sum_i a_i\log(a_i)\leq -s\log s +s\log q \leq  s\log(q) + 1/e,\]
   and we compute:
 \begin{align*}
  H_{m'}(\p) & = -\sum_{i=1}^km'(P_i)\log(m'(P_i))\\
  & \leq m'\left(\bigcup_{i=1}^lP_i\right)\log(l)+m'\left(\bigcup_{i=l+1}^kP_i\right)\log(k)+2/e.
 \end{align*}
 Note that on one hand, $\phi_{-n}K_2$ is contained in $\bigcup_{i=1}^lP_i$, hence $m'\left(\cup_{i=1}^lP_i\right)\geq 1-\epsilon$.
 On the other hand $\phi_{-n}K_1$ does not intersect $\bigcup_{i=1}^lP_i$ for $n$ large enough, indeed if there exist $i\in\{1,\dots,l\}$ and $v\in \phi_{-n}K_1\cap P_i$, then we take a vector $w\in\phi_{-n}K_2\cap B_{T^1M}^{(2n)}(v_i,\epsilon_0/8)$ and use the triangular inequality to obtain that $\tilde{d}_{T^1M}^{(2n)}(\phi_{-n}\phi_nv,\phi_{-n}\phi_nw)\leq \epsilon_0/4$, which is not compatible with \eqref{K1 loin de K2} for $n$ large. As a consequence, $m\left(\cup_{i=1}^lP_i\right)\leq\epsilon$.
 
 We now need to bound from above $k$ and $l$. If $M$ is compact then it is easier to conclude the proof: we can use Fact~\ref{Proposition : prox=bord}, which implies  that $T^1M=T^1M_{\core}=T^1M_{\bip}$, and apply directly Corollary~\ref{measure_of_dyn_ball} to get:

\[l\leq \frac{m(\bigcup_{i=1}^lP_i)}{\min\{m(P_i):1\leq i\leq l\}}\leq \epsilon C e^{2n\delta_\Gamma},\]
 and similarly $k\leq Ce^{2n\delta_\Gamma}$, where $C$ only depends on $\epsilon_0$. Thus we obtain
 \[2nh_{m'}(\phi_1)\leq (1-\epsilon)\log(\epsilon)+\log(C)+2n\delta_\Gamma+2/e,\]
 which is strictly less than $2n\delta_\Gamma$ for $\epsilon$ small enough.
 
 In the general case, we need to take into account that $T^1M_{\core}$ and $T^1M_{\bip}$ might be different and that Corollary~\ref{measure_of_dyn_ball} only holds on $T^1M_{\bip}$; this is where we need Corollary~\ref{ensemble separe d'une boule dynamique} and Lemma~\ref{le coeur fait bip}. Thanks to the latter, there exists $R>0$, which only depends on $\Gamma$ and $\Omega$, and such that for each $1\leq i\leq k$, we can find $w_i\in \tilde{B}_{T^1M}^{(2n)}(v_i,R)\cap T^1M_{\bip}$. Then we can use Corollary~\ref{ensemble separe d'une boule dynamique}.\ref{item:ensemble separe d'une boule dynamique2} to find $I\subset\{1,\dots,k\}$ such that $\{w_i : i\in I\}$ is $(d_{T^1M}^{(2n)},\epsilon_0/8)$-separated and 
 \begin{align*}
  k &\leq \#I\cdot \chi_+(2R+\epsilon_0/16)^2\chi_-(\epsilon_0/32)^{-2}\\
  &\leq \frac{m(\bigcup_{i\in I}B_{T^1M}^{(2n)}(w_i,\epsilon_0/16))}{\min\{m(B_{T^1M}^{(2n)}(w_i,\epsilon_0/16)):i\in I\}}\chi_+(2R+\epsilon_0/16)^2\chi_-(\epsilon_0/32)^{-2}\\
  &\leq C'e^{2n\delta_\Gamma}
 \end{align*}  
 where $C'$ only depends on $\Gamma$ and $\Omega$. Similarly, we can find $w_i'\in \phi_{-n}K_2\cap B_{T^1M}^{(2n)}(v_i,\epsilon_0/8)$ and $w_i''\in \tilde{B}_{T^1M}^{(2n)}(w_i',R)\cap T^1M_{\bip}$ for each $i=1,\dots,l$. Corollary~\ref{ensemble separe d'une boule dynamique}.\ref{item:ensemble separe d'une boule dynamique2} gives us $I'\subset\{1,\dots,l\}$ such that $\{w_i'':i\in I'\}$ is $(d_{T^1M}^{(2n)},\epsilon_0/8)$-separated and
 \begin{align*}
  l &\leq \#I'\cdot \chi_+(2R+\epsilon_0/8)^2\chi_-(\epsilon_0/32)^{-2}\\
  &\leq \frac{m(\bigcup_{i\in I'}B_{T^1M}^{(2n)}(w_i'',\epsilon_0/16))}{\min\{m(B_{T^1M}^{(2n)}(w_i'',\epsilon_0/16)):i\in I'\}}\chi_+(2R+\epsilon_0/8)^2\chi_-(\epsilon_0/32)^{-2}\\
  &\leq \epsilon C''e^{2n\delta_\Gamma}
 \end{align*}
 where $C''$ only depends on $\Gamma$ and $\Omega$, and we have used the fact that $B_{T^1M}^{(2n)}(w_i'',\epsilon_0/16)$ does not intersect $\phi_{-n}K_1$ for any $i=1,\dots,l$ and for $n$ large enough (again because of \eqref{K1 loin de K2}). As we explained in the compact case, this implies that
 \begin{align*}
  2nh_{m'}(\phi_1)  \leq (1-\epsilon)\log(\epsilon)+\log\max(C',C'')+2n\delta_\Gamma+2/e,
 \end{align*}
 which is strictly less that $2n\delta_\Gamma$ for $\epsilon$ small enough.
\end{proof}

\section{Counting closed geodesics}\label{Counting closed geodesics}

In this section we keep on adapting Knieper's article \cite{knieper_BMmeasure} in order to prove Proposition~\ref{Prop : comptage dans le cas compact}, which gives asymptotic estimates for the number of closed geodesics of length less than $t$, when $t$ goes to infinity.

These estimates do not all need the results of the previous sections. More precisely, to prove the upper bound on the number of rank-one close geodesics in \ref{Item : comptage 1} we do not need Theorem~\ref{Thm : The weak Hopf--Tsuji--Sullivan--Roblin dichotomy}. To prove the lower bound in \ref{Item : comptage 1} we need the mixing property of the Bowen--Margulis measure, but not the uniqueness of the measure of maximal entropy. To establish the upper bound on the number of non-rank-one closed geodesics in \ref{Item : comptage 2} and the equidistribution of closed geodesics we need uniqueness of the measure of maximal entropy.

Recall that $[\Gamma]$ is the set of conjugacy classes of $\Gamma$, and $[\Gamma]^{\rgun}\subset[\Gamma]$ (\resp $[\Gamma]^{\sing}$) consists of conjugacy classes of rank-one (\resp non-rank-one) elements of $\Gamma$. For each subset $A\subset [\Gamma]$  and interval $I\subset \R$, we set $A_I=\{c\in A:\ell(c)\in I\}$, and we write $A_T=A_{[0,T]}$ for any $T\geq 0$.

\subsection{The lower bound}\label{ssection:borne inf}

In this section we use the mixing of the Bowen--Margulis measure to obtain a lower bound on the number of closed geodesics.

\begin{prop}\label{minoration}
Let $\Omega\subset\PR(V)$ be a properly convex open set, and $\Gamma\subset\Aut(\Omega)$ a convex cocompact discrete subgroup with $M=\Omega/\Gamma$ rank-one and non-elementary. Then there exists a constant $C>0$ such that for any $T>C$,
 \[\#[\Gamma]^{\rgun}_T\geq \frac{C}{T}e^{\delta_\Gamma T}.\]
\end{prop}

\begin{proof}
 Without loss of generality, we may assume that $\Gamma$ is torsion-free, since for any finite-index subgroup $\Gamma'\subset\Gamma$, for any element $\gamma\in\Gamma'$, there at most $[\Gamma:\Gamma']$ conjugacy classes of $\Gamma'$ inside the conjugacy class of $\gamma$ in $\Gamma$. Let $\epsilon_0>0$ be the injectivity radius of $M$ and $m$ the Bowen--Margulis probability measure. Fix a compact subset $K\subset T^1M_{\sse}$ whose measure $m(K)$ is positive. Using Lemma~\ref{closinglemma} we can find $0<\epsilon<\epsilon_0/3$ such that for any vector $v\in K$, for any time $t\geq 1$, if $d_{T^1M}(v,\phi_t v)\leq\epsilon$ then there exists a rank-one periodic vector $w\in B_{T^1M}^{(t)}(v,\epsilon_0/6)$ with period in $[t-1,t+1]$. Let us denote by $R_t\subset K$ the subset of vectors $v\in K$ such that $d_{T^1M}(v,\phi_tv)\leq\epsilon$; we are going to bound from below its measure. To that end take $\p$ a finite measurable partition of $T^1M_{\bip}$ with diameter less than $\epsilon$, and compute the limit, using the mixing property, established in Theorem~\ref{Thm : The weak Hopf--Tsuji--Sullivan--Roblin dichotomy}:
 \[m(R_t)\geq \sum_{P\in\p}m(P\cap K\cap\phi_t(P\cap K)) \underset{t\to\infty}{\longrightarrow}\sum_{P\in\p}m(P\cap K)^2\geq \frac{m(K)^2}{\#\p}>0.\]
 On the other hand we can bound from above this measure thanks to the closing lemma and Corollary~\ref{measure_of_dyn_ball}. For any conjugacy class $c\in[\Gamma]^{\rgun}$, fix a vector $v_c\in T^1M$ tangent to the projection in $T^1M$ of the axis of any element of $c$.
 \[m(R_t)\leq \sum_{c\in[\Gamma]^{\rgun}_{t+1}}\sum_{k=0}^{\lfloor\frac{6(t+1)}{\epsilon_0}\rfloor}m(B_{T^1M}^{(t)}(\phi_{k\frac{\epsilon_0}6}v_c,\frac{\epsilon_0}3))\leq Ce^{-\delta_\Gamma t}(\frac{6(t+1)}{\epsilon_0}+1)\#[\Gamma]_{t+1}.\]
 This ends the proof.
\end{proof}

\subsection{The upper bound on the number of rank-one closed geodesics}\label{ssection:borne sup rk1}

To bound from above the number of rank-one closed geodesics we do not need Theorem~\ref{Thm : The weak Hopf--Tsuji--Sullivan--Roblin dichotomy}.

\begin{prop}\label{premiere majoration}
 Let $\Omega\subset\PR(V)$ be a properly convex open set, and $\Gamma\subset\Aut(\Omega)$ a convex cocompact discrete subgroup with $M=\Omega/\Gamma$ rank-one and non-elementary. Then there exists a constant $C>0$ such that for any $T>0$,
 \[\#[\Gamma]^{\rgun}_T\leq \frac{C}{T}e^{\delta_\Gamma T}.\]
\end{prop}

\begin{proof}
 Let $\Gamma'\subset\Gamma$ be a finite-index torsion-free subgroup; we set $M'=\Omega/\Gamma'$ and  take $\epsilon_0<1$ smaller than the injectivity radius of $M'$. For each  rank-one conjugacy class $c$ of $\Gamma$, choose $\gamma_c\in c$ and $v_c\in \axis(\gamma_c)$. Consider an integer $n\geq 1$. One easily checks using the triangular inequality that any vector of $T^1\Omega$ belongs to at most 
 $$C_1:=\max_{x\in\Omega}\#\{\gamma\in\Gamma : d_\Omega(x,\gamma x)\leq 1+2\epsilon_0/3\}$$
 balls of the family 
 $$\{B_{T^1\Omega}^{(n+1)}(\phi_{k}v_c,\epsilon_0/6) : c\in[\Gamma]^{\rgun}_{[n,n+1]}, \ 0\leq k< n\}.$$
 By Corollary~\ref{measure_of_dyn_ball}, we can bound from below the $m'$-measure of the projection in $T^1M'$ of these balls by $C^{-1}e^{-\delta_\Gamma n}$ for some constant $C>0$, where $m'$ is the Bowen--Margulis probability measure on $T^1M'$. Since $m'$ is a probability measure we obtain:
 \[n\#[\Gamma]^{\rgun}_{[n,n+1]}\leq C_1Ce^{\delta_\Gamma n \epsilon_0/3}.\qedhere\]
\end{proof}

\subsection{The upper bound on the number of non-rank-one closed geodesics}\label{ssection:borne sup rksup}

Let us bound from of above the number of non-rank-one conjugacy classes of $\Gamma$. The idea is that to each non-rank-one conjugacy class we can associate a closed geodesic which is contained in a flow-invariant closed subset of $T^1M_{\core}$ to which the Bowen--Margulis measure gives zero measure; this implies that the topological entropy of the geodesic flow on this subset is smaller than $\delta_\Gamma$.

\begin{prop}\label{l'entropie de l'adherence des orbites periodiques singulieres} 
 Let $\Omega\subset\PR(V)$ be a properly convex open set, and $\Gamma\subset\Aut(\Omega)$ a convex cocompact discrete subgroup with $M=\Omega/\Gamma$ rank-one and non-elementary. Let $K\subset T^1M_{\core}$ be the $(\phi_t)_t$-invariant closed subset that consists of the vectors whose lifts $v\in T^1\Omega$ are such that $d_{\spl}(\phi_{-\infty}v,\phi_\infty v)\leq 2$. Then the topological entropy of $(\phi_t)_t$ on $K$ is strictly smaller than $\delta_\Gamma$.
\end{prop}

\begin{proof}
 According to Remark~\ref{l'entropie est semi-continue} and Observation~\ref{obs:revetentrop}, there exists a probability measure $m'$ on $K$ whose entropy is the topological entropy on $K$. Observe that $K$ is disjoint from the set $T^1M_{\sse}$ of vectors whose lifts $v\in T^1\Omega$ satisfy $\phi_{-\infty}v,\phi_\infty v\in\partial_{\sse}\Omega$, while the Bowen--Margulis probability measure $m_\Gamma$ is concentrated on $T^1M_{\sse}$ by Theorem~\ref{Thm : The weak Hopf--Tsuji--Sullivan--Roblin dichotomy}. Thus, $m'$ and $m_\Gamma$ are different. By Theorem~\ref{Thm : proba d'entropie max dans le cas compact}, the entropy of $m'$ must be strictly smaller than $\delta_\Gamma$.
\end{proof}

\begin{cor}\label{seconde majoration}
 Let $\Omega\subset\PR(V)$ be a properly convex open set, and $\Gamma\subset\Aut(\Omega)$ a convex cocompact discrete subgroup with $M=\Omega/\Gamma$ rank-one and non-elementary. Then the exponential growth rate of the number of non-rank-one conjugacy classes of translation length shorter than $t$, when $t$ grows, is strictly less than $\delta_\Gamma$.
\end{cor}

\begin{proof}
 Let $\Gamma'\subset\Gamma$ be a finite-index torsion free subgroup; we set $M'=\Omega/\Gamma'$ and  $\epsilon_0$ to be the injectivity radius of $M'$;  we denote by $\epsilon_0$ the injectivity radius of $M'$. It is easy to show that there are only finitely many conjugacy class $c$ of $\Gamma$ such that $\ell(c)\leq \epsilon_0$. Using Corollary~\ref{Corollaire : real d'ell dans core}, for each conjugacy class $c$ with $\ell(c)>0$ we can find  an element $\gamma_c\in c$ and a vector $v_{c}\in T^1\Omega$ such that $\phi_{\pm\infty}v_c\in\Lambda^{\orb}$ and $\gamma_c\pi v_c=\pi\phi_{\ell(c)}v_c$. Consider $t>0$. Using the triangular inequality, one can check that any vector of $T^1\Omega$ belongs to at most $$C_1:=\max_{x\in\Ccal_\Omega^{\core}(\Gamma)}\#\{\gamma\in\Gamma : d_\Omega(x,\gamma x)\leq 1+2\epsilon_0/3\}$$ balls of the family
 $$\{B_{T^1\Omega}^{(t+1)}(v_c,\epsilon_0/6) : c\in[\Gamma]^{\sing}_{[t, t+1]}\}.$$
 Therefore we can extract from $\{v_c : c\in[\Gamma]^{\sing}_{[t, t+1]}\}$ a $(d_{T^1\Omega}^{(t+1)},\epsilon_0/6)$-separated family of size at least $C_1^{-1}\#[\Gamma]^{\sing}_{[t, t+1]}$. The projection in $T^1M'$ of this family belongs to the set $K$ in Proposition~\ref{l'entropie de l'adherence des orbites periodiques singulieres} by Lemma~\ref{lem:realell}.\ref{item:realellpasr1}. By definition of the topological entropy, the exponential growth rate of the size of such a family, when $t$ goes to infinity, is bounded above by the topological entropy on $K$, which is strictly less than $\delta_\Gamma$ by Proposition~\ref{l'entropie de l'adherence des orbites periodiques singulieres} above.
\end{proof}

\subsection{Sums of uniform measures on closed geodesics}\label{ssection:equidistribution}

Let $\Omega\subset\PR(V)$ be a properly convex open set and $\Gamma\subset\Aut(\Omega)$ a discrete subgroup; denote $M=\Omega/\Gamma$. We introduce a few notations. For any conjugacy class $c\in[\Gamma]^{\rgun}$, we denote by $\Lcal c$ the unique $(\phi_t)_{t\in\R}$-invariant probability measure supported on the projection in $T^1M$ of the axis of any element of $c$. For each finite subset $A\subset[\Gamma]^{\rgun}$, we consider
\[\Lcal A=\frac{1}{\#A}\sum_{c\in A}\Lcal c.\]

\begin{prop}\label{Proposition : equidistribution}
  Let $\Omega\subset\PR(V)$ be a properly convex open set, and $\Gamma\subset\Aut(\Omega)$ a convex cocompact discrete subgroup with $M=\Omega/\Gamma$ rank-one and non-elementary. Let $A\subset[\Gamma]^{\rgun}$ be such that $\log(\# A_T)/T$ converges to $\delta_\Gamma$ when $T$ tends to infinity. Then $\Lcal A_T$ converges to the Bowen--Margulis probability measure when $T$ goes to infinity.
\end{prop}

\begin{proof}
 Let $\Gamma'\subset\Gamma$ be a torsion-free finite-index subgroup; let $\epsilon_0$ be the injectivity radius of $M'=\Omega/\Gamma'$. For each conjugacy class $c\in[\Gamma]^{\rgun}$, we choose a representative $\gamma_c\in\Gamma$, and we call $\Lcal' c$ the unique $(\phi_t)_t$-invariant probability measure on the projection in $T^1M'$ of the axis of $\gamma_c$. For any finite subset $B\subset [\Gamma]^{\rgun}$, we set $\Lcal' B=(\#B)^{-1}\sum_{c\in B}\Lcal'c$. By theorem~\ref{Thm : proba d'entropie max dans le cas compact}, it is enough to show that any accumulation point $m'\lim_{k\to\infty}\Lcal'A_T$ on $T^1M'$ has entropy bounded below by $\delta_\Gamma$.
 
 Let us give ourselves a finite measurable partition $\p$ of $T^1M'$ of diameter less than $\epsilon_0/3$ and such that for any element $P\in\p$, we have $m'(\partial P)=0$. Then 
 \[h_{m'}(\phi) \geq H_{m'}(\phi,\p) = \lim_{n\to\infty}\frac{1}{n}H_{m'}(\p^{(n)}).\]
 Fix $n\geq 1$ and note that for each $P\in\p^{(n)}$, we have $m'(\partial P)=0$. As a consequence 
 \[H_{m'}(\p^{(n)})=\lim_{k\to\infty}H_{\Lcal'A_{T_k}}(\p^{(n)})\geq \liminf_{T\to\infty}H_{\Lcal'A_T}(\p^{(n)}).\]
 Consider $\alpha>0$ and let us show that $\liminf_{T\to\infty}H_{\Lcal'A_T}(\p^{(n)})\geq n(\delta_\Gamma - \alpha)$. Let $T_0>0$ be large enough so that $\#A_T\geq e^{(\delta_\Gamma-\alpha)T}$ for any $T\geq T_0$.
 Take $T> T_0$ and decompose $[0,T]$ into disjoint intervals : 
 \[[0,T]=[0,T_0]\sqcup\bigsqcup_{I\in\mathcal{I}_T}I,\]
 such that each $I\in\mathcal{I}_T$ has diameter less than $1$, and $\#\mathcal{I}_T=\lceil T-T_0\rceil$. 
 Then by \cite[Prop.\,4.3.3.6]{katok} we have 
 \begin{align*}
 \liminf_{T\to\infty}H_{\Lcal'A_T}(\p^{(n)})&\geq \liminf_{T\to\infty} \frac{\#A_{T_0}}{\#A_T}H_{\Lcal'A_{T_0}}(\p^{(n)})+\sum_{I\in\mathcal{I}_T}\frac{\#A_I}{\#A_T}H_{\Lcal'A_I}(\p^{(n)})\\
 & \geq \liminf_{T\to\infty}\sum_{I\in\mathcal{I}_T}\frac{\#A_I}{\#A_{]T_0,T]}}H_{\Lcal'A_I}(\p^{(n)})\\
 & \geq \liminf_{T\to\infty}\frac{n}{T}\sum_{I\in\mathcal{I}_T}\frac{\#A_I}{\#A_{]T_0,T]}}H_{\Lcal'A_I}(\p^{(\lceil T\rceil)}),
 \end{align*}
 where we have used the Euclidean division $\lceil T\rceil=q_Tn+r_T$ with the following classical inequality (see \eg \cite[Prop.\,4.3.3.1-4]{katok}): 
 \[H_{\lambda_I}(\p^{(n)}) \geq \frac{1}{q_T}H_{\lambda_I}(\p^{(\lceil T\rceil)}) - \frac{1}{q_T}H_{\lambda_I}(\p^{(r_T)})
  \geq \frac{n}{T+1}H_{\lambda_I}(\p^{(\lceil T\rceil)}) - \frac{n}{T-n}\log(\#\p^{n}).\]
 
 Using the triangular inequality, one checks that for any $P\in\p^{(\lceil T\rceil)}$ and any $I\in\mathcal{I}_T$, there are at most $C_1=\max_{x\in\Ccal^{\core}_\Omega(\Gamma)}\#\{\gamma\in\Gamma:d_\Omega(x,\gamma x)\leq 1+2\epsilon_0/3\}$ conjugacy classes $c\in A_I$ such that $\Lcal'c(P)>0$; this implies that $\Lcal'A_I(P)\leq C_1 \#A_I^{-1}$. Hence $H_{\Lcal'A_I}(\p^{(\lceil T\rceil)})\geq \log(\#A_I)-\log(C_1)$, and we resume our computation, using the concavity of the logarithm and Cauchy--Schwarz inequality:
 \begin{align*}
 \frac{n}{T}\sum_{I\in\mathcal{I}_T}\frac{\#A_I}{\#A_{]T_0,T]}}H_{\Lcal'A_I}(\p^{(\lceil T\rceil)})& \geq \frac{n}{T}\sum_{I\in\mathcal{I}_T}\frac{\#A_I}{\#A_{]T_0,T]}}\log(\#A_I)-\frac nT \log(C_1)\\
 &\geq \frac{n}{T} \log\left(\frac{1}{\#A_{]T_0,T]}}\sum_{I\in\mathcal{I}_T}\#A_I^2\right)-\frac nT \log(C_1)\\
 &\geq \frac{n}{T} \log\left(\frac{\#A_{]T_0,T]}}{\#\mathcal{I}_T}\right)-\frac nT \log(C_1)\\
 &\geq \frac{n}{T} \log\left(\frac{\#A_T-\#A_{T_0}}{\lceil T-T_0\rceil}\right)-\frac nT \log(C_1)\\
 &\geq \frac{n}{T} \log\left(\frac{e^{(\delta_\Gamma-\alpha)T}-\#A_{T_0}}{\lceil T-T_0\rceil}\right)-\frac nT \log(C_1).
 \end{align*}
 This last term converges to $n(\delta_\Gamma-\alpha)$ as $T$ goes to infinity. This concludes the proof.
\end{proof}

We can now end the proof of Proposition~\ref{Prop : comptage dans le cas compact}.

\begin{proof}[Proof of Proposition~\ref{Prop : comptage dans le cas compact}]
 It is the immediate combination of Facts~\ref{smear} and \ref{entropie<d-2}, Theorem~\ref{Thm : proba d'entropie max dans le cas compact}, Propositions~\ref{l'entropie de BM est max}, \ref{minoration}, \ref{premiere majoration}, \ref{Proposition : equidistribution}, and Corollary~\ref{seconde majoration}.
\end{proof}

\subsection{Periodic geodesics and conjugacy classes}\label{sec:perconj}

In this section, we estimate, in some cases, the asymptotic of the ratio of the number of rank-one closed geodesics by the number of rank-one conjugacy classes. The same discussion is carried in more details in \cite[\S6.1]{EeERfH+};  here the discussion is kept at its strict minimum: we only give the necessary definitions and observations.

\begin{defi}\label{defn:corfix}
 Let $M=\Omega/\Gamma$ be a non-elementary rank-one convex projective orbifold. The \emph{core-fixing subgroup} of $\Gamma$ is the kernel of the restriction of $\Gamma$ to the span of $\Lambda_\Gamma$. 
\end{defi}

\begin{obs}[{\cite[Obs.\,6.14]{EeERfH+}}]\label{obs:corfix}
 Let $M=\Omega/\Gamma$ be a rank-one non-elementary convex projective orbifold, and let $\corfix\subset\Gamma$ be the core-fixing subgroup. 
 If $T^1\Omega_{\bip}$ is the preimage by $\pi_\Gamma:T^1\Omega\rightarrow T^1M$ of $T^1M_{\bip}$, then the set of vectors $v\in T^1\Omega_{\bip}$ with $\Stab_\Gamma(v)=F$ is open and dense in $T^1\Omega_{\bip}$, and $\Gamma$-invariant.
\end{obs}

\begin{prop}\label{prop:nbconj}
 Let $\Omega\subset\PR(V)$ be a properly convex open set, and $\Gamma\subset\Aut(\Omega)$ a convex cocompact discrete subgroup with $M=\Omega/\Gamma$ rank-one and non-elementary. Let $\corfix\subset\Gamma$ be the core-fixing subgroup. Let $K\subset T^1M_{\bip}$ be the set vectors of whose lifts $v\in T^1\Omega$ satisfy $\Stab_\Gamma(v)\neq \corfix$. Then the topological entropy of the geodesic flow on $K$ is strictly smaller than $\delta_\Gamma$, as well as the exponential growth rate of the number of rank-one conjugacy classes of $\Gamma$ with axis in $K$.
\end{prop}
\begin{proof}
 The subset $K\subset T^1M_{\bip}$ is closed and $(\phi_t)_t$-invariant, and has empty-interior by Observation~\ref{obs:corfix}; thus, it is given zero measure by the Bozen--Margulis measure, by Theorem~\ref{Thm : The weak Hopf--Tsuji--Sullivan--Roblin dichotomy}. By Remark~\ref{l'entropie est semi-continue} and Observation~\ref{obs:revetentrop}, we can find a probability measure on $K$ whose entropy $\delta$ is the topological entropy of $(\phi_t)_t$; this measure is different from the Bowen--Margulis probability measure since the latter gives zero measure to $K$. Hence $\delta<\delta_\Gamma$ by Theorem~\ref{Thm : proba d'entropie max dans le cas compact}.
 
 Let $\Gamma'\subset\Gamma$ be a finite-index torsion free subgroup; we set $M'=\Omega/\Gamma'$ and  $\epsilon_0$ to be the injectivity radius of $M'$;  we denote by $\epsilon_0$ the injectivity radius of $M'$. Recall that there are only finitely many conjugacy class $c$ of $\Gamma$ such that $\ell(c)\leq \epsilon_0$. For each rank-one conjugacy class $c\in[\Gamma]^{\rgun}$ with $\ell(c)>0$ we can find  an element $\gamma_c\in c$ and a vector $v_{c}\in T^1\Omega_{\bip}:=\pi_\Gamma^{-1}T^1M_{\bip}$ such that $\gamma_c\pi v_c=\pi\phi_{\ell(c)}v_c$. Let $A\subset[\Gamma]^{\rgun}$ be the subset made of conjugacy classes $c$ such that $\Stab_\Gamma(v_c)\neq F$. Consider $t>0$. Using the triangular inequality, one can check that any vector of $T^1\Omega$ belongs to at most $$C_1:=\max_{x\in\Ccal_\Omega^{\core}(\Gamma)}\#\{\gamma\in\Gamma : d_\Omega(x,\gamma x)\leq 1+2\epsilon_0/3\}$$ balls of the family
 $$\{B_{T^1\Omega}^{(t+1)}(v_c,\epsilon_0/6) : c\in A_{[t, t+1]}\}.$$
 Therefore we can extract from $\{v_c : c\in A_{[t, t+1]}\}$ a $(d_{T^1\Omega}^{(t+1)},\epsilon_0/6)$-separated family of size at least $C_1^{-1}\#A_{[t, t+1]}$. The projection in $T^1M'$ of this family is $(d_{T^1M'}^{(t+1)},\epsilon_0/6)$-separated, and belongs to the preimage by $T^1M'\rightarrow T^1M$ of $K$. By definition of the topological entropy, the exponential growth rate of the size of such a family, when $t$ goes to infinity, is bounded above by $\delta$, which is strictly less than $\delta_\Gamma$.
\end{proof}

\begin{rqq}
Let $M=\Omega/\Gamma$ be a non-elementary rank-one convex projective orbifold with compact convex core.
As in \cite[Obs.\,6.15 \& Prop.\,6.16]{EeERfH+}, one may use the notion of strongly primitive rank-one elements (\ie rank-one elements $\gamma\in\Gamma$ such that $\ell(\gamma)\leq\ell(\gamma')$ for any rank-one element $\gamma'\in\Gamma$ with the same axis as $\gamma$) to prove that, if the core-fixing subgroup $\corfix\subset\Gamma$ is the centraliser, then
\begin{equation*}\label{eq:ngconj}\#[\Gamma]_T^{\rgun}\underset{T\to\infty}{\sim}\#[\Gamma]_T^{\prgun}\underset{T\to\infty}{\sim}\#F \cdot \#\mathcal G^{\rgun}_T,\end{equation*}
where $[\Gamma]^{\prgun}$ is the set of strongly primitive rank-one conjugacy classes in $\Gamma$ and $\mathcal G^{\rgun}_T$ is the set of rank-one periodic orbits in $T^1M$ with period at most $T$.

The proof in \cite{EeERfH+} uses equidistribution results which may be replaced for the present purpose by Proposition~\ref{Prop : comptage dans le cas compact} (one would also need Proposition~\ref{prop:nbconj}). This way, one actually obtains that $\frac{\#[\Gamma]_T^{\prgun}}{\#\mathcal G^{\rgun}_T}$ goes \emph{exponentially fast} to $\#F$, while in \cite{EeERfH+} no rate of convergence is given.
\end{rqq}

\bibliographystyle{alpha-mod}
{\small \bibliography{bib}}

\Addresses

\end{document}